\documentclass[12pt]{amsart}
\usepackage{calligra}
\usepackage{amsthm}
\usepackage{amsmath}
\usepackage{amssymb}
\usepackage{amscd}
\usepackage{hyperref}
\usepackage{bbm}
\usepackage[all]{xy}
\usepackage{mathrsfs}
\usepackage[top=72pt,bottom=96pt,left=72pt,right=72pt]{geometry}
\newtheorem{Lemma}{Lemma}[section]
\newtheorem{Remark}[Lemma]{Remark}

\newtheorem{Theorem}[Lemma]{Theorem}
\newtheorem{Corollary}[Lemma]{Corollary}
\newtheorem{Definition}[Lemma]{Definition}
\newtheorem{Proposition}[Lemma]{Proposition}

\def\ad{\mathrm{ad}}
\def\Ad{\mathrm{Ad}}

\def\ad{\mathrm{ad}}

\def\C{\mathbb{C}}

\def\F{\mathfrak{F}}
\def\f{\mathfrak{f}}
\def\qGG{\mathfrak{qGG}}

\def\Hor{\mathrm{Hor}}

\def\dvol{\mathrm{dvol}}

\def\id{\mathrm{id}}

\def\Ker{\mathrm{Ker}}
\def\Im{\mathrm{Im}}

\def\H{\mathrm{H}}
\def\V{\mathrm{V}}
\def\Mor{\textsc{Mor}}
\def\N{\mathbb{N}}

\def\R{\mathbb{R}}
\def\Z{\mathbb{Z}}

\def\l{\mathrm{L}}
\def\r{\mathrm{R}}

\def\c{\mathrm{c}}

\def\r{\mathrm{R}}

\def\U{\mathcal{U}}
\def\G{\mathcal{G}}

\def\T{\mathcal{T}}

\def\H{\mathbb{H}}

\begin{document}
\date{\today}
\title{Two Non--Commutative $U(1)$--Gauge Laplacians in the Quantum Hyperboloid}
\author{Gustavo Amilcar Salda\~na Moncada}

\begin{abstract}
In this paper, we will characterize the spectrum of two $U(1)$--gauge Laplacians, which do not commute with each other, on the upper sheet of a two--sheet quantum hyperboloid equipped with a concrete differential calculus; while we provide a thorough and comprehensive study of the geometric structures on the related quantum spaces.
 \begin{center}
  \parbox{300pt}{\textit{MSC 2020:}\ 46L87, 53C05.}
  \\[5pt]
  \parbox{300pt}{\textit{Keywords:}\ Laplacian, quantum hyperbolid,  quantum principal connection}
 \end{center}
\end{abstract}
\maketitle

\section{Introduction}

The Laplacian is one of the most important differential operators in mathematical physics. In differential geometry, a detailed study of the Laplacian on the hyperboloid can be found in  \cite{otto}. Of course, to study gauge Laplacians, we must first define a principal bundle structure over the hyperboloid. For instance, since the upper sheet of a two–sheeted hyperboloid can be viewed as a homogeneous space, it naturally carries the structure of a principal bundle.

In this way, in the present paper we will use the non--commutative geometrical version  $\zeta_H$ of the above principal bundle to study the Laplacian and the gauge Laplacians on the {\it upper sheet of a two--sheeted quantum hyperboloid}. As a consequence, we will show the existence of two gauge Laplacians that do not commute with each other; thus, they may be regarded as two distinct operators. However, it is worth mentioning that the two operators share some eigenvectors.

Although scarce, there are papers in non--commutative geometry that also work with two gauge Laplacians, for example, \cite{lz,saldhopf}, where the reader can find results similar to those presented here in the context of the quantum Hopf fibration. Nevertheless, the importance and the novelty of this paper rests on the following three facts:
\begin{enumerate}
    \item The procedure leading to the emergence of two non--commuting gauge Laplacians is based on the general theory developed in \cite{saldym,sald1} and is therefore not specific to the quantum principal bundle $\zeta_H$ considered in this paper. This, together with the studies by Landi and Zampini and other authors on the quantum Hopf fibration (see, for example, \cite{lz}), suggests that the existence of non--commuting gauge Laplacians may reflect a more general phenomenon in non--commutative geometry, rather than being merely a \emph{peculiarity} of $\zeta_H$ or of the quantum Hopf fibration. Further research is needed.
    
    \item In non--commutative geometry, it is common to consider only one of the two gauge Laplacians while disregarding the other one. In light of point 1, the results of this paper suggest that such an approach may be incomplete and that both gauge Laplacians should always be considered and studied.
    
    This is of great importance because, in mathematical physics, the Laplacian associated with a space carries physically relevant information. For example, in quantum mechanics, the Laplacian is related to the kinetic energy of matter. This raises the following question: \emph{If a quantum space can describe the physics of a system (which is one of the purposes of non--commutative geometry, as the reader can see in \cite{con}), what is the physical significance of having two non--commuting gauge Laplacians?} Perhaps, if a quantum space can be associated with a physical system, the commutator of the two gauge Laplacians vanishes on the space of physically admissible fields, even if the operators do not commute in general. The need to investigate further is evident.
    
    \item  The study of quantum hyperbolic spaces remains an active area of research, as the reader can check in \cite{landichi}. However, most studies focus on their $C^\ast$--algebraic structures or on their quantum homogeneous space structures. In this way, this paper presents one of the first full and comprehensive studies in the literature on the differential geometry of quantum principal bundles over quantum hyperbolic spaces. 
    
    In other words, this paper is one of the first in the literature that studies a quantum principal bundle on a quantum hyperbolic space, its differential calculus, quantum principal connections, covariant derivatives, curvatures, associated quantum vector bundles, induced quantum linear connections, the quantum gauge group of the bundle and its action on connections. 
    
    It is worth mentioning that these structures are investigated at a level of detail comparable to that achieved in the extensive literature on the quantum Hopf fibration; and with this level of detail, $\zeta_H$ provides the research community with a useful and well--understood quantum principal bundle for exploring and testing new theories on it, much as the quantum Hopf fibration has served as a fundamental example over the years (see, for example, \cite{tomas}). However, $\zeta_H$ exhibits different topological/geometrical properties. For instance, while the total space of the quantum Hopf fibration is a compact  quantum group, the total space of $\zeta_H$ is non--compact \cite{fur}.
\end{enumerate}

For the purposes of this work, we will use Durdevich’s formulation of quantum principal bundles, since this framework naturally allows one to obtain both gauge Laplacians. The reader is encouraged to consult  references \cite{saldym,sald1,micho1,micho2,micho3,stheve} for further details on Durdevich’s formulation.

Following this introduction, in Section 2 we present the quantum principal bundle $\zeta_H$ on which we will work. In this section, we also define a differential structure on $\zeta_H$. In particular, we identify the space of differential forms on the upper sheet of a two--sheet quantum hyperboloid.

Section 3 is divided into three parts. First, we study quantum principal connections, covariant derivatives, and curvatures on  $\zeta_H$. Second, we characterize all associated quantum vector bundles, as well as the induced quantum linear connections. Finally, in the last subsection we define a Hodge operator for our differential structure. With this, we introduce the quantum gauge Laplacians in complete analogy with the definition of the gauge Laplacian in differential geometry.

In Section 4, we compute the spectrum of all our Laplacians. Using these results, in Theorem \ref{gauge0} we prove the non--commutativity of the Laplacians. Finally, in Appendix A we include a brief summary of the theory of the universal differential envelope  $\ast$--calculus. We use this universal differential envelope $\ast$--calculus to define differential forms of $\zeta_H$.

Of course, throughout the entire paper we will use Sweedler notation. As always, for the coproduct of a Hopf algebra, we will write 
$a^{(1)}\otimes a^{(2)}$, while for any other right corepresentation, we will use the notation $a^{(0)}\otimes a^{(1)}$.

\section{A Quantum Principal Bundle over the Quantum Hyperboloid and its Differential Calculus}

In differential geometry, the Lie group $SU(1,1)$ is defined as $$ SU(1,1):=\left\{ \begin{pmatrix}
\alpha & \gamma^\ast \\
\gamma & \alpha^\ast 
\end{pmatrix} \in M_2(\C)\;\; \left|\right.\;\; |\alpha|^2- |\beta|^2=1 \right\}.$$ Furthermore, the Lie group $$U(1)=\{z \in \C \mid |z|=1 \}$$ can be viewed as a Lie subgroup of $SU(1,1)$ by means of $$\iota: U(1)\longrightarrow SU(1,1),\qquad z\longmapsto \begin{pmatrix}
z & 0 \\
0 & z^\ast 
\end{pmatrix} $$ and the homogeneous space $$\H^2={SU(1,1)\over U(1)}\cong {SL(2,\R)\over U(1)} $$ is the upper sheet of a two--sheet hyperboloid. In this way, there is a {\it natural} homogeneous principal $U(1)$--bundle structure over $\H^2$: $$U(1)\hookrightarrow SU(1,1)\cong SL(2,\R)\rightarrow \H^2.$$

In this section, which is based on  \cite{micho1,micho2,micho3,stheve,perla,woro2,woro1}, we are going to present the {\it non--commutative geometrical version} of the previous bundle, as well as a differential calculus on it.

\subsection{The Quantum Principal $\U(1)$--Bundle over the Quantum Hyperboloid.}

Consider the $\ast$--Hopf algebra
\begin{equation}
\label{ec.2.0}
(P:=SU_q(1,1),\cdot, \mathbbm{1},\ast,\Delta,\epsilon,S)
\end{equation}
with $q$ $\in$ $(-1,1)-\{0\}$ (the quantum $SU(1,1)$ group \cite{fur}). The space $P$ is the $\ast$--algebra generated by two letters $\{\alpha,\gamma\}$ which satisfy
\begin{equation}
\label{ec.2.1}
\alpha^{\ast}\alpha-\gamma^{\ast}\gamma=\mathbbm{1},\quad \alpha\alpha^{\ast}-q^{2}\gamma\gamma^{\ast}=\mathbbm{1},\quad
\gamma\gamma^{\ast}=\gamma^{\ast}\gamma,\quad q\gamma\alpha=\alpha\gamma, \quad q\gamma^{\ast}\alpha=\alpha\gamma^{\ast},
\end{equation}
and  
\begin{equation}
\label{ec.2.2}
\Delta(\alpha)=\alpha\otimes\alpha+q\gamma^\ast\otimes\gamma, \qquad \Delta(\gamma)=\gamma\otimes\alpha+\alpha^\ast\otimes\gamma,
\end{equation}
\begin{equation}
\label{ec.2.3}
\epsilon(\alpha)=1, \qquad \epsilon(\gamma)=0,
\end{equation}
\begin{equation}
\label{ec.2.4}
S(\alpha)=\alpha^\ast, \quad S(\alpha^\ast)=\alpha,\quad S(\gamma)=-q\gamma,\quad S(\gamma^\ast)=-q^{-1}\gamma^\ast.
\end{equation}
It is worth mentioning that 
\begin{equation}
    \label{ec.2.5}
    \beta_P:=\{\alpha^m\gamma^k\gamma^{\ast l}\mid m \in \Z,\;\; k, l\in \N_0 \}
\end{equation}
is a linear basis of $P$. Here, there is an abuse of notation: for all $m$ $\in$ $\N$, we have considered that $\alpha^{-m}:=\alpha^{\ast m}$ and $\gamma^0=\alpha^0=\mathbbm{1}$. However, by equation (\ref{ec.2.1}), it is clear that $\alpha^\ast$ is not the multiplicative inverse of $\alpha$. 

Now, let us consider the canonical compact quantum group $\U(1)$ associated with the Lie group $U(1)$. Its dense $\ast$--Hopf algebra is given by 
\begin{equation}
\label{ec.2.65}
G:=\C[z,z^\ast] \quad \mbox{ with }\quad z^\ast=z^{-1},
\end{equation}
and 
\begin{equation}
\label{ec.2.7}
\Delta'(z)=z\otimes z, \qquad \epsilon'(z)=1, \qquad S'(z)=z^\ast, \qquad S'(z^\ast)=z.
\end{equation}

Define the linear map
\begin{equation}
\label{ec.2.8}
j:P \longrightarrow G
\end{equation}
such that $$j(\alpha)=z,\qquad j(\gamma)=0.$$ A routine calculation shows that $j$ is, in fact, a $\ast$--Hopf algebra epimorphism.

\begin{Definition}
    \label{2.1}
    In accordance with references \cite{micho2,stheve,perla}, by considering the $\ast$--algebra morphism
\begin{equation}
\label{ec.2.9}
\Delta_{P}:=(\id_{P}\otimes j)\circ \Delta:P\longrightarrow P\otimes G,
\end{equation}
the triple
\begin{equation}
\label{ec.2.10}
\zeta_{H}=(P,\H^2_q,\Delta_{P})
\end{equation}
is a homogeneous quantum principal $\U(1)$--bundle, where 
\begin{equation}
\label{ec.2.11}
\H^2_q:=\{b\in P\mid \Delta_P(b)=b\otimes \mathbbm{1} \}.
\end{equation}
\end{Definition}

In particular, we have
\begin{equation}
\label{ec.2.12}
\Delta_P(\alpha)=\alpha\otimes z,  \qquad  \Delta_P(\alpha^\ast)=\alpha^\ast\otimes z^\ast,\qquad \Delta_P(\gamma)=\gamma\otimes z, \qquad \Delta_P(\gamma^\ast)=\gamma^\ast\otimes z^\ast.
\end{equation}

It is easy to see that $\H^2_q$ is generated  as a $\ast$--algebra by (\cite{perla})
\begin{equation}
\label{ec.2.13}
\{\rho:=\gamma\gamma^\ast,\qquad \xi:=\alpha\gamma^\ast \}
\end{equation}
and it is straightforward to verify that the following relations hold
\begin{equation}
\label{ec.2.14}
(q^2\rho+{1\over 2}\mathbbm{1})^2-\xi\,\xi^\ast=(\rho+{1\over 2}\mathbbm{1})^2-\xi^\ast\,\xi={1\over 4}\mathbbm{1}.
\end{equation}
Since $\rho=\gamma\gamma^\ast$, we can consider $\rho$ as a real positive coordinate. In the same way, we can consider $\xi$ as a complex coordinate. Hence, taking the limit $q\longrightarrow 1$ (the {\it classical} limit), equation (\ref{ec.2.14}) is exactly the equation of the upper sheet $\H^2$ of a two--sheet hyperboloid. Therefore, we will refer to $\H^2_q$ as the {\it upper sheet of a quantum two--sheet hyperboloid or simply, as the quantum hyperboloid}. In addition, by construction, in the limit $q\longrightarrow 1$ (the {\it classical} limit), the quantum principal bundle $\zeta_H$ is the algebraical--dual of the  principal bundle $$U(1)\hookrightarrow SU(1,1)\rightarrow \H^2.$$ Nevertheless, in the limit $q\longrightarrow -1$, equation (\ref{ec.2.14}) is also the equation of the upper sheet $\H^2$ of a two--sheet hyperboloid, but, the total space remains being non--commutative.

It is worth mentioning that
\begin{equation}
\label{ec.2.15}
\beta_B:=\{\alpha^m\gamma^k\gamma^{\ast l}\mid m+k-l=0 \}
\end{equation}
is a linear basis of $\H^2_q$.

\subsection{A Differential Calculus for $\zeta_H$}

Let us take the left covariant $\ast$--First Order Differential Calculus (abbreviated as $\ast$--FODC)
\begin{equation}
\label{ec.2.16}
(\Omega^1(P),d)
\end{equation}
on $P$ given by the following right $P$--ideal of $\Ker(\epsilon)$  (\cite{stheve,woro2})
\begin{equation}
\label{ec.2.17}
\mathcal{R}_3:=\langle\{ \gamma^2,\,\gamma^{\ast\,2},\,\gamma\gamma^\ast,\,\alpha\gamma-\gamma,\, \alpha\gamma^\ast-\gamma^\ast,\alpha^\ast\gamma-\gamma^,\,\alpha^\ast\gamma^\ast-\gamma^\ast,\,q^2\alpha+\alpha^\ast-(1+q^2)\mathbbm{1} \}\rangle,
\end{equation}
where $\langle x\rangle$ denotes the ideal generated by $x$. We will refer to the $\ast$--FODC $(\Omega^1(P),d)$ as the {\it $3D$ differential calculus of} $P$. It receives this name because, as in the case of the Woronowicz $3D$ differential calculus of $SU_q(2)$ (and the proof is completely analogous \cite{woro2}), the $\C$--vector space $$\mathfrak{su_q}(1,1)^\#:=\dfrac{\Ker(\epsilon)}{\mathcal{R}_3}$$ has dimension $3$ and the set
\begin{equation}
\label{ec.2.18}
\beta:=\{\eta_3:=\pi(\alpha-\alpha^\ast),\quad \eta_+:=\pi(\gamma),\quad \eta_-:=q\,\pi(\gamma^\ast) \}
\end{equation}
is a linear basis of $\mathfrak{su}_q(1,1)^\#$, where $\pi$ is the corresponding quantum germs map (Section 6.4 of  \cite{stheve}). The quantum germs map
\begin{equation}
\label{ec.2.19}
\pi: P\longrightarrow \mathfrak{su_q}(1,1)^\#
\end{equation}
is defined by $$\pi(p)=S(p^{(1)})dp^{(2)},$$ with $\Delta(p)=p^{(1)}\otimes p^{(2)}$. This map has several useful properties, for example, $\pi|_{\Ker(\epsilon)}$ is surjective and the following formulas hold (see Sections 6.4 and 11.2 of \cite{stheve})
\begin{equation}
    \label{ec.2.19.1}
    dp=p^{(1)}\pi(p^{(2)}), \qquad \pi(p)^{\ast}=-\pi(S(p)^\ast).
\end{equation}
The space $\mathfrak{su}_q(1,1)^\#$ has a canonical structure of right $P$--module given by (see Section 6.4 of \cite{stheve})
\begin{equation}
    \label{ec.2.19.2}
    \pi(p_1)\odot p_2:=\pi(p_1p_2-\epsilon(p_1)p_2)=S(p_2^{(1)})\,\pi(p_1)\,p_2^{(2)}
\end{equation}
for $p_1$, $p_2$ $\in$ $P$.  A simple but tedious direct calculation using equations  (\ref{ec.2.1})--(\ref{ec.2.4}) and (\ref{ec.2.19.1}), (\ref{ec.2.19.2}), proves the next proposition.
\begin{Proposition}
\label{algo1}
    The following relations hold
     \begin{equation*}
       \pi(\alpha^\ast)=-q^2 \pi(\alpha), \qquad \eta_3=(1+q^2)\pi(\alpha),
     \end{equation*}
    \begin{equation*}
        \pi(\gamma^2)=0,\quad \pi(\gamma^{\ast 2})=0, \quad \pi(\gamma\gamma^\ast)=0, \quad \pi(\alpha\gamma)=\pi(\gamma), \quad \pi(\alpha\gamma^\ast)=\pi(\gamma^\ast)
            \end{equation*}
            \begin{equation*}
        \pi(\alpha^\ast\gamma)=\pi(\gamma),\qquad \pi(\alpha^\ast\gamma^\ast)=\pi(\gamma^\ast), \qquad q^2\pi(\alpha^2)=(1+q^2)\pi(\alpha),
            \end{equation*}
            \begin{equation*}
        \eta_3\,\alpha= q^{-2}\alpha\,\eta_3, \quad \eta_3\,\alpha^\ast= q^{2}\alpha^\ast\,\eta_3,\quad \eta_3\,\gamma=q^{-2}\gamma\,\eta_3, \quad \eta_3\,\gamma^\ast=q^{2}\gamma^\ast\,\eta_3,
        \end{equation*}
        \begin{equation*}
        \eta_\pm \,\alpha=q^{-1}\alpha\,\eta_\pm, \quad \,\eta_\pm\,\alpha^\ast=q\alpha^\ast\,\eta_\pm, \quad \eta_\pm \,\gamma=q^{-1}\gamma\, \eta_\pm, \quad \eta_\pm \,\gamma^\ast=q\gamma^\ast\, \eta_\pm,
        \end{equation*}
        \begin{equation*}
        \eta^\ast_3=-\eta_3,\qquad \qquad \qquad \eta^\ast_-=\eta_+, \qquad \qquad \qquad \eta^\ast_+=\eta_-,
        \end{equation*}
        \begin{equation*}
        d\alpha={1\over 1+q^{2}}\alpha\,\eta_3+q\,\gamma^\ast\,\eta_+,\qquad  d\alpha^\ast=-{q^2\over 1+q^{2}}\alpha^{\ast}\,\eta_3+\gamma\,\eta_-,
        \end{equation*}
        \begin{equation*}
        d\gamma={1\over 1+q^{2}}\gamma\,\eta_3+\alpha^\ast\,\eta_+,\qquad  d\gamma^\ast=-{q^2\over 1+q^{2}}\gamma^{\ast}\,\eta_3+q^{-1}\alpha\,\eta_-.  
    \end{equation*}
\end{Proposition}
It is worth mentioning that, according to Section 6 of reference \cite{stheve}, $\beta=\{\eta_3,\eta_+,\eta_- \}$ is a left/right $P$--basis of $\Omega^1(P)$, i.e., 
\begin{equation}
    \label{ec.2.19.3.1}
    \Omega^1(P)=P\,\eta_-+P\,\eta_++P\,\eta_3 .
\end{equation}

Let us consider the universal differential envelope $\ast$--calculus $$(\Omega^\bullet(P),d,\ast) $$ of $(\Omega^1(P),d)$ (see \cite{micho1,stheve} or Appendix A for a brief summary).

\begin{Proposition}
\label{algo2}
    In  $(\Omega^\bullet(P),d,\ast) $, the following equations are satisfied
    \begin{equation*}
            \eta_+\,\eta_-=-q^2\eta_-\,\eta_+,\quad \eta_+\eta_3=-q^4\eta_3\,\eta_+, \quad \eta_3\,\eta_-=-q^4\,\eta_-\,\eta_3,\quad \eta^2_+=\eta^2_-=\eta^2_3=0,
    \end{equation*}
    \begin{equation*}
            d\eta_3=-(1+q^2)\,\eta_-\,\eta_+,\quad d\eta_+=q^{2}\,\eta_3\,\eta_+,\quad d\eta_-=-q^{-2}\,\eta_3\,\eta_-.  
    \end{equation*}
\end{Proposition}

\begin{proof}
    Consider the two--side ideal $S^\wedge$ of $\otimes^\bullet\mathfrak{su}_q(1,1)^\#$ (see equation (\ref{2.f12.1})). Since $\gamma^2$ $\in$ $\mathcal{R}_3$ and $$\Delta(\gamma^2)=\gamma^2\otimes \alpha^2+\alpha^\ast\gamma\otimes \gamma\alpha+q\alpha^\ast\gamma\otimes \alpha\gamma+\alpha^{\ast 2}\otimes \gamma^2,$$ we have that
   \begin{eqnarray*}
       \pi(\gamma^{2(1)})\otimes \pi(\gamma^{2(2)})&=&\pi(\gamma^2)\otimes \pi(\alpha^2)+q^{-1}\pi(\alpha^\ast\gamma)\otimes \pi(\alpha\gamma)+q\pi(\alpha^\ast\gamma)\otimes \pi(\alpha\gamma)+\pi(\alpha^{\ast 2})\otimes\pi(\gamma^2)
\\
&=&
q^{-1}\pi(\alpha^\ast\gamma)\otimes \pi(\alpha\gamma)+q\pi(\alpha^\ast\gamma)\otimes \pi(\alpha\gamma)
\\
&=&
q^{-1}\pi(\gamma)\otimes \pi(\gamma)+q\pi(\gamma)\otimes \pi(\gamma)
\\
&=&
q^{-1}\eta_+\otimes \eta_++q\eta_+\otimes \eta_+=(q^{-1}+q)\, \eta_+\otimes \eta_+.
   \end{eqnarray*}
This shows that $\eta_+\otimes \eta_+$ $\in$ $S^\wedge$ and hence $0=\eta_+\eta_+=\eta^2_+$ in $\Omega^\bullet(P)$.

Applying the same strategy to $\gamma^{\ast 2}$, $\gamma\gamma^\ast$, $\alpha\gamma-\gamma$ and $q^2\alpha+\alpha^\ast-(1+q^2)\mathbbm{1}$, we get the first column of the desired equations, except for the equation $\eta_3\,\eta_-=-q^4\eta_-\eta_3 $. However, since $\Omega^\bullet(P)$ is a graded $\ast$--algebra and $\eta_+\eta_3=-q^4\eta_3\eta_+$, by Proposition \ref{algo1} we get $$(\eta_+\eta_3)^\ast=-q^4(\eta_3\eta_+)^\ast\;\Longrightarrow\; -\eta^\ast_3\eta^\ast_+=q^4\eta^\ast_+\eta^\ast_3 \;\Longrightarrow\; \eta_3\,\eta_-=-q^4\,\eta_-\,\eta_3.$$

On the other hand, according to Theorem 10.7 of  \cite{stheve}, there is Maurer--Cartan formula in the universal differential envelope $\ast$--calculus, i.e., for every $\theta=\pi(p)$ $\in$ $\mathfrak{su}_q(1,1)^\#$, we have that
$$d\theta=d(\pi(p))=-\pi(p^{(1)})\pi(p^{(2)}).$$ Thus, by equation (\ref{ec.2.2}) and Proposition \ref{algo1} we obtain
$$d\pi(\alpha)=-\pi(\alpha)\pi(\alpha)-q\,\pi(\gamma^\ast)\pi(\gamma)=-{1\over (1+q^2)^2}\eta^2_3-\eta_-\eta_+=-\eta_-\eta_+,$$ and
$$d\pi(\alpha^\ast)=-\pi(\alpha^\ast)\pi(\alpha^\ast)-q\pi(\gamma)\pi(\gamma^\ast)=-{q^2\over (1+q^2)^2}\eta^2_3-\eta_+\,\eta_-= -\eta_+\,\eta_-=q^2\eta_-\,\eta_+,$$ which implies that $$d\eta_3=d\pi(\alpha-\alpha^\ast)=-\eta_-\,\eta_+-q^2\,\eta_-\,\eta_+=-(1+q^2)\,\eta_-\,\eta_+.$$ Moreover,
\begin{eqnarray*}
    d\eta_+=d\pi(\gamma)=-\pi(\gamma)\pi(\alpha)-\pi(\alpha^\ast)\pi(\gamma)&=&-{1\over 1+q^2}\eta_+\,\eta_3+{q^2\over 1+q^2}\eta_3\,\eta_+
    \\
    &=&
    {q^4+q^2\over 1+q^2}\eta_3\,\eta_+ 
    \\
    &=&
    q^2\,\eta_3\,\eta_+ 
\end{eqnarray*}
and 
\begin{eqnarray*}
    d\eta_-=d(\eta^\ast_+)=(d\eta_+)^\ast= q^2\,(\eta_3\,\eta_+)^\ast=-q^2\,\eta^\ast_+\,\eta^\ast_3\,=q^2\,\eta_-\,\eta_3=-q^{-2}\,\eta_3\,\eta_-.
\end{eqnarray*}
\end{proof}

It follows from the previous proposition that
\begin{equation}
    \label{ec.2.19.3}
    \Omega^2(P)=P\,\eta_-\,\eta_++P\,\eta_-\,\eta_3+P\,\eta_+\,\eta_3,
\end{equation}
\begin{equation}
    \label{ec.2.19.4}
    \Omega^3(P)=P\,\eta_-\,\eta_+\,\eta_3
\end{equation}
and
\begin{equation}
    \label{ec.2.19.5}
    \Omega^k(P)=\{0\} \qquad \mbox{ for }\qquad k\geq 4. 
\end{equation}
In addition, with the formulas of $d\alpha,d\alpha^\ast,d\gamma,d\gamma^\ast, d\eta_3,d\eta_-,d\eta_+$ and the graded Leibniz rule, we can calculate the whole differential map  $$d:\Omega^\bullet(P)\longrightarrow \Omega^\bullet(P).$$

\begin{Remark}
\label{rema1}
    From now on, the graded differential $\ast$--algebra $(\Omega^\bullet(P),d,\ast)$ will play the role of the space of quantum differential forms on $P$.
\end{Remark}

Now, let us take the right $G$--ideal (see equations (\ref{ec.2.8}), (\ref{ec.2.17})) 
\begin{equation}
\label{ec.2.29}
\mathcal{R}':=j(\mathcal{R}_3)=\langle\{q^2z+z^{\ast}-(1+q^2)\mathbbm{1} \}\rangle \subseteq \Ker(\epsilon').
\end{equation}
Consider the right adjoint $G$--corepresentation $$\Ad': G\longrightarrow G\otimes G,\qquad g\longmapsto g^{(2)}\otimes S'(g^{(1)})g^{(3)}.$$ By equation (\ref{ec.2.7}), it easy to see that 
\begin{equation}
\label{ec.2.30}
    \Ad'(g)=g\otimes \mathbbm{1}
\end{equation}
for all $g$ $\in$ $G$, so $\Ad'(\mathcal{R}')=\mathcal{R}'\otimes \mathbbm{1}$. In addition, by equation (\ref{ec.2.7}), it is also easy to see that $S'(\mathcal{R}')^\ast=\mathcal{R}'$. Thus, in light of \cite{stheve}, the $\ast$--FODC induced by $\mathcal{R}'$ 
\begin{equation}
\label{ec.2.31}
(\Gamma,d')
\end{equation}
is bicovariant. Let us take the $\C$--vector space
\begin{equation}
\label{ec.2.32}
\mathfrak{u}_q(1)^\#:={\Ker(\epsilon')\over \mathcal{R}'}
\end{equation}
and the corresponding quantum germs map (\cite{stheve})
\begin{equation}
\label{ec.2.33}
\pi': G\longrightarrow \mathfrak{u}_q(1)^\#
\end{equation}
defined by $$\pi'(g)=S'(g^{(1)})d'g^{(2)}$$ for all $g$ $\in$ $G$.
It follows directly from equation (\ref{ec.2.29}) that
\begin{equation}
\label{ec.2.33.1}
q^2\pi'(z^{n+1})+\pi'(z^{n-1})-(1+q^2)\pi'(z^n)=0
\end{equation}
for all $n$ $\in$ $\Z$, and hence,  $\mathfrak{u}_q(1)^\#=\mathrm{span}_\C\{\pi'(z)\}$. In particular 
\begin{equation}
\label{ec.2.33.2}
\pi'(z^{\ast})=-q^2\pi'(z).
\end{equation}
Thus, let us fix the following linear basis of $\mathfrak{u}_q(1)^\#$
\begin{equation}
\label{ec.2.34}
\beta':=\{\varsigma:=\pi'(z-z^\ast)\}.
\end{equation}
Now, we have, for example 
\begin{equation}
\label{ec.2.34.3}
\varsigma=(1+q^2)\,\pi'(z),\qquad  \varsigma=-(1+q^{-2})\,\pi'(z^\ast) ,\qquad\varsigma=q^{2}\pi'(z^2). 
\end{equation}

According to Section 6.4 of \cite{stheve}, there is a canonical structure of right $G$--module on $\mathfrak{u}_q(1)^\#$ given by
\begin{equation}
\label{ec.2.35}
    \pi'(g_1)\diamondsuit g_2:=\pi'(g_1g_2-\epsilon'(g_1)g_2))=S'(g^{(1)}_2)\pi'(g_1) g^{(2)}_2
\end{equation}
for $g_1$, $g_2$ $\in$ $G$. For example, a straightforward calculation shows that
    \begin{equation}
        \label{ec.2.35.1}
        \varsigma\diamondsuit z=q^{-2}\varsigma, \qquad \varsigma\diamondsuit z^\ast=q^2\varsigma, \qquad \varsigma\diamondsuit z^2=q^{-4}\varsigma,\qquad \varsigma\diamondsuit z^{\ast 2}=q^4\varsigma.
    \end{equation}

As in Proposition \ref{algo1}, using the basic properties of $\pi'$ we obtain
\begin{Proposition}
    \label{algo3}
    The following relations hold
   \begin{equation*}
      \varsigma^\ast=-\varsigma,\quad \varsigma \,z=q^{-2}z\,\varsigma,\quad \varsigma \,z^{\ast}=q^{2} z^{\ast}\varsigma, \quad d'z={1\over 1+q^2} z\,\varsigma,\quad d'z^{\ast}=-{q^2\over 1+q^2} z^\ast\,\varsigma.
    \end{equation*}
\end{Proposition}

Let us take the universal differential envelope $\ast$--calculus $$(\Gamma^\wedge,d',\ast) $$ of $(\Gamma,d')$ (see \cite{micho1,stheve} or Appendix A for a brief summary).

\begin{Proposition}
    \label{algo4}
    For every $k\geq 2$, we have $\Gamma^{\wedge\,k}=\{0\}$. 
\end{Proposition}

\begin{proof}
     Since $\beta'=\{\varsigma\}$ is also left $G$--basis of $\Gamma$ (\cite{stheve}) and the universal differential envelope $\ast$--calculus is always generated by its degree 0 elements, by the commutating relations of the last proposition it follows that every element of $\Gamma^{\wedge\,k}$ is of the form $$\vartheta=g\,\varsigma^k$$ for some $g$ $\in$ $G$. We are going to prove that $\varsigma^2=0$ and hence $\vartheta=0$ for $k\geq 2$. 
     
     Consider the two--side ideal $S^\wedge$ of $\otimes^\bullet\mathfrak{u}_q(1)^\#$ (see equation (\ref{2.f12.1})). Since $a:=q^2z+z^{\ast}-(1+q^2)\mathbbm{1}$ $\in$ $\mathcal{R}'$ and 
     \begin{eqnarray*}
        \Delta'(a)= \Delta'(q^2z+z^{\ast}-(1+q^2)\mathbbm{1})&=&q^2\Delta'(z)+\Delta'(z^{\ast})-(1+q^2)\Delta'(\mathbbm{1}) 
         \\
         &=&
         q^2z\otimes z+z^\ast\otimes z^\ast-(1+q^2)\mathbbm{1}\otimes \mathbbm{1},
     \end{eqnarray*}
     we get
     \begin{eqnarray*}
        \pi'(a^{(1)})\otimes \pi'(a^{(2)})&=& q^2\pi'(z)\otimes \pi'(z)+\pi'(z^\ast)\otimes \pi'(z^\ast)-(1+q^2)\pi'(\mathbbm{1})\otimes \pi'(\mathbbm{1})
        \\
        &=&
        q^2\pi'(z)\otimes \pi'(z)+\pi'(z^\ast)\otimes \pi'(z^\ast)
        \\
        &=&
        q^2\pi'(z)\otimes \pi'(z)+q^4\pi'(z)\otimes \pi'(z)
        \\
        &=&
        (q^2+q^4)\pi'(z)\otimes \pi'(z)
     \end{eqnarray*}
     This implies that $\pi'(z)\otimes \pi'(z)$ $\in$ $S^\wedge$ and hence $\varsigma\otimes \varsigma$ $\in$ $S^\wedge$. Thus, $0=\varsigma\,\varsigma=\varsigma^2$ in $\Gamma^\wedge$.
\end{proof}

\begin{Remark}
\label{rema2}
    From now on, the graded differential $\ast$--algebra $(\Gamma^\wedge,d',\ast)$ will play the role of the space of quantum differential forms on $G$.
\end{Remark}

The space of quantum differential forms on $P$ and the space of quantum differential forms on $G$ are related through the following proposition

\begin{Proposition}
    \label{algo5}
    The linear epimorphism $j:P\longrightarrow G$ of equation (\ref{ec.2.8}) can be extended to a graded differential $\ast$--algebra morphism $$j:\Omega^\bullet(P)\longrightarrow \Gamma^\wedge.$$ In particular, $j$ is surjective and 
    \begin{equation}
    \label{ec.necesito}
        j\circ \pi=\pi'\circ j.
    \end{equation}
\end{Proposition}
\begin{proof}
 Let us denote by $j^0$ the map $j$ of equation (\ref{ec.2.8}).
    Since $\beta=\{\eta_3,\eta_+,\eta_- \}$ is a $P$--basis of $\Omega^1(P)$, we define $$j^1:\Omega^1(P)\longrightarrow \Gamma^{\wedge 1}=\Gamma$$ as the linear map given by  $$j^1(p\,\eta_3):=j^0(p)\,\varsigma,\qquad j^1(p\,\eta_-):=j^1(p\,\eta_+):=0$$ for all $p$ $\in$ $P$. 

    In light of Propositions \ref{algo1}, \ref{algo3}, we have $$j^1(p\,d\alpha)={1\over 1+q^2}\,j^1(p\,\alpha\,\eta_3)+q\,j^1(p\,\gamma^\ast\eta_+)={1\over 1+q^2}\,j^0(p)\,z\,\varsigma=j^0(p)\,d'z=j^0(p)\,d'(j^0(\alpha)),$$ $$j^1(p\,d\alpha^\ast)=-{q^2\over 1+q^2}\,j^1(p\,\alpha^\ast\,\eta_3)+j^1(p\,\gamma\eta_-)=-{q^2\over 1+q^2}\,j^0(p)\,z^\ast\,\varsigma=j^0(p)\,d'z^\ast=j^0(p)\,d'(j^0(\alpha^\ast)),$$ $$j^1(p\,d\gamma)=0=j^0(p)\,d'(j^0(\gamma)),\qquad j(p\,d\gamma^\ast)=0=j^0(p)\,d'(j^0(\gamma^\ast)),$$ for all $p$ $\in$ $P$ and by linearity,  we can conclude that $$j^1(a\,db)=j^0(a)\,d'(j^0(b))$$ for all $a$, $b$ $\in$ $P$. So, the proposition follows from Proposition \ref{A.1} of Appendix A. In particular, since $\Im(j)=\Gamma^\wedge$, $j$ is surjective. Furthermore, since $j$ is a morphism of $\ast$--Hopf algebras (in degree $0$), for all $p$ $\in$ $P$ we obtain $$j(\pi(p))=j(S(p^{(1)})dp^{(2)})=S'(j(p^{(1)}))d'j(p^{(2)})=\pi'(j(p))$$ and equation (\ref{ec.necesito}) is satisfied.
\end{proof}

\noindent Furthermore, we have
\begin{Proposition}
 \label{algo6}
    The $\ast$--algebra morphism $$\Delta_P:P\longrightarrow P\otimes G$$ of equation (\ref{ec.2.9}) can be extended to a graded differential $\ast$--algebra morphism $$\Delta_{\Omega^\bullet(P)}:\Omega^\bullet(P)\longrightarrow \Omega^\bullet(P)\otimes \Gamma^\wedge.$$ In the last tensor product, we have considered the tensor product of graded differential $\ast$--algebras.
\end{Proposition}
\begin{proof}
    The proof of this proposition follows the same strategy of the last one. Let us denote by $\Delta^0_{\Omega^\bullet(P)}$ the map $\Delta_P$ of equation (\ref{ec.2.9}). Since $\beta=\{\eta_3,\eta_+,\eta_- \}$ is a $P$--basis of $\Omega^1(P)$, we define $$\Delta^1_{\Omega^\bullet(P)}: \Omega^1(P)\longrightarrow (\Omega^\bullet(P)\otimes \Gamma^\wedge)^1:=(\Omega^1(P)\otimes G) \oplus (P\otimes \Gamma)$$ as the linear map given by $$\Delta^1_{\Omega^\bullet(P)}(p\,\eta_3):=\Delta^0_{\Omega^\bullet(P)}(p)\,(\eta_3\otimes \mathbbm{1}+\mathbbm{1}\otimes \varsigma)=p^{(0)}\eta_3\otimes p^{(1)} + p^{(0)}\otimes p^{(1)}\,\varsigma,$$ $$ \Delta^1_{\Omega^\bullet(P)}(p\,\eta_+):=\Delta^0_{\Omega^\bullet(P)}(p)\,(\eta_+\otimes z^2)=p^{(0)}\eta_+\otimes p^{(1)}z^2,$$ $$\Delta^1_{\Omega^\bullet(P)}(p\,\eta_-):=\Delta^0_{\Omega^\bullet(P)}(p)\,(\eta_-\otimes z^{\ast 2})=p^{(0)}\eta_-\otimes p^{(1)}z^{\ast 2} $$ for all $p$ $\in$ $P$.

    According to Propositions \ref{algo1}, \ref{algo3}, we get
    \begin{eqnarray*}
        \Delta^1_{\Omega^\bullet(P)}(p\,d\alpha)&=&{1\over 1+q^2}\,\Delta^1_{\Omega^\bullet(P)}(p\,\alpha\,\eta_3)+q\,\Delta^1_{\Omega^\bullet(P)}(p\,\gamma^\ast\eta_+)
        \\
        &=&
        {1\over 1+q^2}\,\Delta^0_{\Omega^\bullet(P)}(p)\,\Delta^1_{\Omega^\bullet(P)}(\alpha\,\eta_3)+q\,\Delta^0_{\Omega^\bullet(P)}(p)\,\Delta^1_{\Omega^\bullet(P)}(\gamma^\ast\eta_+)
        \\
        &=&
        {1\over 1+q^2}\,\Delta^0_{\Omega^\bullet(P)}(p)\,(\alpha\, \eta_3\otimes z+\alpha\otimes z\,\varsigma) +
        q\,\Delta^0_{\Omega^\bullet(P)}(p)\,(\gamma^\ast\eta_+\otimes z^\ast\,z^2)
        \\
        &=&
        \Delta^0_{\Omega^\bullet(P)}(p)\,\left({1\over 1+q^2}\,\alpha\, \eta_3\otimes z+
        q\,\gamma^\ast\eta_+\otimes z+{1\over 1+q^2}\,\alpha\otimes z\,\varsigma\right)
        \\
        &=&
        \Delta^0_{\Omega^\bullet(P)}(p)\,\left( ({1\over 1+q^2}\,\alpha\, \eta_3+
        q\,\gamma^\ast\eta_+)\otimes z+{1\over 1+q^2}\,\alpha\otimes z\,\varsigma\right)
        \\
        &=&
        \Delta^0_{\Omega^\bullet(P)}(p)\,\left( d\alpha\otimes z+ \alpha\otimes d'z\right)
        \\
        &=&
        \Delta^0_{\Omega^\bullet(P)}(p)\,d_\otimes(\alpha\otimes z)
        \\
        &=&
        \Delta^0_{\Omega^\bullet(P)}(p)\,d_\otimes(\Delta^0_{\Omega^\bullet(P)}(\alpha)),
    \end{eqnarray*}
    for all $p$ $\in$ $P$, where $d_\otimes$ is the differential of the graded differential $\ast$--algebra $$(\Omega(P)\otimes \Gamma^\wedge,d_\otimes,\ast).$$
    Also, we obtain 
    \begin{eqnarray*}
        \Delta^1_{\Omega^\bullet(P)}(p\,d\alpha^\ast)&=&-{q^2\over 1+q^2}\,\Delta^1_{\Omega^\bullet(P)}(p\,\alpha^\ast\,\eta_3)+\Delta^1_{\Omega^\bullet(P)}(p\,\gamma\,\eta_-)
        \\
        &=&
        -{q^2\over 1+q^2}\,\Delta^0_{\Omega^\bullet(P)}(p)\,\Delta^1_{\Omega^\bullet(P)}(\alpha^\ast\,\eta_3)+\Delta^0_{\Omega^\bullet(P)}(p)\,\Delta^1_{\Omega^\bullet(P)}(\gamma\,\eta_-)
        \\
        &=&
        -{q^2\over 1+q^2}\,\Delta^0_{\Omega^\bullet(P)}(p)\,(\alpha^\ast\, \eta_3\otimes z^\ast+\alpha^\ast\otimes z^\ast\,\varsigma) +
        \Delta^0_{\Omega^\bullet(P)}(p)\,(\gamma\,\eta_-\otimes z\,z^{\ast 2})
        \\
        &=&
        \Delta^0_{\Omega^\bullet(P)}(p)\,\left(-{q^2\over 1+q^2}\,\alpha^\ast\, \eta_3\otimes z^\ast+
        \gamma\,\eta_-\otimes z^\ast-{q^2\over 1+q^2}\,\alpha^\ast\otimes z^\ast\,\varsigma\right)
        \\
        &=&
        \Delta^0_{\Omega^\bullet(P)}(p)\,\left((-{q^2\over 1+q^2}\,\alpha^\ast\, \eta_3+
        \gamma\,\eta_-)\otimes z^\ast-{q^2\over 1+q^2}\,\alpha^\ast\otimes z^\ast\,\varsigma\right)
        \\
        &=&
        \Delta^0_{\Omega^\bullet(P)}(p)\,\left( d\alpha^\ast\otimes z^\ast+ \alpha^\ast\otimes d'z^\ast\right)
        \\
        &=&
        \Delta^0_{\Omega^\bullet(P)}(p)\,d_\otimes(\alpha^\ast\otimes z^\ast)
        \\
        &=&
        \Delta^0_{\Omega^\bullet(P)}(p)\,d_\otimes(\Delta^0_{\Omega^\bullet(P)}(\alpha^\ast)),
    \end{eqnarray*}
    and
      \begin{eqnarray*}
        \Delta^1_{\Omega^\bullet(P)}(p\,d\gamma)&=&{1\over 1+q^2}\,\Delta^1_{\Omega^\bullet(P)}(p\,\gamma\,\eta_3)+\Delta^1_{\Omega^\bullet(P)}(p\,\alpha^\ast\eta_+)
        \\
        &=&
        {1\over 1+q^2}\,\Delta^0_{\Omega^\bullet(P)}(p)\,\Delta^1_{\Omega^\bullet(P)}(\gamma\,\eta_3)+\Delta^0_{\Omega^\bullet(P)}(p)\,\Delta^1_{\Omega^\bullet(P)}(\alpha^\ast\eta_+)
        \\
        &=&
        {1\over 1+q^2}\,\Delta^0_{\Omega^\bullet(P)}(p)\,(\gamma\, \eta_3\otimes z+\gamma\otimes z\,\varsigma) +
        \Delta^0_{\Omega^\bullet(P)}(p)\,(\alpha^\ast\eta_+\otimes z^\ast\,z^2)
        \\
        &=&
        \Delta^0_{\Omega^\bullet(P)}(p)\,\left({1\over 1+q^2}\,\gamma\, \eta_3\otimes z+
        \alpha^\ast\eta_+\otimes z+{1\over 1+q^2}\,\gamma\otimes z\,\varsigma\right)
        \\
        &=&
        \Delta^0_{\Omega^\bullet(P)}(p)\,\left( ({1\over 1+q^2}\,\gamma\, \eta_3+
        \alpha^\ast\eta_+)\otimes z+{1\over 1+q^2}\,\gamma\otimes z\,\varsigma\right)
        \\
        &=&
        \Delta^0_{\Omega^\bullet(P)}(p)\,\left( d\gamma\otimes z+ \gamma\otimes d'z\right)
        \\
        &=&
        \Delta^0_{\Omega^\bullet(P)}(p)\,d_\otimes(\gamma\otimes z)
        \\
        &=&
        \Delta^0_{\Omega^\bullet(P)}(p)\,d_\otimes(\Delta^0_{\Omega^\bullet(P)}(\gamma)).
    \end{eqnarray*}
     Finally
     
       \begin{eqnarray*}
        \Delta^1_{\Omega^\bullet(P)}(p\,d\gamma^\ast)&=&-{q^2\over 1+q^2}\,\Delta^1_{\Omega^\bullet(P)}(p\,\gamma^\ast\,\eta_3)+q^{-1}\Delta^1_{\Omega^\bullet(P)}(p\,\alpha\,\eta_-)
        \\
        &=&
        -{q^2\over 1+q^2}\,\Delta^0_{\Omega^\bullet(P)}(p)\,\Delta^1_{\Omega^\bullet(P)}(\gamma^\ast\,\eta_3)+q^{-1}\Delta^0_{\Omega^\bullet(P)}(p)\,\Delta^1_{\Omega^\bullet(P)}(\alpha\,\eta_-)
        \\
        &=&
        -{q^2\over 1+q^2}\,\Delta^0_{\Omega^\bullet(P)}(p)\,(\gamma^\ast\, \eta_3\otimes z^\ast+\gamma^\ast\otimes z^\ast\,\varsigma) + q^{-1}
        \Delta^0_{\Omega^\bullet(P)}(p)\,(\alpha\,\eta_-\otimes z\,z^{\ast 2})
        \\
        &=&
        \Delta^0_{\Omega^\bullet(P)}(p)\,\left(-{q^2\over 1+q^2}\,\gamma^\ast\, \eta_3\otimes z^\ast+
        q^{-1}\alpha\,\eta_-\otimes z^\ast-{q^2\over 1+q^2}\,\gamma^\ast\otimes z^\ast\,\varsigma\right)
        \\
        &=&
        \Delta^0_{\Omega^\bullet(P)}(p)\,\left((-{q^2\over 1+q^2}\,\gamma^\ast\, \eta_3+q^{-1}
        \alpha\,\eta_-)\otimes z^\ast-{q^2\over 1+q^2}\,\gamma^\ast\otimes z^\ast\,\varsigma\right)
        \\
        &=&
        \Delta^0_{\Omega^\bullet(P)}(p)\,\left( d\gamma^\ast\otimes z^\ast+ \gamma^\ast\otimes d'z^\ast\right)
        \\
        &=&
        \Delta^0_{\Omega^\bullet(P)}(p)\,d_\otimes(\gamma^\ast\otimes z^\ast)
        \\
        &=&
        \Delta^0_{\Omega^\bullet(P)}(p)\,d_\otimes(\Delta^0_{\Omega^\bullet(P)}(\gamma^\ast)).
    \end{eqnarray*}
    Therefore, by linearity, we can conclude that $$\Delta^1_{\Omega^\bullet(P)}(a\,db)= \Delta^0_{\Omega^\bullet(P)}(a)\;d_{\otimes}(\Delta^0_{\Omega^\bullet(P)}(b))$$ for all $a$, $b$ $\in$ $P$. Thus, the proposition follows from Proposition \ref{A.1} of Appendix A. 
\end{proof}

In this way, the pair 
\begin{equation}
    \label{ec.2.20.1}
    (\Omega^\bullet(P),\Delta_{\Omega^\bullet(P)})
\end{equation}
constitutes a differential calculus on $\zeta_H$ in the sense of \cite{micho2,micho3,stheve}. In the rest of this subsection, we will study some properties of this differential calculus.

\begin{Definition}
 \label{algo7}
    We define the space of horizontal quantum differential forms (or simply the space of horizontal forms) as
    $$\Hor^\bullet\,P:=\{\varphi \in \Omega^\bullet(P)\mid \Delta_{\Omega^\bullet(P)}(\varphi)\in \Omega^\bullet(P)\otimes G \}.$$
\end{Definition}

In accordance with \cite{micho2,stheve}, the space of horizontal forms is a graded $\ast$--subalgebra of $\Omega^\bullet(P)$.

\begin{Proposition}
    \label{algo8}
    We have $$ \Hor^0\,P=P, \qquad \Hor^1\,P=P\,\eta_-+P\,\eta_+, \qquad \Hor^2\,P=P\,\eta_-\,\eta_+. $$
\end{Proposition}

\begin{proof}
    By construction, $\Hor^0\,P=P$. The other two identities  follow directly from equations (\ref{ec.2.19.3.1}), (\ref{ec.2.19.3}), (\ref{ec.2.19.4}), the fact that $\Delta_{\Omega^\bullet(P)}$ is graded multiplicative and  $$\Delta_{\Omega^\bullet(P)}(\eta_3)=\eta_3\otimes \mathbbm{1}+\mathbbm{1}\otimes \varsigma,\qquad \Delta_{\Omega^\bullet(P)}(\eta_+)=\eta_+\otimes z^2,\qquad \Delta_{\Omega^\bullet(P)}(\eta_-)=\eta_-\otimes z^{\ast 2}.$$
\end{proof}

In light of \cite{micho2,stheve}, the restriction of $\Delta_{\Omega^\bullet(P)}$ on $\Hor^\bullet\, P$ 
\begin{equation}
    \label{ec.2.38}
    \Delta_{\Hor}:=\Delta_{\Omega^\bullet(P)}|_{\Hor^\bullet\, P}:\Hor^\bullet\, P\longrightarrow \Hor^\bullet\, P\otimes G
\end{equation}
is a $G$--corepresentation.

\begin{Definition}
    \label{baseforms}
    We define the space of quantum differential forms on the base space (or simply the space of base forms) as
    $$\Omega^\bullet(\H^2_q):=\{\mu \in \Omega^\bullet(P)\mid \Delta_{\Omega^\bullet(P)}(\mu)=\mu\otimes \mathbbm{1} \}.$$
\end{Definition}
According to \cite{micho2,stheve}, $\Omega^\bullet(\H^2_q)$ is a graded $\ast$--subalgebra and  $d(\Omega^\bullet(\H^2_q))\subseteq \Omega^\bullet(\H^2_q)$. In other words,
$$(\Omega^\bullet(\H^2_q),d,\ast)$$ is a graded differential $\ast$--subalgebra of $(\Omega^\bullet(P),d,\ast)$. Notice that $\Omega^\bullet(\H^2_q)\subseteq \Hor^\bullet\,P$.

\begin{Proposition}
    \label{2.2}
    The following relations hold
\begin{gather*}
\Omega^0(\H^2_q)= \H^2_q,
\\
\begin{gathered}
\Omega^1(\H^2_q)= \{x\,\eta_-+y\,\eta_+\in \Hor^1 \,P \mid \Delta_P(x)=x\otimes z^2,\; \Delta_P(y)=y\otimes z^{\ast 2}\},
\end{gathered}
\\
\Omega^2(\H^2_q)= \H^2_q\,\eta_-\eta_+.
\end{gather*}
\end{Proposition}
\begin{proof}
By construction, it is clear that $\Omega^0(\H^2_q)=\H^2_q$. Let $\mu=x\eta_-+y\eta_+$ $\in$ $\Omega^1(\H^2_q)$, with
$$x=\sum_{m,k,l}\lambda_{mkl}\,\alpha^m\gamma^k\gamma^{\ast\,l},\qquad y=\sum_{r,s,t}\lambda'_{rst}\,\alpha^r\gamma^s\gamma^{\ast\,t} $$
for $\lambda_{mkl}$, $\lambda'_{rst}$ $\in$ $\C$, $m,r$ $\in$ $\Z$, $k,l,s,t$ $\in$ $\N_0$. Then
\begin{eqnarray*}
(x\eta_-+y\eta_+)\otimes \mathbbm{1}= \Delta_\Hor(x\eta_-+y\eta_+)
&=& \sum_{m,k,l}\lambda_{mkl}\,\alpha^m\gamma^k\gamma^{\ast\,l}\eta_-\otimes z^{m+k-l-2}\\ 
 & +&  \sum_{r,s,t}\lambda'_{rst}\,\alpha^r\gamma^s\gamma^{\ast\,t}\eta_+\otimes z^{r+s-t+2};
\end{eqnarray*} 
so $$\sum_{m+k-l\not=2}\lambda_{mkl}\,\alpha^m\gamma^k\gamma^{\ast\,l}\eta_-\otimes (z^{m+k-l-2}-\mathbbm{1}) + \sum_{r+s-t\not=-2}\lambda'_{rst}\,\alpha^r\gamma^s\gamma^{\ast\,t}\eta_+\otimes (z^{r+s-t+2}-\mathbbm{1})=0.$$ Since the set $\{\alpha^m\gamma^k\gamma^{\ast\,l}\eta_-\,,\,\alpha^r\gamma^s\gamma^{\ast\,t}\eta_+ \}$ is a linear basis of $\Hor^1\,P$ and $\{z^a \}_{a\in \Z}$ is a linear basis of $G$, we conclude that $$\lambda_{mkl}=\mu_{rst}=0 $$ with $m+k-l\not=2$ and $r+s-t\not=-2$. This proves our claim for $\Omega^1(\H^2_q)$. 

An analogous calculation shows that $\Omega^2(\H^2_q)=\H^2_q\,\eta_-\eta_+.$
\end{proof}

\begin{Remark}
    \label{rema3}
    According to the proof of the last proposition, if $x$ $\in$ $P$ satisfies $$\Delta_P(x)=x\otimes z^2,\qquad \mbox{ then } \qquad x=\sum_{m,k,l}\lambda_{mkl}\,\alpha^m\gamma^k\gamma^{\ast\,l}$$ such that $m+k-l=2$. In the same way, if $y$ $\in$ $P$ satisfies $$\Delta_P(y)=y\otimes z^{\ast 2}, \qquad \mbox{ then } \qquad y=\sum_{m,k,l}\lambda_{mkl}\,\alpha^m\gamma^k\gamma^{\ast\,l} $$ such that $m+k-l=-2$. Moreover, if $b$ $\in$ $\H^2_q$, then $$\Delta_P(b)=b\otimes \mathbbm{1}\qquad \mbox{ which implies that } \qquad b=\sum_{m,k,l}\lambda_{mkl}\,\alpha^m\gamma^k\gamma^{\ast\,l} $$ such that $m+k-l=0$. By Proposition \ref{algo1}, it follows that 
    \begin{equation}
    \label{ec.2.36}
        \eta_\pm\,x=q^{-2}\,x\,\eta_\pm,\qquad \eta_\pm\,y=q^{2}\,y\,\eta_\pm, \qquad \eta_\pm\, b=b\, \eta_\pm.
    \end{equation}
\end{Remark}

To finalize this section, we have

\begin{Definition}
\label{2.3}
    We define the adjoint $G$--corepresentation $$\ad':\mathfrak{u}_q(1)^\#\longrightarrow \mathfrak{u}_q(1)^\#\otimes G$$ as the one define for the formula (\cite{stheve}) $$\ad'\circ \pi'=(\pi'\otimes \id_G)\circ\Ad'.$$ By equation (\ref{ec.2.30}), it follows that 
    \begin{equation}
        \label{ec.2.37}
        \ad'(\theta)=\theta\otimes \mathbbm{1}
    \end{equation}
     for all $\theta$ $\in$ $\mathfrak{u}_q(1)^\#$.
\end{Definition}

\section{Differential Geometry of $\zeta_H$}

In this section, we are going to obtain geometrical properties of $\zeta_H$ with respect to the differential calculus of equation (\ref{ec.2.20.1}).

\subsection{Quantum Principal Connections, Covariant Derivatives and Curvatures} 

This subsection will be based on \cite{micho1,micho2,micho3,stheve,perla}. Let us star with the following definitions.

\begin{Definition}
    \label{3.1}
    A quantum principal connection (abbreviated qpc) is a linear map $$\omega: \mathfrak{u}_q(1)^\#\longrightarrow \Omega^1(P) $$ such that $$\omega(\theta^\ast)=\omega(\theta)^\ast$$ and $$\Delta_{\Omega^\bullet(P)}(\omega(\theta))=(\omega\otimes \id_G)\ad'(\theta)+\mathbbm{1}\otimes \theta$$ for all $\theta$ $\in$ $\mathfrak{u}_q(1)^\#$. In addition, a qpc $\omega$ is regular if
    $$\omega(\theta)\varphi=(-1)^k\varphi^{(0)}\omega(\theta\diamondsuit \varphi^{(1)}) $$ for all $\theta$ $\in$ $\mathfrak{u}_q(1)^\#$, $\varphi$ $\in$ $\Hor^k\,P$ (see equation (\ref{ec.2.35})).
\end{Definition}

In accordance with references \cite{micho2,stheve}, the set of all qpc's 
\begin{equation}
    \label{ec.3.1}
    \mathfrak{qpc}(\zeta_H)
\end{equation}
is a real affine space modeled by the real vector space
\begin{equation}
    \label{ec.3.2}
    \begin{aligned}
        \overrightarrow{\mathfrak{qpc}(\zeta_{H})}:=\{ \lambda: \mathfrak{u}_q^\#\longrightarrow \Hor^1\,P &\mid 
         \lambda \mbox{ is a linear map such that } 
        \\
       & \lambda\circ \ast=\ast \circ \lambda \;\;\mbox{ and }\;\; 
       \Delta_\Hor\circ \lambda=(\lambda\otimes \id_G)\circ \ad' \}.
    \end{aligned}
\end{equation}

\begin{Definition}
    \label{3.2}
    Given a qpc $\omega$, its covariant derivative is the first--order linear map $$D^\omega: \Hor^\bullet\,P\longrightarrow \Hor^\bullet\,P $$ such that $$D^\omega(\varphi)=d\varphi-(-1)^k\varphi^{(0)}\omega(\pi(\varphi^{(1)})).$$ Furthermore, the dual covariant derivative is the first--order linear map 
    $$\widehat{D}^\omega: \Hor^\bullet\,P\longrightarrow \Hor^\bullet\,P $$ given by $$\widehat{D}^\omega=\ast\circ D^\omega\circ \ast.$$ In concrete (\cite{sald1,micho3}),
    $$\widehat{D}^\omega(\varphi)=d\varphi+ \omega(\pi'(S^{-1}(\varphi^{(1)})))\varphi^{(0)}.$$
\end{Definition}

According to \cite{micho3}, the covariant derivatives $D^\omega$, $\widehat{D}^\omega$ of $\omega$ satisfies the graded Leibniz rule if and only if $\omega$ is regular. Furthermore, when $\omega$ is regular, we have (\cite{micho2,micho3})
\begin{equation}
    \label{add.1}
    D^\omega=\widehat{D}^\omega \quad \mbox{ or equivalently }\quad D^\omega\circ \ast=\ast\circ D^\omega.
\end{equation}
By defining (see equation (\ref{ec.2.38}))
\begin{equation}
    \label{ec.3.3}
    \begin{aligned}
    \Mor(\Delta_\Hor,\Delta_\Hor):=\{&\tau: \Hor^\bullet\,P\longrightarrow \Hor^\bullet\,P
    \mid 
    \\
    &\tau \mbox{ is a linear map such that }  
    \Delta_\Hor\circ \tau=(\tau\otimes \id_G)\circ \Delta_\Hor \},
    \end{aligned}
\end{equation}
we obtain (\cite{micho2,stheve})

\begin{equation}
    \label{ec.3.4}
   D^\omega,\; \widehat{D}^\omega \in \Mor(\Delta_\Hor,\Delta_\Hor)
\end{equation}
for every $\omega$ $\in$ $\mathfrak{qpc}(\zeta_H)$. Furthermore  (\cite{micho2,stheve}) 
\begin{equation}
    \label{ec.3.4.1}
D^\omega|_{\Omega^\bullet(\H^2_q)}=\widehat{D}^\omega|_{\Omega^\bullet(\H^2_q)}=d|_{\Omega^\bullet(\H^2_q)}.
\end{equation}

On the other hand, in Durdevich formulation of qpb's we have the following definition (\cite{micho2,stheve})

\begin{Definition}
    \label{3.3}
    We define an {\it embedded differential} as a linear map $$\Theta:\mathfrak{u}_q(1)^\#\longrightarrow \mathfrak{u}_q(1)^\#\otimes \mathfrak{u}_q(1)^\# $$ such that
\begin{enumerate}
    \item  $\Theta$ $\in$ $\Mor(\ad',\ad^{'\otimes 2})$, where $\ad^{'\otimes 2}$ is the tensor product $\G$--corepresentation of $\ad'$ with itself (\cite{woro1}). 
    \item If $\Theta(\theta)=\displaystyle\sum^m_{i,j=1}\theta_i\otimes \hat{\theta}_j$, then $d'\theta=\displaystyle\sum^m_{i,j=1}\theta_i \hat{\theta}_j$ and $\Theta(\theta^\ast)=\displaystyle\sum^m_{i,j=1}\hat{\theta}^\ast_j \otimes \theta^\ast_i$ for all $\theta$ $\in$ $\mathfrak{u}_q(1)^\#$.
\end{enumerate}
Fix an embedded differential $\Theta$. Then, the curvature of a qpc $\omega$ is defined as the linear map $$R^\omega:\mathfrak{u}_q(1)^\#\longrightarrow \Omega^2(P)$$ given by $$R^\omega=d\omega-\langle\omega,\omega\rangle \qquad \mbox{ where }\qquad 
\langle\omega,\omega\rangle:=m_\Omega\circ (\omega\otimes \omega)\circ \Theta,$$ with $m_\Omega:\Omega^\bullet(P)\otimes \Omega^\bullet(P)\longrightarrow \Omega^\bullet(P)$ the product map.
\end{Definition}
    
According to \cite{micho2,stheve}, we have   
\begin{equation}
    \label{ec.3.5}
    R^\omega \in \Mor^2(\ad',\Delta_\Hor),
\end{equation}
for all $\omega$ $\in$ $\mathfrak{qpc}(\zeta_H)$,  where 
 \begin{equation}
    \label{ec.3.6}
    \begin{aligned}
       \Mor^2(\ad',\Delta_\Hor)= \{ \tau:  \mathfrak{u}_q^\# & \longrightarrow \Hor^2\,P \mid\\ 
        & \tau \mbox{ is a linear map such that }   
        \Delta_\Hor\circ \tau=(\tau\otimes \id_G)\circ \ad' \}.
    \end{aligned}
\end{equation}

The auxiliary map $\Theta$ is introduced in the definition of curvature in the Durdevich formulation to ensure that $R^\omega$ has $\mathfrak{u}_q(1)^\#$ as its domain and that equation (\ref{ec.3.5}) holds. These properties are the non--commutative geometrical counterparts of some fundamental results concerning curvature in differential geometry. For example,  equation (\ref{ec.3.5}) is the \emph{dualization} of the fact that, in differential geometry, the curvature of a principal connection is always a basic differential $2$--form of type adjoint \cite{nodg}. The reader is encouraged to consult references \cite{saldym,micho2,stheve} for more details about this definition of the curvature. 

In the rest of this subsection, we will characterize properly all the geometrical objects that we have just presented. First of all, we have

\begin{Proposition}
    \label{3.4}
    The only embedded differential $$\Theta:\mathfrak{u}_q(1)^\#\longrightarrow \mathfrak{u}_q(1)^\#\otimes \mathfrak{u}_q(1)^\#$$ is $\Theta=0$
\end{Proposition}
\begin{proof}
    Let $\Theta$ be an embedded differential. Since $\Gamma^{\wedge\,2}=\{0\}$ (see Proposition \ref{algo4}), we have that $d'\theta=0$ for all $\theta$ $\in$ $\mathfrak{u}_q(1)^\#$ and therefore $\Theta(\theta)=0$ for all $\theta$ $\in$ $\mathfrak{u}_q(1)^\#$.
\end{proof}

The last proposition implies that for every $\omega$ $\in$ $\mathfrak{qpc}(\zeta_H)$, its curvature satisfies

\begin{equation}
    \label{ec.3.7}
    R^\omega=d\omega
\end{equation}
as in differential geometry for principal $U(1)$--bundles \cite{nodg}.

\begin{Proposition}
\label{3.5}
The linear map
\begin{equation*}
\omega^c:\mathfrak{u}_q(1)^\#  \longrightarrow \Omega^1(P)
\end{equation*}
such that $$\omega^c(\varsigma)=\eta_3 $$
is a regular qpc and it is called the canonical quantum principal connection of $\zeta_{H}$. Furthermore,  its covariant derivative is the operator $$D:=D^{\omega^c}: \Hor^\bullet\,P\longrightarrow \Hor^\bullet\,P$$ characterized by the graded Leibniz rule and the formulas $$D(p)=p^{(1)}(\pi_-(p^{(2)})+\pi_+(p^{(2)})),\qquad D(\eta_-)=D(\eta_+)=0$$ where $\Delta(p)=p^{(1)}\otimes p^{(2)}$, $\pi_\pm:=\rho_\pm\circ \pi$, with $$\pi:P\longrightarrow \mathfrak{su}_q(1,1)^\ast$$ the quantum germs map and $$\rho_\pm:\mathfrak{su}_q(1,1)^\ast \longrightarrow \mathrm{span}_\C\{\eta_\pm\}$$ the canonical projections.  In addition $$R^{\omega^c}(\varsigma)=-(1+q^2)\,\eta_-\eta_+.$$
\end{Proposition}
\begin{proof}
Since $\eta^\ast_3=-\eta_3$ and $\varsigma^\ast=-\varsigma$ (see Section 2), it follows that $$\omega^c(\varsigma^\ast)=\omega^c(\varsigma)^\ast$$ and by linearity we conclude that
 $$\omega^c(\theta^\ast)=\omega^c(\theta)^\ast$$ for all $\theta$ $\in$ $\mathfrak{u}_q(1)^\#$.
 
 Moreover, by equation (\ref{ec.2.37}) we obtain 
 \begin{eqnarray*}
     \Delta_{\Omega^\bullet(P)}(\omega^c(\varsigma))=\Delta_{\Omega^\bullet(P)}(\eta_3)=\eta_3\otimes \mathbbm{1}+\mathbbm{1}\otimes \varsigma=(\omega^c\otimes \id_G)\ad'(\varsigma)+\mathbbm{1}\otimes \varsigma
 \end{eqnarray*}
 and by linearity we conclude that 
 $$\Delta_{\Omega^\bullet(P)}(\omega^c(\theta))=(\omega^c\otimes \id_G)\ad'(\theta)+\mathbbm{1}\otimes \theta $$ for all $\theta$ $\in$ $\mathfrak{u}_q(1)^\#$. Therefore, $\omega^c$ is a qpc.

On the other hand, by Propositions \ref{algo1}, \ref{algo2} and equation (\ref{ec.2.35.1}), we get
$$\omega^c(\varsigma)\,\alpha=\eta_3\,\alpha=q^{-2}\alpha\,\eta_3=\alpha\,\omega^c(q^{-2}\varsigma)=\alpha\,\omega^c(\varsigma\diamondsuit z),$$ 
$$\omega^c(\varsigma)\,\alpha^\ast=\eta_3\,\alpha^\ast=q^{2}\alpha^\ast\,\eta_3=\alpha^\ast\,\omega^c(q^{-2}\varsigma)=\alpha^\ast\,\omega^c(\varsigma\diamondsuit z^\ast),$$ 
$$\omega^c(\varsigma)\,\gamma=\eta_3\,\gamma=q^{-2}\gamma\,\eta_3=\gamma\,\omega^c(q^{-2}\varsigma)=\gamma\,\omega^c(\varsigma\diamondsuit z),$$ $$\omega^c(\varsigma)\,\gamma^\ast=\eta_3\,\gamma^\ast=q^{2}\gamma^\ast\,\eta_3=\gamma^\ast\,\omega^c(q^{2}\varsigma)=\gamma^\ast\,\omega^c(\varsigma\diamondsuit z^\ast),$$ $$\omega^c(\varsigma)\,\eta_+=\eta_3\,\eta_+=-q^{-4}\,\eta_3\,\eta_+=-\eta_+\,\omega^c(q^{-4}\varsigma)=-\eta_+\,\omega^c(\varsigma\diamondsuit z^2),$$ $$\omega^c(\varsigma)\,\eta_-=\eta_3\,\eta_-=-q^4\,\eta_3\,\eta_-=-\eta_-\,\omega^c(q^4\varsigma)=-\eta_+\,\omega^c(\varsigma\diamondsuit z^{\ast 2}),$$ $$\omega^c(\varsigma)\,\eta_-\eta_+=\eta_3\,\eta_-\,\eta_+=-q^4\,\eta_-\eta_3\,\eta_+=\eta_-\,\eta_+\,\eta_3=\eta_-\,\eta_+\,\omega^c(\varsigma)=\eta_-\,\eta_+\,\omega^c(\varsigma\diamondsuit \mathbbm{1})$$ and by linearity and the fact that $\Delta_{\Hor}$ is graded multiplicative, we conclude that $$\omega^c(\theta)\varphi=(-1)^k\,\varphi^{(0)}\omega^c(\theta\diamondsuit \varphi^{(1)}),$$ i.e., $\omega^c$ is regular. 

On the other hand, notice that $$dp=p^{(1)}\pi(p^{(2)})=p^{(1)}(\pi_-(p^{(2)})+\pi_+(p^{(2)})+\pi_3(p^{(2)}))$$  where $\Delta(p)=p^{(1)}\otimes p^{(2)}$, $\pi_\pm:=\rho_\pm\circ \pi,$ $\pi_3:=\rho_3\circ \pi$, with $$\pi:P\longrightarrow \mathfrak{su}_q(1,1)^\ast$$ the quantum germs map and $$\rho_\pm:\mathfrak{su}_q(1,1)^\ast \longrightarrow \mathrm{span}_\C\{\eta_\pm\}, \qquad \rho_3:\mathfrak{su}_q(1,1)^\ast \longrightarrow \mathrm{span}_\C\{\eta_3\}$$ the canonical projections. Consider the graded differential $\ast$--algebra epimorphism $j$ of Proposition \ref{algo5}. Then, for every $\vartheta=w_3\eta_3+w_-\eta_-+w_+\eta_+$ $\in$ $\mathfrak{su}_q(1,1)^\ast$ with $w_\pm, w_3$ $\in$ $\C$, we obtain $$\rho_3(\vartheta)=w_3\,\eta_3=w_3\,\omega^c(\varsigma)=\omega^c(w_3\,\varsigma)=\omega^c(j(\vartheta)).$$ Thus $$\rho_3=\omega^c\circ j.$$ 

The covariant derivative of $\omega^c$ is the operator $D$ such that for all $p$ $\in$ $P$  $$D(p)=dp-p^{(0)}\omega^c(\pi'(p^{(1)}))$$ with $\Delta_P(p)=p^{(0)}\otimes p^{(1)}$. By construction, $$\Delta_P(p)=p^{(1)}\otimes j(p^{(2)})$$ in the Sweedler notation for the coproduct $\Delta(p)=p^{(1)}\otimes p^{(2)}$. So, by equation (\ref{ec.necesito}) we have 
\begin{eqnarray*}
    D(p)&=&p^{(1)}\,\pi_-(p^{(2)})+p^{(1)}\,\pi_+(p^{(2)})+p^{(1)}\,\pi_3(p^{(2)})-p^{(1)}\,\omega^c(\pi'(j(p^{(2)}))) 
    \\
    &=&
    p^{(1)}\,\pi_-(p^{(2)})+p^{(1)}\,\pi_+(p^{(2)})+p^{(1)}\,\pi_3(p^{(2)})-p^{(1)}\,\omega^c(j(\pi(p^{(2)}))) 
    \\
    &=&
    p^{(1)}\,\pi_-(p^{(2)})+p^{(1)}\,\pi_+(p^{(2)})+p^{(1)}\,\pi_3(p^{(2)})-p^{(1)}\,\rho_3(\pi(p^{(2)}))
    \\
    &=&
    p^{(1)}\,\pi_-(p^{(2)})+p^{(1)}\,\pi_+(p^{(2)})+p^{(1)}\,\pi_3(p^{(2)})-p^{(1)}\,\pi_3(p^{(2)})
    \\
    &=&
    p^{(1)}\,\pi_-(p^{(2)})+p^{(1)}\,\pi_+(p^{(2)}).
\end{eqnarray*}
for every $p$ $\in$ $P$.  Also, by equation (\ref{ec.2.34.3}) we get
\begin{eqnarray*}
    D(\eta_+)=d\eta_++\eta_+\omega^c(\pi'(z^2))=q^{2}\eta_3\eta_-+q^{-2}\eta_+\omega^c(\varsigma) &=& q^{2}\eta_3\eta_-+q^{-2}\eta_+\eta_3
    \\
    &=&
    -q^{-2}\eta_+\eta_3+q^{-2}\eta_+\eta_3
    \\
    &=&
    0.
\end{eqnarray*}

Since $\omega^c$ is regular, we know that (\cite{micho2}) $D\circ \ast=\ast\circ D$ and hence (see Proposition \ref{algo1}) $$ 0=D(\eta_+)^\ast=D(\eta^\ast_+)=D(\eta_-).$$ In addition, the regularity property of $\omega^c$ implies that $D$ satisfies the graded Leibniz rule (\cite{micho2}), and this completely characterized the operator $D$. 

Finally, in light of Proposition \ref{algo2} we obtain $$R^{\omega^c}(\varsigma)=d\omega^c(\varsigma)=d\eta_3=-(1+q^2)\eta_-\,\eta_+.$$
\end{proof}

It is worth mentioning that, according to the general theory presented in \cite{micho3}, since $\omega^c$ is regular, the covariant derivative and the dual covariant derivative of $\omega$ are the same operator. As a corollary of the previous proposition and Proposition \ref{algo1}, we have

\begin{Corollary}
    \label{coro1}
    $$D(\alpha)= q\,\gamma^\ast\,\eta_+,\quad D(\alpha^\ast)=\gamma\,\eta_-,\quad D(\gamma)=\alpha^\ast\,\eta_+,\quad D(\gamma^\ast)=q^{-1}\,\alpha\,\eta_-.$$ In addition, by the Leibniz rule 
    \begin{equation*}
        D(\alpha^n)= q\left(\sum^{n-1}_{k=0}q^{-2k}\right)\,\alpha^{n-1}\gamma^{\ast}\eta_+=q^{3-2n}\,[n]\,\alpha^{n-1}\gamma^{\ast}\eta_+, 
    \end{equation*}
    \begin{equation*}
        D(\alpha^{\ast n})= q^{-1}\left(\sum^{n-1}_{k=0}q^{2k+1}\right)\,\alpha^{\ast\, n-1}\gamma\,\eta_-=[n]\,\alpha^{\ast n-1}\gamma\,\eta_-, 
    \end{equation*}
    \begin{equation*}
        D(\gamma^{n})= \left(\sum^{n-1}_{k=0}q^{-2k}\right)\,\gamma^{n-1}\alpha^\ast\,\eta_+=q^{2-2n}\,[n]\,\gamma^{n-1}\alpha^\ast\,\eta_+= q^{1-n}\,[n]\,\alpha^\ast\gamma^{n-1}\,\eta_+,
    \end{equation*}
    \begin{equation*}
        D(\gamma^{\ast n})= q^{-1}\left(\sum^{n-1}_{k=0}q^{2k}\right)\,\gamma^{\ast\, n-1}\alpha\,\eta_-=q^{-1}\,[n]\,\gamma^{\ast\, n-1}\alpha\,\eta_-= q^{-n}\,[n]\,\alpha\,\gamma^{\ast\,n-1}\,\eta_-
    \end{equation*}
    for every $n$ $\in$ $\N$, where
   $$[n]:=[n]_{q^2}:={q^{2n}-1\over q^2-1 }.$$
\end{Corollary}

Since $\ad'(\theta)=\theta\otimes \mathbbm{1}$ for all $\theta$ $\in$ $\mathfrak{u}_q(1)^\#$ (see equation (\ref{ec.2.37})), it follows that every $\lambda$ $\in$ $\overrightarrow{\mathfrak{qpc}(\zeta_{H})}$ satisfies $$\Im(\lambda)\subseteq \Omega^1(\H^2_q).$$ Furthermore, since $\mathfrak{u}_q(1)^\#=\mathrm{span}_\C\{\varsigma \}$, every $\lambda$ $\in$ $\overrightarrow{\mathfrak{qpc}(\zeta_{H})}$ is completely determined by $$\lambda(\varsigma)=\mu\;\in \; \Omega^1(\H^2_q)\quad \mbox{ with }\quad \mu^\ast=-\mu.$$ Hence, every qpc $\omega$ is of the form 
\begin{equation}
    \label{ec.3.8}
    \omega=\omega^c+\lambda \qquad \mbox{ with }\qquad \omega(\varsigma)=\eta_3+\mu.
\end{equation}
The covariant derivative of $\omega$ in terms of the previous equation is
\begin{eqnarray}
    \label{ec.3.9}
    D^\omega(\varphi)&=&d(\varphi)-(-1)^k\varphi^{(0)}\omega(\pi'(\varphi^{(1)}))\nonumber
    \\
    &=& 
    d(\varphi)-(-1)^k\varphi^{(0)}\omega^c(\pi'(\varphi^{(1)}))-(-1)^k\varphi^{(0)}\lambda(\pi'(\varphi^{(1)}))
    \\
    &=&
    D(\varphi)-(-1)^k\varphi^{(0)}\lambda(\pi'(\varphi^{(1)}))\nonumber
\end{eqnarray}
and the dual covariant derivative is
\begin{eqnarray}
    \label{ec.3.10}
    \widehat{D}^\omega(\varphi)&=&d(\varphi)+\omega(\pi'(S^{-1}(\varphi^{(1)})))\varphi^{(0)}\nonumber
    \\
    &=& 
    d(\varphi)+\omega^c(\pi'(S^{-1}(\varphi^{(1)})))\varphi^{(0)}+\lambda(\pi'(S^{-1}(\varphi^{(1)})))\varphi^{(0)}
    \\
    &=&
    D(\varphi)+\lambda(\pi'(S^{-1}(\varphi^{(1)})))\varphi^{(0)}\nonumber
\end{eqnarray}
for $\varphi$ $\in$ $\Hor^k\,P$. Additionally, the curvature of $\omega$ is given by 
\begin{eqnarray}
\label{ec.3.12}
    R^\omega(\varsigma)=R^{\omega^c+\lambda}(\varsigma)=d\omega^c(\varsigma)+d\lambda(\varsigma)&=&d\eta_3+d\mu\nonumber
    \\
    &=&
    -(1+q^2)\,\eta_-\eta_++d\mu.\nonumber
\end{eqnarray}

To finalize this subsection, we have 

\begin{Proposition}
\label{3.6}
The only regular qpc is $\omega^\c$.
\end{Proposition}

\begin{proof}
Let $\omega$ be a regular qpc. Since $\omega$ is of the form
\begin{equation*}
\omega^\c+\lambda,
\end{equation*}
 the regularity condition implies that
\begin{eqnarray}
    \label{ec.3.11}
    \lambda(\theta)\,\varphi= (-1)^k \varphi^{(0)} \lambda(\theta\diamondsuit \varphi^{(1)})
\end{eqnarray}
for all $\varphi$ $\in$ $\Hor^{k} \,P$, $\theta$ $\in$ $\mathfrak{u}_q(1)^\#$.

We claim  $\lambda=0$. In fact, by equation (\ref{ec.2.37}), it follows that  $\Im(\lambda)\subseteq \Omega^1(\H^2_q)$; and in light of Proposition \ref{2.2} $$\lambda(\varsigma)=x\eta_-+y\eta_+$$
where $$x=\displaystyle \sum_{m+k-l=2}\lambda_{mkl}\,\alpha^m\gamma^k\gamma^{\ast l},\qquad y=\displaystyle\sum_{r+s-t=-2}\,\lambda'_{rst}\,\alpha^r\gamma^s\gamma^{\ast t},$$ for $m,r$ $\in$ $\Z$, $k,l,s,t$ $\in$ $\N$, $\lambda_{mkl}$  $\lambda'_{rst}$ $\in$ $\C$ such that $m+k-l=2$ and $r+s-t=-2$. By equation (\ref{ec.2.35.1}) we know that $$\varsigma\diamondsuit z=q^{-2}\varsigma$$ and  hence $$\lambda(\varsigma)\,\alpha=\alpha\,\lambda(\varsigma\diamondsuit z)=q^{-2}\,\alpha\,\lambda(\varsigma).$$ However, by equation (\ref{ec.2.1}) and Proposition \ref{algo1} we obtain
\begin{eqnarray*}
    \lambda(\varsigma)\,\alpha=(x\eta_-+y\eta_+)\alpha=q^{-1}\,x\,\alpha\,\eta_-+q^{-1}\,y\,\alpha\,\eta_+
    &=&
    q^{-1}\,\displaystyle \sum_{m+k-l=2}\lambda_{mkl}\,\alpha^m\gamma^k\gamma^{\ast l}\,\alpha\,\eta_-
    \\
    &+&
    q^{-1}\,\sum_{r+s-t=-2}\,\lambda'_{rst}\,\alpha^r\gamma^s\gamma^{\ast t}\,\alpha\,\eta_+
    \\
    &=&
    q^{-(k+l+1)}\,\displaystyle \sum_{m+k-l=2}\lambda_{mkl}\,\alpha^m\,\alpha\,\gamma^k\gamma^{\ast l}\,\eta_-
    \\
    &+&
    q^{-(s+t+1)}\,\sum_{r+s-t=-2}\,\lambda'_{rst}\,\alpha^r\,\alpha\,\gamma^s\gamma^{\ast t}\,\eta_+.
\end{eqnarray*}
We conclude that $k+l=1=s+t$ and $r,\,m\geq 0$.

On the other hand $$\lambda(\varsigma)\,\gamma= \gamma \,\lambda(\varsigma\diamondsuit z) = q^{-2}\,\gamma\, \lambda(\varsigma).$$  However, by equation (\ref{ec.2.1}) and Proposition \ref{algo1} we have
\begin{eqnarray*}
    \lambda(\varsigma)\,\gamma=(x\eta_-+y\eta_+)\gamma
    =
    q^{-1}\,x\,\gamma\,\eta_-+q^{-1}\,y\,\gamma\,\eta_+
    &=&
    q^{-1}\,\displaystyle \sum_{m+k-l=2}\lambda_{mkl}\,\alpha^m\gamma^k\gamma^{\ast l}\,\gamma\,\eta_-
    \\
    &+&
    q^{-1}\,\sum_{r+s-t=-2}\,\lambda'_{rst}\,\alpha^r\gamma^s\gamma^{\ast t}\,\gamma\,\eta_+
    \\
    &=&
    q^{-1}\,\displaystyle \sum_{m+k-l=2}\lambda_{mkl}\,\alpha^m\,\gamma\,\gamma^k\gamma^{\ast l}\,\eta_-
    \\
    &+&
    q^{-1}\,\sum_{r+s-t=-2}\,\lambda'_{rst}\,\alpha^r\,\gamma\,\gamma^s\gamma^{\ast t}\,\eta_+
    \\
    &=&
    q^{-1+m}\,\displaystyle \sum_{m+k-l=2}\lambda_{mkl}\,\gamma\,\alpha^m\,\gamma^k\gamma^{\ast l}\,\eta_-
    \\
    &+&
    q^{-1+r}\,\sum_{r+s-t=-2}\,\lambda'_{rst}\,\gamma\,\alpha^r\,\gamma^s\gamma^{\ast t}\,\eta_+.
\end{eqnarray*}
Hence $-1+m=-2=-1+r$; but $m, \,r\geq 0$, which is a contraction. This implies that $x=y=0$ and therefore, $\lambda=0$.
\end{proof}

In light of \cite{micho2,micho3}, the last proposition tells us that the only covariant derivative $D^\omega$ such that
\begin{enumerate}
    \item Satisfies the graded Leibniz rule.
    \item Commutes with the $\ast$ operation.
    \item Matches with its dual covariant derivative.
\end{enumerate}
is $D$, the one associated with the canonical qpc $\omega^c$.

\subsection{Associated Quantum Vector Bundles, Gauge Quantum Linear Connections and The Quantum Gauge Group}

This subsection will be based on  \cite{sald1}. Consider the {\it classical} principal $U(1)$--bundle 
\begin{equation}
    \label{class1}
    \rho: SU(1,1)\longrightarrow \H^2,
\end{equation}
where $\rho$ is the canonical projection of $SU(1,1)$ in the homogeneous space $\H^2=\displaystyle{SU(1,1)\over U(1)}$. For every unitary finite--dimensional $U(1)$--representation $$\delta_V: U(1)\longrightarrow GL(V),$$ we can consider the associated vector bundle \cite{nodg} 
\begin{equation}
    \label{class2}
    \pi_V: V\H^2\longrightarrow \H^2,
\end{equation}
where $$V\H^2:=SU(1,1)\times_{\delta_V} V:=(SU(1,1)\times V)/U(1).$$ Furthermore, according to the Serre--Swan theorem, $$\pi_V:V\H^2\longrightarrow \H^2$$ is equivalent to its space of smooth sections  $\Gamma(V\H^2)$, which is a finitely generated projective $C^\infty(\H^2)$--bimodule. Moreover, $\Gamma(V\H^2)$ is isomorphic to the space of $U(1)$--equivariant maps (\cite{nodg})
\begin{equation}
\label{class3}
    \begin{aligned}
        C^\infty(SU(1,1),V)^{U(1)}=\{f:SU(1,1)\longrightarrow V\mid f \mbox{ is smooth and } \\ f(xA)=\delta_V(A^{-1})f(x) \mbox{ for all }x\in SU(1,1),\; A\in U(1) \}.
    \end{aligned}
\end{equation}

By \emph{dualizing} with the pull--back and identifying $V$ with its dual space $V^\ast$, one gets that the non--commutative geometrical counterpart of $C^\infty(SU(1,1),V)^{U(1)}$ is the space 
\begin{equation}
\label{maps0}
    \Mor(\delta^V,\Delta_P):=\{ T: V\longrightarrow P\mid T \mbox{ is a linear map such that } \Delta_P\,\circ T=(T\otimes \id_G)\circ \delta^V\},
\end{equation}
where $\delta^V$ is a finite--dimensional (right) $G$--corepresentation coacting in $V$. So, in light of the Serre--Swan theorem, in order to consider $\Mor(\delta^V,\Delta_P)$ as a quantum vector bundle, we should prove that $\Mor(\delta^V,\Delta_P)$ is a finitely generated projective left or right $\H^2_q$--module.

First, notice that there is a canonical left $\H^2_q$--module structure on $\Mor(\delta^V,\Delta_P)$ given by 
\begin{equation}
    \label{leftstruc}
    \cdot: \H^2_q\times \Mor(\delta^V,\Delta_P)\longrightarrow \Mor(\delta^V,\Delta_P), \qquad (b,T)\longmapsto b\,T,
\end{equation}
where $(b\,T)(v):=b \,T(v)$ for all $v$ $\in$ $V$; and there is a canonical right $\H^2_q$--module structure on $\Mor(\delta^V,\Delta_P)$ given by
\begin{equation}
    \label{rightstruc}
    \cdot: \Mor(\delta^V,\Delta_P)\times \H^2_q\longrightarrow \Mor(\delta^V,\Delta_P), \qquad (T,b)\longmapsto T\,b,
\end{equation}
where $(T\,b)(v):= T(v)\,b$ for all $v$ $\in$ $V$.

For the canonical quantum group $\U(1)$ associated with the Lie group $U(1)$, it is well--known that a complete set of mutually inequivalent irreducible (unitary) finite--dimensional corepresentations is given by
\begin{equation}
    \label{ec.3.15}
    \T=\{\delta^n\mid n\in \Z \},
\end{equation}
where 
\begin{equation}
    \label{ec.3.16}
    \delta^n:\C\longrightarrow \C\otimes G,\qquad w\longmapsto w\otimes z^n.
\end{equation}

\begin{Proposition}
    \label{propqvb}
    The space $\Mor(\delta^n,\Delta_P)$ is a free left/right $\H^2_q$--module of rank $n+1$.
\end{Proposition}

\begin{proof}
    Let $T$ $\in$ $\Mor(\delta^n,\Delta_P)$. For $n=0$, consider 
    \begin{equation}
        \label{nec-1}
         T^0_0:\C\longrightarrow P,\qquad w\longmapsto w\,\mathbbm{1}.
    \end{equation}
    Then $T^0_0$ $\in$ $\Mor(\delta^0,\Delta_P)$ and we obtain  
    \begin{equation}
        \label{nec0}
        T=b^{_T}\,T^0_0  \qquad \mbox{ with }\qquad b^{_T}=T(1)\,\in\, \H^2_q.
    \end{equation}
Assume $n$ $\in$ $\N$ and consider the linear map
\begin{equation}
    \label{leftgen}
    T^j_n: \C\longrightarrow P,\qquad w\longmapsto w\,\alpha^{n-j}\gamma^{j} 
\end{equation}
for $j=0,...,n$. Since $$\Delta_P(T^j_n(1))=\Delta_P(\alpha^{n-j}\,\gamma^j)=\alpha^{n-1}\,\gamma^j\otimes z^n,$$ we get $T^j_n$ $\in$ $\Mor(\delta^n,\Delta_P)$. Notice that 
\begin{eqnarray*}
    T^0_n(1)^\ast\, T^0_n(1) &=& \alpha^{\ast\, n}\alpha^n=\mathbbm{1}+f^n_1\,\gamma\,\gamma^\ast+f^n_2\,\gamma^2\,\gamma^{\ast 2}+\cdots +f^n_n\,\gamma^n\,\gamma^{\ast n}
    \\
    T^1_n(1)^\ast\, T^1_n(1)&=&\gamma^\ast \alpha^{\ast n-1}\alpha^{n-1}\gamma=\gamma \gamma^\ast+f^{n-1}_1\,\gamma^2\,\gamma^{\ast 2}+\cdots +f^{n-1}_{n-1}\gamma^n\,\gamma^{\ast n}
    \\
    T^2_n(1)^\ast\, T^2_n(1)&=& \gamma^{\ast 2}\alpha^{\ast n-2}\alpha^{n-2}\gamma^2=\gamma^2\gamma^{\ast 2}+f^{n-2}_1\gamma^3\gamma^{\ast 3}+\cdots +f^{n-2}_2\gamma^n\gamma^{\ast n}
    \\
    &\vdots&
    \\
    T^{n-1}_n(1)^\ast \,T^{n-1}_n(1)&=& \gamma^{\ast n-1}\,\alpha^\ast\,\alpha\,\gamma^{n-1}=\gamma^{\ast n-1}\,\gamma^{n-1}+f^1_1\,\gamma^n\,\gamma^{\ast n}
    \\
   T^n_n(1)^\ast\, T^n_n(1) &=&\gamma^{\ast n}\gamma^n
\end{eqnarray*}
for some polynomials $f^n_1$,..., $f^n_n$, $f^{n-1}_1$,..., $f^{n-1}_{n-1}$, $f^{n-2}_1$, ..., $f^{n-2}_{n-2}$,..., $f^1_1$ in $q$. This form an upper triangular matrix 
$$
\left(\begin{matrix}
1 & f^n_1 &f^n_2 &\cdots & f^n_{n-1} & f^n_n \\
0 & 1 &f^{n-1}_1 & \cdots &f^{n-1}_{n-2} &f^{n-1}_{n-1} \\
\vdots & \vdots & \vdots & \cdots & \vdots & \vdots\\
0 &0 & 0& \cdots & 1 & f^1_1\\
0 & 0 & 0 & \cdots &0 & 1
\end{matrix}\right).$$ Let us index the columns and rows of this matrix from $0$ to $n$. Now, we can multiply the $j$--th row, for $j = 1,..., n$, by a scalar coefficient $r^n_j$ in such a way that the entries in each column sum to zero, except for the $0$--th column, whose entries sum to $1$. In other words, consider the matrix
$$
\left(\begin{matrix}
1 & f^n_1 &f^n_2 &\cdots & f^n_{n-1} & f^n_n \\
0 & r^n_1 &r^n_1\,f^{n-1}_1 & \cdots & r^n_1\,f^{n-1}_{n-2} & r^n_1\,f^{n-1}_{n-1} \\
0 & 0 & r^n_2  & \cdots & r^n_2\,f^{n-2}_{n-3} & r^n_2\,f^{n-2}_{n-2} \\
\vdots & \vdots & \vdots & \cdots &  \vdots & \vdots \\
0 & 0 & 0 & \cdots &  r^n_{n-1} & r^n_{n-1}\,f^1_1\\
0 & 0 & 0 & \cdots &  0 & r^n_n
\end{matrix}\right).$$
such that $$f^n_1 + r^n_1=0,$$ $$f^n_2+r^n_1\,f^{n-1}_1+r^n_2=0, $$ $$\vdots $$ $$f^n_{n-1}+r^n_1\,f^{n-1}_{n-2}+r^n_2\,f^{n-2}_{n-3}+\cdots r^n_{n-2}\,f^{2}_{1} +r^n_{n-1}=0,$$ $$f^n_n+r^n_1\,f^{n-1}_{n-1}+r^n_2\,f^{n-2}_{n-2}+\cdots +r^n_{n-1}\,f^1_1+r^n_n=0.$$ Notice that this system of $n$--equations with $n$--variables has a unique solution. 

Define 
\begin{equation}
    \label{coe1}
    x_{n0}:=T^0_n(1)^\ast= \alpha^{\ast n}, \qquad  x_{nk}:=r^n_{k} \,T^k_n(1)^\ast=r^n_{k} \,\gamma^{\ast k}\,\alpha^{\ast n-k}
\end{equation}
for $k=1,...,n$. Thus
$$\Delta_P(T(1)\,x_{ni})=\Delta_P(T(1))\,\Delta_P(x_{ni})=(T(1)\otimes z^n)(x_{ni}\otimes z^{\ast n})=T(1)\,x_{ni}\otimes \mathbbm{1}$$ for all $i=0,...,n$; which implies that $$b^{_T}_i:=T(1)\,x_{ni}\;\in\; \H^2_q.$$ We claim that
\begin{equation}
    \label{coe2}
   T=b^{_T}_0\,T^0_n+\cdots+ b^{_T}_n\,T^n_n.
\end{equation}
Indeed, 
\begin{eqnarray*}
    b^{_T}_0\,T^0_n(1)+\cdots+ b^{_T}_n\,T^n_n(1)&=&T(1)\,x_{n0}\,T^0_n(1)+\cdots+T(1)\,x_{nn}\,T^n_n(1)
    \\
    &=&
    T(1)\,(x_{n0}\,T^0_n(1)+\cdots+\,x_{nn}\,T^n_n(1))
    \\
    &=&
    T(1)\,(T^0_n(1)^\ast\,T^0_n(1)+ r^n_1\, T^1_n(1)^\ast\,T^1_n(1)+\cdots+r^n_{n}\,T^n_n(1)^\ast\,T^n_n(1))
    \\
    &=&
     T(1)\,\mathbbm{1}
     \\
    &=&
    T(1).
\end{eqnarray*}
and by linearity we conclude that equation (\ref{coe2}) is satisfied.

Since the set $\{ \alpha^n,\alpha^{n-1}\,\gamma,..., \gamma^n\}$ is linear independent and $b^{_T}_0,\cdots b^{_T}_n$ can also be expressed in terms of the basis $\beta_B$ (see equation (\ref{ec.2.15})), it follows that the elements $b^{_T}_0,\cdots b^{_T}_n$ are unique. Therefore  $\{ T^j_n\}^n_{j=0}$ is left $\H^2_q$--basis of $\Mor(\delta^n,\Delta_P)$. 
 
 Assume now $n=-m$ with $m$ $\in$ $\N$. Following the same strategy as above, we can conclude that the maps 
\begin{equation}
    \label{rightgen}
    T^j_{-m}: \C\longrightarrow P,\qquad w\longmapsto w\,\alpha^{-m+j}\gamma^{\ast j} 
\end{equation}
 for $j=0,...,m$ form a left  $\H^2_q$--basis of $\Mor(\delta^{n},\Delta_P)$.
 
The next step is to find a right  $\H^2_q$--basis. Let $T$ $\in$ $\Mor(\delta^n,\Delta_P)$. For $n=0$, we have 
$$T=T^0_0\, b^{_T} \qquad \mbox{ with }\qquad b^{_T}=T(1)\,\in\, \H^2_q.$$
Assume $n$ $\in$ $\N$. Notice that
\begin{eqnarray*}
    T^0_n(1)\, T^0_n(1)^\ast &=& \alpha^n\alpha^{\ast\, n}=\mathbbm{1}+h^n_1\,\gamma\,\gamma^\ast+h^n_2\,\gamma^2\,\gamma^{\ast 2}+\cdots +h^n_n\,\gamma^n\,\gamma^{\ast n}
    \\
    T^1_n(1)\, T^1_n(1)^\ast&=&\alpha^{n-1}\gamma\,\gamma^\ast \alpha^{\ast n-1}=\gamma \gamma^\ast+h^{n-1}_1\,\gamma^2\,\gamma^{\ast 2}+\cdots +h^{n-1}_{n-1}\gamma^n\,\gamma^{\ast n}
    \\
    T^2_n(1)\, T^2_n(1)^\ast&=& \alpha^{n-2}\,\gamma^2\,\gamma^{\ast 2}\alpha^{\ast n-2}=\gamma^2\gamma^{\ast 2}+h^{n-2}_1\gamma^3\gamma^{\ast 3}+\cdots +h^{n-2}_2\gamma^n\gamma^{\ast n}
    \\
    &\vdots&
    \\
    T^{n-1}_n(1) \,T^{n-1}_n(1)^\ast&=& \alpha\,\gamma^{n-1}\,\gamma^{\ast n-1}\,\alpha^\ast=\gamma^{n-1}\,\gamma^{\ast n-1}+h^1_1\,\gamma^n\,\gamma^{\ast n}
    \\
   T^n_n(1)\, T^n_n(1)^\ast &=&\gamma^n\gamma^{\ast n}
\end{eqnarray*}
for some polynomials $h^n_1$,..., $h^n_n$, $h^{n-1}_1$,..., $h^{n-1}_{n-1}$, $h^{n-2}_1$, ..., $h^{n-2}_{n-2}$,..., $h^1_1$ in $q$. This form an upper triangular matrix 
$$
\left(\begin{matrix}
1 & h^n_1 &h^n_2 &\cdots & h^n_{n-1} & h^n_n \\
0 & 1 &h^{n-1}_1 & \cdots &h^{n-1}_{n-2} &h^{n-1}_{n-1} \\
\vdots & \vdots & \vdots & \cdots & \vdots & \vdots\\
0 &0 & 0& \cdots & 1 & h^1_1\\
0 & 0 & 0 & \cdots &0 & 1
\end{matrix}\right).$$ Let us index the columns and rows of this matrix from $0$ to $n$. Now, we can multiply the $j$--th row, for $j = 1,..., n$, by a scalar coefficient $s^n_j$ in such a way that the entries in each column sum to zero, except for the $0$--th column, whose entries sum to $1$. In other words, consider the matrix
$$
\left(\begin{matrix}
1 & h^n_1 &h^n_2 &\cdots & h^n_{n-1} & h^n_n \\
0 & s^n_1 &s^n_1\,h^{n-1}_1 & \cdots & s^n_1\,h^{n-1}_{n-2} & s^n_1\,h^{n-1}_{n-1} \\
0 & 0 & s^n_2  & \cdots & s^n_2\,h^{n-2}_{n-3} & s^n_2\,h^{n-2}_{n-2} \\
\vdots & \vdots & \vdots & \cdots &  \vdots & \vdots \\
0 & 0 & 0 & \cdots &  s^n_{n-1} & s^n_{n-1}\,h^1_1\\
0 & 0 & 0 & \cdots &  0 & s^n_n
\end{matrix}\right).$$
such that $$h^n_1 + s^n_1=0,$$ $$h^n_2+s^n_1\,h^{n-1}_1+s^n_2=0, $$ $$\vdots $$ $$h^n_{n-1}+s^n_1\,h^{n-1}_{n-2}+s^n_2\,h^{n-2}_{n-3}+\cdots s^n_{n-2}\,h^{2}_{1} +s^n_{n-1}=0,$$ $$h^n_n+s^n_1\,h^{n-1}_{n-1}+s^n_2\,h^{n-2}_{n-2}+\cdots +s^n_{n-1}\,h^1_1+s^n_n=0.$$ Notice that this system of $n$--equations with $n$--variables has a unique solution. 

Define 
\begin{equation}
    \label{coe2.5}
    y_{n0}:=T^0_n(1)^\ast=\alpha^{\ast n}, \qquad y_{nj}:=s^r_k \,T^k_n(1)^\ast=s^n_k\,\gamma^{\ast k}\,\alpha^{\ast n-k}
\end{equation}
for $k=1,...,n$. Thus
$$\Delta_P(y_{ni}\,T(1))=\Delta_P(y_{ni})\,\Delta_P(T(1))=(y_{ni}\otimes z^{\ast n})\,(T(1)\otimes z^n)=y_{ni}\,T(1)\otimes \mathbbm{1}$$ for all $i=0,...,n$; which implies that 
\begin{equation}
    \label{coe3}
    \hat{b}^{_T}_i:=y_{ni}\,T(1) \;\in\; \H^2_q.
\end{equation}
We claim that
\begin{equation}
    \label{coe2.6}
    T=T^0_n\,\hat{b}^{_T}_0+\cdots +T^n_n\,\hat{b}^{_T}_n.
\end{equation}
Indeed, 
\begin{eqnarray*}
   T^0_n(1)\,\hat{b}^{_T}_0+\cdots +T^n_n(1)\,\hat{b}^{_T}_n&=& T^0_n(1)\,y_{n0}\,T(1)+\cdots +T^n_n(1)\,y_{nn}\,T(1)
    \\
    &=&
    (T^0_n(1)\,y_{n0}+\cdots+\,T^n_n(1)\,y_{nn})\,T(1)
    \\
    &=&
    (T^0_n(1)\,T^0_n(1)^\ast+ s^n_1\, T^1_n(1)\,T^1_n(1)^\ast+\cdots+s^n_{n}\,T^n_n(1)\,T^n_n(1)^\ast)\,T(1)
    \\
    &=&
     \mathbbm{1}\,T(1)
     \\
    &=&
    T(1).
\end{eqnarray*}
and by linearity we conclude that equation (\ref{coe2.6}) is satisfied.

Since the set $\{ \alpha^n,\alpha^{n-1}\,\gamma,..., \gamma^n\}$ is linear independent and $\hat{b}^{_T}_0,\cdots ,\hat{b}^{_T}_n$ can also be expressed in terms of the basis $\beta_B$, it follows that the elements $\hat{b}^{_T}_0,\cdots \hat{b}^{_T}_n$ are unique. Therefore  $\{ T^j_n\}^n_{j=0}$ is right $\H^2_q$--basis of $\Mor(\delta^n,\Delta_P)$. 

Assume now $n=-m$ with $m$ $\in$ $\N$. Following the same strategy as above, we can conclude that the maps 
\begin{equation}
    \label{rightgen1}
    T^j_{-m}: \C\longrightarrow P,\qquad w\longmapsto w\,\alpha^{-m+j}\gamma^{\ast j} 
\end{equation}
 for $j=0,...,n$ form a right $\H^2_q$--basis of $\Mor(\delta^n,\Delta_P)$.
\end{proof}

Let $\delta^V$ be a finite--dimensional (unitary) $G$--corepresentation. According to \cite{woro1}, we have $$\delta^V\cong \bigoplus_{n_i}\delta^{n_i} $$ for a finite number of $\delta^{n_i}$ $\in$ $\T$ and hence $$\Mor(\delta^V,\Delta_P)\cong \bigoplus_{n_i} \Mor(\delta^{n_i},\Delta_P)$$ is also a free left/right  $\H^2_q$--module of finite rank. 

\begin{Definition}
    \label{qvb}
    Let $\delta^V$ be a finite--dimensional (unitary) $G$--corepresentation. We define the the associated left quantum vector bundle (abbreviated qvb) of $\zeta_H$ with respect to $\delta^V$ as the free left $\H^2_q$--module
    $$E^V_\l:=\Mor(\delta^V,\Delta_P).$$ Moreover, we define the associated right quantum vector bundle of $\zeta_H$ with respect to $\delta^V$ as the free right $\H^2_q$--module
    $$E^V_\r:=\Mor(\delta^V,\Delta_P).$$ For an element $\delta^n$ of $\T$, we will use the notation $$E^n_\l,\qquad E^n_\r,$$ respectively.
\end{Definition}

In accordance with \cite{br2}, we have that $P\,\square^{H}\, V \cong E^V_\l$, where $\square^{H}$ denotes the cotensor product. This construction is the common one for associated qvb's.  Nevertheless, we have chosen to work with $E^V_\l$ and $E^V_\r$  instead of $P\,\square^{H}\, V$ because
\begin{enumerate}
    \item This is the standard definition in Durdevich's formulation of qpb's for associated left/right qvb's, as the reader can check in \cite{micho2}.
    \item The definitions of the gauge quantum linear connections and their exterior covariant derivatives become completely analogous to their {\it classical} counterparts (see Remark \ref{clasico} below).
    \item They are easier to work with, as they permit explicit calculations, as the reader will see in the rest of this subsection and the next section. 
\end{enumerate}

Now, let us return to the \emph{classical} case considering the principal $U(1)$--bundle of equation (\ref{class1}) and its associated vector bundle with respect to the linear $U(1)$--representation $\delta_V$ (see equation (\ref{class2})). Let us take the space of basic differential forms of type $\delta_V$ (\cite{nodg})
\begin{equation}
    \label{class4}
    \Omega^\bullet_\mathrm{Hor}(SU(1,1),V)^{U(1)}=\bigoplus_{k}\Omega^k_\mathrm{Hor}(SU(1,1),V)^{U(1)},
\end{equation}
where 
\begin{equation*}
    \begin{aligned}
        \Omega^k_\mathrm{Hor}(SU(1,1),V)^{U(1)}=\{\eta \in \Omega^k(SU(1,1),V)\mid &\;\eta \mbox{ is horizontal}\\ & \mbox{and }r^\#_A(\eta)=\delta_V(A^{-1})\eta \mbox{ for every } A\in U(1)\}
    \end{aligned}
\end{equation*}
with $\Omega^k(SU(1,1),V)$  the space of $V$--valued differential $k$--forms of $SU(1,1)$,  $$r_A:SU(1,1)\longrightarrow SU(1,1),\qquad x\longmapsto xA$$ for a fixed element $A$ of $U(1)$, and $r^\#_A(\eta)$ denotes the pull--back of $\eta$ with $r_A$.

As in the begging of this subsection, by \emph{dualizing} with the pull--back and identifying $V$ with its dual space $V^\ast$, one gets that the non--commutative geometrical counterpart of $\Omega^\bullet_\mathrm{Hor}(SU(1,1),V)^{U(1)}$ is the space
\begin{equation}
    \label{maps1}
\Mor(\delta^V,\Delta_\Hor)=\bigoplus_k\Mor^k(\delta^V,\Delta_\Hor),
\end{equation}
where
\begin{equation*}
\begin{aligned}
    \Mor^k(\delta^V,\Delta_\Hor):=\{ \tau: \C\longrightarrow \Hor^k\,P\mid \tau \mbox{ is a}& \mbox{ linear map} \\ 
    & \mbox{ such that } \Delta_\Hor\circ \tau=(\tau\otimes \id_G)\circ \delta^n\}
\end{aligned}
\end{equation*}
with $\delta^V$ a finite--dimensional (right) $G$--corepresentation coaction in $V$.

In particular, notice that (\cite{nodg}) $$\Omega^0_\mathrm{Hor}(SU(1,1),V)^{U(1)}=C^\infty(SU(1,1),V)^{U(1)} $$ ans since $\Delta_\Hor|_P=\Delta_P$, we get $$\Mor^0(\delta^V,\Delta_\Hor)=\Mor(\delta^V,\Delta_P).$$

The space $\Mor(\delta^V,\Delta_\Hor)$ will be interpreted as the space of \emph{quantum differential forms of type $\delta^V$}. In particular, this is why, for example, the equation (\ref{ec.3.5})
is the dualization of the \emph{classical} fact that in differential geometry, the curvature of a principal connection is a basic differential $2$--form of type adjoint 

In differential geometry, the Gauge Principle establishes that $\Omega^\bullet_\mathrm{Hor}(SU(1,1),V)^{U(1)}$ is isomorphic to the space of $\Gamma(V\H^2)$--valued differential forms of $\H^2$ \cite{nodg}. In other words, there exists an isomorphism 
\begin{equation}
    \label{class5}
    GP: \Omega^\bullet_\mathrm{Hor}(SU(1,1),V)^{U(1)}\longrightarrow \Omega^\bullet(\H^2)\otimes_{C^\infty(\H^2)}\Gamma(V\H^2).
\end{equation}
In particular (\cite{nodg}) $$GP^{-1}(\mu\otimes_{C^\infty(\H^2)} s)=\mu \,s$$ for $\mu$ $\in$ $\Omega^\bullet(\H^2)$, $s$ $\in$ $\Gamma(V\H^2)$.

Let us show that the Gauge Principle is also satisfied in our case. First, notice that we can equip $\Mor(\delta^V,\Delta_\Hor)$ with a canonical left/right $\Omega^\bullet(\H^2_q)$--module structure similar to the ones presented in equations (\ref{leftstruc}), (\ref{rightstruc}). Now, consider $\delta^n$ $\in$ $\T$ and $\tau$ $\in$ $\Mor(\delta^n,\Delta_\Hor)$. Since  $$\Delta_\Hor(\tau(1))=\tau(1)\otimes z^n,$$ we can follow the proof of Proposition \ref{propqvb} to get that 
\begin{equation}
    \label{coe4}
    \tau=\mu^\tau_0\,T^0_n+\mu^\tau_1\,T^1_n+\cdots +\mu^\tau_{n-1}\,T^{n-1}_n+\mu^\tau_n\,T^n_n \quad \mbox{ where }\quad   \mu^\tau_i:=\tau(1)\,x_{ni} \,\in\, \Omega^\bullet(\H^2_q)
\end{equation}
if $n$ $\in$ $\N$,  
\begin{equation}
    \label{coe5}
    \tau=\mu^\tau_0\,T^0_{-m}+\mu^\tau_1\,T^1_{-m}+\cdots +\mu^\tau_{m-1}\,T^{m-1}_{-m}+\mu^\tau_m\,T^m_{-m} \quad \mbox{ where }\quad   \mu^\tau_i:=\tau(1)\,x_{-mi} \,\in\, \Omega^\bullet(\H^2_q)
\end{equation}
if $n=-m$ with $m$ $\in$ $\N$, and 
\begin{equation}
    \label{coe6}
    \tau=\mu^\tau\,T^0_{0} \quad \mbox{ where }\quad   \mu^\tau:=\tau(1) \,\in\, \Omega^\bullet(\H^2_q)
\end{equation}
if $n=0$. Thus, a direct calculation shows that the left 
 $\Omega^\bullet(\H^2_q)$--module morphism
\begin{equation}
    \label{ec.3.23}
    \Upsilon_n: \Mor(\delta^n,\Delta_\Hor)\longrightarrow \Omega^\bullet(\H^2_q)\otimes_{\H^2_q}E^n_\l
\end{equation}
given by 
\begin{equation}
    \label{ec.3.24}
    \Upsilon_n(\tau)=\sum^{|n|}_{k=0} \mu^{\tau}_k\otimes_{\H^2_q} T^k_n 
\end{equation}
is an isomorphism with inverse function given by
\begin{equation}
    \label{ec.3.25}
    \Upsilon^{-1}_n(\mu\otimes_{\H^2_q}T)=\mu\,T.
\end{equation}

Similarly,
\begin{equation}
    \label{coe7}
    \tau=T^0_{n}\,\hat{\mu}^\tau_0+ T^1_{n}\,\hat{\mu}^\tau_1+\cdots +T^{n-1}_{n}\,\hat{\mu}^\tau_{n-1}+ T^n_{n}\,\hat{\mu}^\tau_n \quad \mbox{ where }\quad   \hat{\mu}^\tau_i:=y_{ni}\,\tau(1) \,\in\, \Omega^\bullet(\H^2_q),
\end{equation}
if $n$ $\in$ $\N$,
\begin{equation}
    \label{coe8}
    \tau=T^0_{-m}\,\hat{\mu}^\tau_0+ T^1_{-m}\,\hat{\mu}^\tau_1+\cdots +T^{m-1}_{-m}\,\hat{\mu}^\tau_{m-1}+ T^m_{-m}\,\hat{\mu}^\tau_n \quad \mbox{ where }\quad   \hat{\mu}^\tau_i:=y_{-mi}\,\tau(1) \,\in\, \Omega^\bullet(\H^2_q), 
\end{equation}
if $n=-m$ with $m$ $\in$ $\N$, and
\begin{equation}
    \label{coe9}
    \tau=T^0_{0}\, \hat{\mu}^\tau \quad \mbox{ where }\quad   \hat{\mu}^\tau:=\tau(1) \,\in\, \Omega^\bullet(\H^2_q)
\end{equation}
if $n=0$. Thus, a direct calculation shows that the right  $\Omega^\bullet(\H^2_q)$--module morphism 
\begin{equation}
    \label{ec.3.26}
    \widehat{\Upsilon}_n: \Mor(\delta^n,\Delta_\Hor)\longrightarrow E^n_\r\otimes_{\H^2_q}\Omega^\bullet(\H^2_q),
\end{equation}
 given by 
\begin{equation}
    \label{ec.3.27}
    \widehat{\Upsilon}_n(\tau)=\sum^{|n|}_{k} T^k_n\otimes_{\H^2_q} \hat{\mu}^\tau_k \quad \mbox{ with } \quad \tau=\sum^{|n|}_{k=1} T^k_n\, \hat{\mu}^\tau_k
\end{equation}
is an isomorphism with inverse function  given by
\begin{equation}
    \label{ec.3.28}
    \widehat{\Upsilon}^{-1}_n(T\otimes_{\H^2_q}\mu)=T\,\mu.
\end{equation}

Let $\delta^V$ be a finite--dimensional (unitary) $G$--corepresentation. As above, according to \cite{woro1}, we obtain $$\delta^V\cong\bigoplus_{n_i}\delta^{n_i}$$ for some finite number of $\delta^{n_i}$ $\in$ $\T$  and therefore $$  \Mor(\delta^V,\Delta_\Hor)\cong \bigoplus_{n_i} \Mor(\delta^{n_i},\Delta_\Hor).$$ This implies the existence of left/right $\Omega^\bullet(\H^2_q)$--module isomorphisms 
\begin{equation}
    \label{ec.3.29}
    \Upsilon_V: \Mor(\delta^V,\Delta_\Hor)\longrightarrow \Omega^\bullet(\H^2_q)\otimes_{\H^2_q}E^V_\l,
\end{equation}
\begin{equation}
    \label{ec.3.30}
    \widehat{\Upsilon}_V: \Mor(\delta^V,\Delta_\Hor)\longrightarrow E^V_\r\otimes_{\H^2_q}\Omega^\bullet(\H^2_q).
\end{equation}
for every finite--dimensional (unitary) $G$--corepresentation $\delta^V$.

Since $\Upsilon^{-1}_V$, $\widehat{\Upsilon}^{-1}_V$ are the product maps, we conclude that these maps are the \emph{non--commutative geometrical counterpart} of the map $GP$ of equation (\ref{class5}). In particular,  $\Omega^\bullet(\H^2_q)\otimes_{\H^2_q}E^V_\l$ can be interpreted as $E^V_\l$--valued quantum differential forms of $\H^2_q$; while $E^V_\r\otimes_{\H^2_q}\Omega^\bullet(\H^2_q)$ can be interpreted as $E^V_\r$--valued quantum differential forms of $\H^2_q$.

Let $\omega$ be a qpc. By equation (\ref{ec.3.3}), the covariant derivatives can be considered as operators
\begin{equation}
    \label{ec.3.31}
    D^\omega: \Mor(\delta^V,\Delta_\Hor)\longrightarrow \Mor(\delta^V,\Delta_\Hor), \qquad \tau\longmapsto D^\omega(\tau)
\end{equation}
\begin{equation}
    \label{ec.3.32}
    \widehat{D}^\omega: \Mor(\delta^V,\Delta_\Hor)\longrightarrow \Mor(\delta^V,\Delta_\Hor), \qquad \tau\longmapsto \widehat{D}^\omega(\tau)
\end{equation}
where $$D^\omega(\tau)(v):=D^\omega(\tau(v)), \qquad \widehat{D}^\omega(\tau)(v):=\widehat{D}^\omega(\tau(v)),$$ for all $v$ $\in$ $V$. 

\begin{Definition}
\label{3.7}
    Let $\delta^V$ be a finite--dimensional $G$--corepresentation and let $\omega$ $\in$ $\mathfrak{qpc}(\zeta_H)$. We define the  gauge quantum linear connection (abbreviated qlc) of $\omega$ on $E^V_\l$ as the linear map
    $$\nabla^\omega_V: E^V_\l\longrightarrow \Omega^1(\H^2_q)\otimes_{\H^2_q}E^\V_\l, \qquad T\longmapsto \Upsilon_V(D^\omega(T)).$$ In the same way, we define the  gauge quantum linear connection of $\omega$ on $E^V_\r$ as the linear map
    $$\widehat{\nabla}^\omega_V: E^V_\r\longrightarrow E^\V_\r\otimes_{\H^2_q}\Omega^1(\H^2_q), \qquad T\longmapsto \widehat{\Upsilon}_V(\widehat{D}^\omega(T)).$$ 
\end{Definition}
According to Proposition 1 of \cite{micho3}, $\nabla^\omega_V$ satisfies the left Leibniz rule, and $\widehat{\nabla}^\omega_V$ satisfies the right Leibniz rule. For $E^n_\l$ and $E^n_\r$, the gauge qlc's will be denote by
\begin{equation}
    \label{ec.3.33}
    \nabla^\omega_n,\qquad \widehat{\nabla}^\omega_n,
\end{equation}
respectively.

\begin{Remark}
\label{rema5}
    In light of \cite{saldym, micho3}, the map $T\longmapsto \Upsilon_V(\widehat{D}^\omega(T))$ does not satisfies the left Leibniz rule and hence we cannot induce a qlc on $E^V_\l$ with the dual covariant derivative $\widehat{D}^\omega$. Similarly, the map $T\longmapsto \widehat{\Upsilon}_V(D^\omega(T))$ does not satisfies the right Leibniz rule and hence we cannot induce a qlc on $E^V_\r$ with the  covariant derivative $D^\omega$. Of course, for $\omega^c$ (the only regular qpc), this is not a problem because $D=D^{\omega^c}=\widehat{D}^{\omega^c}$.
\end{Remark}

In addition, extending $\nabla^{\omega}_{V}$ to the exterior covariant derivative $$d^{\nabla^{\omega}_{V}}: \Omega^\bullet(\H^2_q)\otimes_{\H^2_q} E^V_\l\longrightarrow \Omega^\bullet(\H^2_q)\otimes_{\H^2_q} E^V_\l$$ such that for all $\mu$ $\in$ $\Omega^k(\H^2_q)$ 
\begin{equation}
\label{ec.3.34}
d^{\nabla^{\omega}_{V}}(\mu\otimes_{\H^2_q} T)=d\mu\otimes_{\H^2_q} T +(-1)^k\mu \nabla^{\omega}_{V}(T), 
\end{equation}
it is easy to check that the following formula holds:
\begin{equation}
\label{ec.3.35}
d^{\nabla^{\omega}_{V}}= \Upsilon_{V}\circ D^{\omega}\circ \Upsilon^{-1}_{V}.
\end{equation}
Similarly, by using the exterior covariant derivative of $\widehat{\nabla}^{\omega}_{V}$  $$d^{\widehat{\nabla}^{\omega}_{V}}: E^V_\r \otimes_{\H^2_q} \Omega^\bullet({\H^2_q})\longrightarrow E^V_\r \otimes_{\H^2_q} \Omega^\bullet({\H^2_q}),$$ which is given by  
\begin{equation}
\label{ec.3.36}
d^{\widehat{\nabla}^{\omega}_{V}}(T\otimes_{\H^2_q} \mu)=\widehat{\nabla}^{\omega}_{V}(T) \mu +T \otimes_{\H^2_q} d\mu , 
\end{equation}
it is easy to check that the following formula holds:
\begin{equation}
\label{ec.3.37}
d^{\widehat{\nabla}^{\omega}_{V}}= \widehat{\Upsilon}_{V}\circ \widehat{D}^{\omega}\circ \widehat{\Upsilon}^{-1}_{V}.
\end{equation}

\begin{Remark}
    \label{clasico}
    In accordance with \cite{nodg}, in differential geometry, if $\omega$ is a principal connection of the principal $U(1)$--bundle of equation (\ref{class1}) and $D^\omega$ denotes its covariant derivative, then the exterior derivative of the gauge linear connection on an associated vector bundle $\pi_V:V\H^2\longrightarrow \H^2$ (see equation (\ref{class2})) satisfies  $$GP\circ D^{\omega}\circ GP^{-1}.$$ Therefore, equations (\ref{ec.3.35}), (\ref{ec.3.37}) are the natural non--commutative geometrical counterparts of the exterior derivative of the classical gauge linear connection. In particular, Definition \ref{3.7} is the non--commutative geometrical counterpart of the classical gauge linear connection.
\end{Remark}

On the other hand, in Durdevich formulation of qpb's we have the following definition

\begin{Definition}
\label{3.9}
    We define the quantum gauge group of the qpb $\zeta_H$ as
    \begin{equation*}
        \begin{aligned}
            \qGG=\{\F:&\Omega^\bullet(P)\longrightarrow \Omega^\bullet(P)\mid \F  \mbox{ is a } \mbox{graded left } \Omega^\bullet(\H^2_q)-\mbox{module isomorphism  }\\   &\;\mbox{such that}\;\;\;\F(\mathbbm{1})=\mathbbm{1}, \; \Delta_{\Omega^\bullet(P)}\circ \F=(\F\otimes \id_{\Gamma^\wedge})\circ \Delta_{\Omega^\bullet(P)}\\ 
     &\mbox{ and }\;\F(\Im(\omega)^\ast)=\F(\Im(\omega))^\ast \;\mbox{ for all }\; \omega \in \mathfrak{qpc}(\zeta_H) \}.
        \end{aligned}
    \end{equation*}
\end{Definition}

The group $\qGG$ is a generalization at the level of differential calculus of the quantum gauge group presented in \cite{br1}. In particular, it is isomorphic to a subgroup of the group of all convolution--invertible maps (\cite{appendix}) $$\f:\Gamma^\wedge\longrightarrow \Omega^\bullet(P)$$ such that $$\f(\mathbbm{1})=\mathbbm{1}\quad \mbox{ and }\quad (\f\otimes \id_{\Gamma^\wedge})\circ \Ad'=\Delta_{\Omega^\bullet(P)}\circ \f,$$ where the map $\Ad'$ is given in equation (\ref{2.f11}) in Appendix A, for the differential structure of $G$. Elements of $\qGG$ are called {\it quantum gauge transformations} and it is worth mentioning that in general, a quantum gauge transformation \emph{does not} commute with the corresponding differentials.

In Durdevich's framework, the action of $\qGG$ on $\mathfrak{qpc}(\zeta_H)$ given by $$ \omega \longmapsto \f\;\widetilde{\ast}\; \omega \;\widetilde{\ast}\; \f^{-1}+ \f\;\widetilde{\ast}\; (d\circ \f^{-1})$$ is not well--defined. Here, $\widetilde{*}$ denotes the convolution product. Furthermore, even if we extend the domain of $\omega$ to $G$ by using the quantum germs map $\pi'$, the induced action on the curvature $$\f \;\widetilde{\ast} \;R^\omega \;\widetilde{\ast}\; \f^{-1} $$ remains ill-defined. However, in accordance with Theorem 4.7 points 1 and 2 of \cite{sald1}, $\qGG$ has a well--defined group action on the space $\mathfrak{qpc}(\zeta_H)$ by 
\begin{equation}
    \label{ec.3.40}
    \F^\circledast \omega:=\F\, \circ \omega,
\end{equation}
and this formula induces a well--defined action on the curvature (see Proposition 4.8 of \cite{sald1}).  Notice that $\F^\circledast \omega$ is only the {\it dualization} of the action of the gauge group on principal connections via the pull--back in differential geometry.

\begin{Proposition}
    \label{3.10}
    The action of $\qGG$ on $\mathfrak{qpc}(\zeta_H)$ is transitivity. 
\end{Proposition}
\begin{proof}
    Let $\omega$ be a qpc different from $\omega^c$. Then, $\omega$ is of the form $$\omega=\omega^c+\lambda,$$ with $$0\not=\lambda(\varsigma)=\mu=x\,\eta_-+y\,\eta_+\,\in\,\Omega^1(\H^2_q)\quad \mbox{ with }\quad \mu^\ast=-\mu.$$
    Thus, consider the left $P$--module morphism 
    \begin{equation}
    \label{ec.3.41}
    \F: \Omega^\bullet(P)\longrightarrow \Omega^\bullet(P)
\end{equation}
such that  (see equation (\ref{ec.2.19.3.1}), (\ref{ec.2.19.3}), (\ref{ec.2.19.4}))  $$\F(\eta_+)=\eta_+,\qquad \F(\eta_-)=\eta_-,\qquad \F(\eta_3)=\mu+\eta_3,$$ $$\F(\eta_-\,\eta_+)=\eta_-\,\eta_+,\quad \F(\eta_+\,\eta_3)=\eta_+\,\mu+\eta_+\,\eta_3,\quad \F(\eta_-\,\eta_3)=\eta_-\,\mu+\eta_-\,\eta_3,$$ $$\F(\eta_-\,\eta_+\,\eta_3)=\eta_-\,\eta_+\,\mu+\eta_-\,\eta_+\,\eta_3. $$ It is clear that $\F$ is a left $\H^2_q$--module morphism.  Let $\psi$ $\in$ $\Omega^\bullet(P)$ and $\sigma=x'\,\eta_-+y'\,\eta_+$ $\in$ $\Omega^1(\H^2_q)$. If $\psi$ $\in$ $P$, then 
\begin{eqnarray*}
    \F(\sigma\,\psi)=\F(x'\,\eta_-\,\psi+y'\,\eta_+\,\psi)=\F(q^r\,x'\,\psi\,\eta_-+q^l\,y'\,\psi\,\eta_+)&=&q^r\,x'\,\psi\,\eta_-+q^l\,y'\,\psi\,\eta_+
    \\
    &=&
    x'\,\eta_-\,\psi+y'\,\eta_+\,\psi
    \\
    &=&
    \sigma\,\psi
    \\
    &=&
    \sigma\,\F(\psi)
\end{eqnarray*}
for some $r$, $l$ $\in$ $\Z$. If $\psi=p_+\,\eta_++p_-\,\eta_-+p_3\,\eta_3$ $\in$ $\Omega^1(P)$;  by Proposition  \ref{algo2} and equation (\ref{ec.2.36}), we obtain  
\begin{eqnarray*}
    \F(\sigma\,\psi)&=& \F((x'\,\eta_-+y'\,\eta_+)(p_+\,\eta_++p_-\,\eta_-+p_3\,\eta_3))
    \\
    &=&
    \F(q_+\,x'\,p_+\,\eta_-\,\eta_+)+\F(q^-_3\,x'\,p_3\,\eta_-\,\eta_3)
    \\
    &+&
    \F(q_-\,y'\,p_-\,\eta_+\,\eta_-)+\F(q^+_3\,y'\,p_3\,\eta_+\,\eta_3)
    \\
    &=&
    \F(q_+\,x'\,p_+\,\eta_-\,\eta_+)+\F(q^-_3\,x'\,p_3\,\eta_-\,\eta_3)
    \\
    &+&
    \F(-q^{2}q_-\,y'\,p_-\,\eta_-\,\eta_+)+\F(q^+_3\,y'\,p_3\,\eta_+\,\eta_3)
    \\
    &=&
    q_+\,x'\,p_+\,\eta_-\,\eta_++q^-_3\,x'\,p_3\,\eta_-\mu+q^-_3\,x'\,p_3\,\eta_-\,\eta_3
    \\
    &-&
    q^{2}q_-\,y'\,p_-\,\eta_-\,\eta_++q^+_3\,y'\,p_3\,\eta_+\,\mu+q^+_3\,y'\,p_3\,\eta_+\,\eta_3
\end{eqnarray*}
and
\begin{eqnarray*}
    \sigma\,\F(\psi)&=&\sigma\,\F(p_+\,\eta_++p_-\,\eta_-+p_3\,\eta_3)
    \\
    &=&
    \sigma\,(p_+\,\eta_+)+\sigma\,(p_-\,\eta_-)+\sigma\,(p_3\,\mu+p_3\,\eta_3) 
    \\
    &=&
    (x'\,\eta_-+y'\,\eta_+)\,(p_+\,\eta_+)+(x'\,\eta_-+y'\,\eta_+)\,(p_-\,\eta_-)+(x'\,\eta_-+y'\,\eta_+)\,(p_3\,\mu+p_3\,\eta_3)
    \\
    &=&
    q_+\,x'\,p_+\,\eta_-\,\eta_++q_-\,y'\,p_-\,\eta_+\,\eta_-
    \\
    &+&
    q^-_3\,x'\,p_3\,\eta_-\,\mu+q^-_3\,x'\,p_3\,\eta_-\,\eta_3+q^+_3\,y'\,p_3\,\eta_+\,\mu+q^+_3\,y'\,p_3\,\eta_+\,\eta_3
    \\
    &=&
    q_+\,x'\,p_+\,\eta_-\,\eta_+-q^2q_-\,y'\,p_-\,\eta_-\,\eta_+
    \\
    &+&
    q^-_3\,x'\,p_3\,\eta_-\,\mu+q^-_3\,x'\,p_3\,\eta_-\,\eta_3+q^+_3\,y'\,p_3\,\eta_+\,\mu+q^+_3\,y'\,p_3\,\eta_+\,\eta_3,
\end{eqnarray*}
where $q_+$, $q_-$, $q^+_3$, $q^-_3$ are powers of $q$. So $\F(\sigma\,\psi)=\sigma\,\F(\psi)$. Similar calculations show that $ \F(\sigma\,\psi)=\sigma\,\F(\psi)$ for all $\sigma$ $\in$ $\Omega^k(\H^2_q)$, $\psi$ $\in$ $\Omega^l(P)$, and by linearity we can conclude that $\F$ is a graded left $\Omega^\bullet(\H^2_q)$--module morphism. Additionally, $\F$ is invertible and a direct calculation proves that its inverse is
the left $P$--module morphism 
    \begin{equation}
    \label{ec.3.41.1}
    \F^{-1}: \Omega^\bullet(P)\longrightarrow \Omega^\bullet(P)
\end{equation}
such that $$\F^{-1}(\eta_+)=\eta_+,\qquad \F^{-1}(\eta_-)=\eta_-,\qquad \F^{-1}(\eta_3)=-\mu+\eta_3,$$ $$\F^{-1}(\eta_-\,\eta_+)=\eta_-\,\eta_+,\quad \F^{-1}(\eta_+\,\eta_3)=-\eta_+\,\mu+\eta_+\,\eta_3,\quad \F^{-1}(\eta_-\,\eta_3)=-\eta_-\,\mu+\eta_-\,\eta_3,$$ $$\F^{-1}(\eta_-\,\eta_+\,\eta_3)=-\eta_-\,\eta_+\,\mu+\eta_-\,\eta_+\,\eta_3. $$

On the other hand, notice that for all $\omega'$ $\in$ $\mathfrak{qpc}(\zeta_H)$ we have $$\Im(\omega')= \Omega^\bullet(\H^2_q)^\dagger+\mathrm{span}_\C\{\eta_3\},$$ where $$\Omega^\bullet(\H^2_q)^\dagger=\{ \sigma\in\Omega^\bullet(\H^2_q) \mid \sigma^\ast=-\sigma\}.$$ So, by Proposition \ref{algo1}, for all $\sigma+w\,\eta_3$ $\in$ $\Omega^\bullet(\H^2_q)^\dagger+\mathrm{span}_\C\{ \eta_3\}$ with $w$ $\in$ $\C$, we get 
\begin{eqnarray*}
    \F((\sigma+w\,\eta_3)^\ast)=\F(\sigma^\ast+w^\ast\,\eta^\ast_3 )=\F(-\sigma-w^\ast\,\eta_3 )&=& -\sigma-w^\ast \,\mu-w^\ast\,\eta_3
    \\
    &=&
    \sigma^\ast+(w\,\mu)^\ast+(w\,\eta_3)^\ast
    \\
    &=&
    (\sigma+w\,\mu+w\,\eta_3)^\ast
    \\
    &=&
    \F(\sigma+w\,\eta_3)^\ast
\end{eqnarray*}
and hence $$\F(\Im(\omega)^\ast)=\F(\Im(\omega))^\ast.$$ This shows that $\F$ $\in$ $\qGG$. Finally
\begin{eqnarray*}
    (\F^\circledast \omega^c)(\theta)=\F(\omega^c(\theta))=w\,\F(\omega^c(\varsigma))=w\,\F(\eta_3)
    =w\,(\mu +\,\eta_3)&=&w\,\omega(\varsigma)
    \\
    &=&
    \omega(w\,\varsigma)
    \\
    &=&
    \omega(\theta).
\end{eqnarray*}
for all $\theta$ $\in$ $\mathfrak{u}_q(1)^\#$ with $\theta=w\,\varsigma$ for some $w$ $\in$ $\C$. We conclude that the action of $\qGG$ on $\mathfrak{qpc}(\zeta_H)$ is transitivity.
\end{proof}

Last proposition tells us that, up a quantum gauge transformation, the only qpc in $\zeta_H$ with the $3D$--differential calculus of $P$, is $\omega^c$.

%As we commented above, this subsection is based on \cite{sald1}, but it is worth mentioning that the theory presented in Section 3 of \cite{sald1} cannot be applied in the qpb $\zeta_H$. This is because, the theory in \cite{sald1} depends on a set of operators $\{T^\l_k\}^d_{k=1}$ $\subseteq$ $\Mor(\delta^V,\Delta_P)$ such that $$\sum^d_{k=1} T^\l_k(e_i)^\ast \;T^\l_k(e_j)=\delta_{ij}\,\mathbbm{1}$$ for every irreducible finite--dimensional corepresentation $\delta^V$ (see Remark 2.8 of \cite{sald1}). However, in our case, our operators $T^j_n$ do not satisfies this equation, except for $n=0$, i.e., except for the trivial $U(1)$--corepresentaiton. 

%In other words, the results presented in this subsection were inspired by Section 3 of \cite{sald1}, but all the constructions and proofs presented here are necessary because $\zeta_H$ do not satisfy the hypothesis on which Section 3 of \cite{sald1} is formulated.

\subsection{A Quantum Riemannian Metric and its Quantum Hodge Operator}

Let us define
\begin{equation}
    \label{ec.3.44}
    \eta_-\eta_+:=\dvol
\end{equation}
and by Proposition \ref{2.2}, we have $$\Omega^2(\H^2_q)=\H^2_q\,\dvol.$$

Following the \emph{classical} case, the proper definition of Hodge operators requires the notion of Riemannian metrics.

\begin{Definition}
    \label{3.11}
    Consider the $\H_q$--valued $\C$--sesquilinear map (antilinear in the second coordinate)
    \begin{equation*}
        \langle-,-\rangle_\l: \Omega^\bullet(\H^2_q)\times\Omega^\bullet(\H^2_q)\longrightarrow \H^2_q
    \end{equation*}
    such that quantum differential forms of different degree are orthogonal, and
    $$\langle b_1,b_2\rangle_\l=b_1\,b^\ast_2$$ for all $b_1$, $b_2$ $\in$ $\H^2_q$,  
    $$\langle x_1\eta_-+y_1\eta_+, x_2\eta_-+y_2\eta_+\rangle_\l=q^2\, x_1x^\ast_2+y_1y^\ast_2 $$ for all $x_1\eta_-+y_1\eta_+$, $x_2\eta_-+y_2\eta_+$ $\in$ $\Omega^1(\H_q)$, and since $\Omega^2(\H^2_q)=\H^2_q\,\dvol$, we define
    $$\langle b_1\,\dvol,b_2\,\dvol\rangle_\l=b_1\,b^\ast_2.$$

    On the other hand, consider the $\H_q$--valued $\C$--sesquilinar map (now antilinear in the first coordinate)
    \begin{equation*}
        \langle-,-\rangle_\r: \Omega^\bullet(\H^2_q)\times\Omega^\bullet(\H^2_q)\longrightarrow \H^2_q,\qquad \langle \mu_1,\mu_2\rangle_\r:=\langle\mu^\ast_1,\mu^\ast_2\rangle_\l.
    \end{equation*}
\end{Definition}

By definition, 
$$\langle b\,\mu_1,\mu_2\rangle_\l=b\,\langle \mu_1,\mu_2\rangle_\l, \quad \langle \mu_1,\mu_2\,b\rangle_\r=\langle \mu_1,\mu_2\rangle_\r\,b $$
and
$$\langle \mu_1b,\mu_2\rangle_\l=\langle\mu_1,\mu_2b^\ast\rangle_\l,\qquad \langle b\mu_1,\mu_2\rangle_\r=\langle\mu_1,b^\ast\mu_2\rangle_\r  $$ 

\noindent for every $\mu_1$, $\mu_2$ $\in$ $\Omega^k(\H^2_q)$, $b$ $\in$ $\H^2_q$. Moreover, since $\H^2_q$ can be considered as a subset of a $C^\ast$--algebra (\cite{fur}), it follows that $$\langle \mu,\mu\rangle_\l=0\quad \Longleftrightarrow \quad \mu=0,\qquad \langle \mu,\mu\rangle_\r=0\quad \Longleftrightarrow \quad \mu=0,$$ i.e., $\langle-,-\rangle_\l$ and $\langle-,-\rangle_\r$ are actually  $\H^2_q$--valued inner products. We will refer to $\langle-,-\rangle_\l$, $\langle-,-\rangle_\r$ as the left/right quantum Riemmanian metric, respectively.

\begin{Proposition}
    \label{3.12}
    There exists an antilinear isomorphism $$\star_\l: \Omega^{k}(\H^2_q) \longrightarrow \Omega^{2-k}(\H^2_q)$$ such that 
    \begin{equation}
        \label{ec.3.45}
        \mu_1\,\mu_2=\langle\mu_1,\star^{-1}_\l \mu_1\rangle_\l \,\dvol
    \end{equation}
    for all $\mu_1$ $\in$ $\Omega^k(\H^2_q)$, $\mu_2$ $\in$ $\Omega^{2-k}(\H^2_q)$. The map $\star_\l$ receives the name of left quantum Hodge operator. 
\end{Proposition}

\begin{proof}
    Let us define $$\star_\l: \H^2_q\longrightarrow \Omega^2(\H^2_q)\qquad b\longmapsto \star_\l(b):=b^\ast\,\dvol,$$ $$\star_\l :\Omega^1(\H^2_q)\longrightarrow \Omega^1(\H^2_q),\qquad \mu=x\,\eta_-+y\,\eta_+\longmapsto \star_\l(\mu):=-y^\ast\eta_-+x^\ast\eta_+,$$ $$\star_\l:\Omega^2(\H^2_q)\longrightarrow \H^2_q,\qquad b\,\dvol\longmapsto \star_\l(b\,\dvol)=b^\ast.$$ Notice that $$\star^2_\l=(-1)^{k(2-k)}\;\id_{\Omega^k(\H^2_q)}$$ and hence, $\star_\l$ is invertible. So, we obtain 
    
$$\star^{-1}_\l:\Omega^2(\H^2_q)\longrightarrow \H^2_q,\qquad b\,\dvol\longmapsto\star^{-1}_\l(b\,\dvol)=b^\ast,$$ 
     $$\star^{-1}_\l :\Omega^1(\H^2_q)\longrightarrow \Omega^1(\H^2_q),\qquad \mu=x\,\eta_-+y\,\eta_+\longmapsto \star^{-1}_\l(\mu):=y^\ast\eta_--x^\ast\eta_+,$$ $$\star^{-1}_\l: \H^2_q\longrightarrow \Omega^2(\H^2_q)\qquad b\longmapsto \star^{-1}_\l(b):=b^\ast\,\dvol.$$

     It should be clear that equation (\ref{ec.3.45}) is satisfied for $k=0,2$. For $k=1$,  by  equation (\ref{ec.2.36}), we get
     \begin{eqnarray*}
         (x_1\eta_-+y_1\eta_+)(x_2\eta_-+y_2\eta_+)=x_1\eta_-y_2\eta_++y_1\eta_+x_2\eta_-&=&q^2 x_1y_2\eta_-\eta_++q^{-2}y_1x_2\eta_+\eta_-
         \\
         &=&
         (q^2 x_1y_2-y_1x_2)\,\dvol;
     \end{eqnarray*}
     while 
     \begin{eqnarray*}
         \langle x_1\eta_-+y_1\eta_+,\star^{-1}_\l(x_2\eta_-+y_2\eta_+)\rangle_\l\,\dvol&=&\langle x_1\eta_-+y_1\eta_+,y^\ast_2\eta_--x^\ast_2\eta_+\rangle_\l\,\dvol
         \\
         &=&
     (q^2 x_1y_2-y_1x_2)\,\dvol
     \end{eqnarray*}
for all $x_1\eta_-+y_1\eta_+$, $x_2\eta_-+y_2\eta_+$ $\in$ $\Omega^1(\H^2_q)$ and the proposition follows.
\end{proof}

\begin{Corollary}
    There exists a linear isomorphism $$\star_\r: \Omega^{k}(\H^2_q) \longrightarrow \Omega^{2-k}(\H^2_q)$$ such that 
    \begin{equation}
        \label{ec.3.46}
        \mu^\ast_1\,\mu_2=\langle\mu_1,\star^{-1}_\r \mu_1\rangle_\r \,\dvol
    \end{equation}
     for all $\mu_1$ $\in$ $\Omega^k(\H^2_q)$, $\mu_2$ $\in$ $\Omega^{2-k}(\H^2_q)$. The map $\star_\r$ receives the name of right quantum Hodge operator and it is given by $$\star_\r:=\star_\l\circ \ast.$$ 
\end{Corollary}

\begin{Definition}
    \label{codif}
    Following the classical case, we define the left codifferential $$d^{\star_\l}: \Omega^{k+1}(\H^2_q)\longrightarrow \Omega^k(\H^2_q)$$ as the operator $$d^{\star_\l}:=(-1)^{k+1}\star^{-1}_\l \,\circ \,d\,\circ \star_\l.$$ For $k+1=0$, we take $d^{\star_\l}=0$. Similarly, we define the right codifferential $$d^{\star_\r}: \Omega^{k+1}(\H^2_q)\longrightarrow \Omega^k(\H^2_q)$$ as the operator $$d^{\star_\r}:=(-1)^{k+1}\star^{-1}_\r \,\circ \,d\,\circ \star_\r=\ast\,\circ \,d^{\star_\l}\,\circ \ast .$$ For $k+1=0$, we take $d^{\star_\r}=0$. 
\end{Definition}

Now, it is possible to study Hodge theory on $(\Omega^\bullet(\H^2_q),d,\ast)$, but this is not the purpose of this paper. We prefer to focus the rest of this work in the following quantum differential operators.

\begin{Definition}
    \label{lapla1}
    We define the left quantum Laplacian as the operator $$\square_q:\Omega^k(\H^2_q)\longrightarrow \Omega^k(\H^2_q) $$ given by $$\square_q:=d\circ d^{\star_\l}+d^{\star_\l}\circ d.$$ Similarly, we define the right quantum Laplacian as the operator $$\widehat{\square}_q:\Omega^k(\H^2_q)\longrightarrow \Omega^k(\H^2_q) $$ given by $$\widehat{\square}_q:=d\circ d^{\star_\r}+d^{\star_\r}\circ d=d\circ \ast\circ d^{\star_\l}\circ \ast +\ast\circ d^{\star_\l}\circ \ast\circ d =\ast \circ \square_q\circ \ast.$$
\end{Definition}

\begin{Definition}
    \label{cogaugedif}
    Let $\delta^V$ be a finite--dimensional (unitary) $G$--corepresentation and $\omega$ $\in$ $\mathfrak{qpc}(\zeta_H)$. 
    Following the classical case, by considering the exterior covariant derivative $d^{\nabla^\omega_\V}$ of the gauge qlc $\nabla^\omega_\V$, we define the operator $$d^{\nabla^\omega_V\, \star_\l}: \Omega^{k+1}(\H^2_q)\otimes_{\H^2_q} E^V_\l\longrightarrow \Omega^k(\H^2_q)\otimes_{\H^2_q} E^V_\l$$ given by $$d^{\nabla^\omega_V\, \star_\l}:=(-1)^{k+1}((\star^{-1}_\l\circ \ast)\otimes_{\H^2_q}\id_{E^V_\l}) \,\circ \,d^{\nabla^\omega_\V}\,\circ ((\ast\circ \star_\l)\otimes_{\H^2_q}\id_{E^V_\l}).$$ For $k+1=0$, we take $d^{\nabla^\omega_V\, \star_\l}=0$.
    
    Similarly, by considering the exterior covariant derivative $d^{\widehat{\nabla}^\omega_\V}$ of the gauge qlc $\widehat{\nabla}^\omega_\V$, we define the operator $$d^{\widehat{\nabla}^\omega_V\, \star_\r}: E^V_\r\otimes_{\H^2_q}\Omega^{k+1}(\H^2_q) \longrightarrow E^V_\r\otimes_{\H^2_q}\Omega^k(\H^2_q)$$ given by $$d^{\widehat{\nabla}^\omega_V\, \star_\r}:=(-1)^{k+1}(\id_{E^V_\r}\otimes_{\H^2_q} \star^{-1}_\r) \,\circ \,d^{\widehat{\nabla}^\omega_\V}\,\circ (\id_{E^V_\l}\otimes_{\H^2_q}  \star_\r).$$ For $k+1=0$, we take $d^{\widehat{\nabla}^\omega_V\, \star_\r}=0$.
\end{Definition}

\begin{Definition}
    \label{lapla2}
    Let $\delta^V$ be a finite--dimensional (unitary) $G$--corepresentation and $\omega$ $\in$ $\mathfrak{qpc}(\zeta_H)$. We define the left quantum gauge Laplacian as the operator $$\square^{\omega}_q:\Omega^{k}(\H^2_q)\otimes_{\H^2_q} E^V_\l\longrightarrow \Omega^k(\H^2_q)\otimes_{\H^2_q} E^V_\l $$ given by $$\square^{\omega}_q:=d^{\nabla^\omega_\V}\circ d^{\nabla^\omega_V\, \star_\l}+d^{\nabla^\omega_V\, \star_\l}\circ d^{\nabla^\omega_\V}.$$ Similarly, we define the right quantum gauge Laplacian as the operator $$\widehat{\square}^{\omega}_q:E^V_\r\otimes_{\H^2_q} \Omega^{k}(\H^2_q) \longrightarrow E^V_\r\otimes_{\H^2_q} \Omega^{k}(\H^2_q) $$ given by $$\widehat{\square}^{\omega}_q:=d^{\widehat{\nabla}^\omega_\V}\circ d^{\widehat{\nabla}^\omega_V\, \star_\r}+d^{\widehat{\nabla}^\omega_V\, \star_\r}\circ d^{\widehat{\nabla}^\omega_\V}.$$
\end{Definition}

\section{The Quantum Gauge Laplacians}

The main purpose of this paper is to exhibit two non--commuting $U(1)$--gauge Laplacians on the quantum hyperboloid $\H^2_q$, while also providing a thorough and comprehensive study of the geometric structures on the related quantum spaces. The latter objective has been achieved in the preceding sections, and we now have all the necessary tools to accomplish the first objective. We will exclusively focus on the left/right quantum gauge Laplacian associated with the canonical qpc $\omega^c$ for a $G$--corepresentation $\delta^n$ $\in$ $\T$ and degree $0$. 

In the \emph{classical} case, to define a Hilbert space structure on (see equation (\ref{class3})) $$\Gamma(V\H^2)\;\cong\; C^\infty(SU(1,1),V)^{U(1)}$$ and calculate the spectrum of the gauge Laplacian on $L^2$--elements it is used the Haar measure of $SU(1,1)$. However, in the \emph{quantum} case,  $P=SU_q(1,1)$ does not posses a Haar measure (actually, it is not a compact quantum group \cite{fur}). Therefore, there is not a \emph{proper} Hilbert space structure on $$\Mor(\delta^n,\Delta_P).$$

For this reason, in this section, we will treat $\Mor(\delta^n,\Delta_P)$ exclusively as an infinite--dimensional $\C$--vector space without considering any additional structure on it (such as the left/right $\H^2_q$--module structure or some possible Hilbert space structure). In addition, we are going to consider  
$$\square^{\omega^c}_q:\Mor(\delta^n,\Delta_P)\longrightarrow \Mor(\delta^n,\Delta_P)  $$ and $$ \widehat{\square}^{\omega^c}_q:\Mor(\delta^n,\Delta_P)\longrightarrow \Mor(\delta^n,\Delta_P)$$ solely as linear operators, for every $n$ $\in$ $\Z$. Therefore, we will define the spectrum of each gauge Laplacian as the set of its eigenvalues with respect to a basis of eigenvectors, i.e., 
  \begin{equation}
  \label{spec1}
  \begin{aligned}
      \mathrm{Spec}(\square^{\omega^c}_q):=\{ \lambda\in\C \mid \square^{\omega^c}_q(T_j)=\lambda \,T_j \;\mbox{ with }&\,\{T_j\}_{j\in J} \\
      &\mbox{ some linear basis of } \Mor(\delta^n,\Delta_P) \},
  \end{aligned}
  \end{equation}
   \begin{equation}
   \label{spec2}
   \begin{aligned}
      \mathrm{Spec}(\widehat{\square}^{\omega^c}_q):=\{ \lambda\in\C \mid \widehat{\square}^{\omega^c}_q(T_j)=\lambda \,T_j \;\mbox{ with }&\,\{T_j\}_{j\in J} 
      \\
      & \mbox{ some linear basis of } \Mor(\delta^n,\Delta_P) \},
       \end{aligned}
  \end{equation}
  for some infinite index set $J$.

\begin{Proposition}
    \label{lap}
    For $n=0$, we have $$\square^{\omega^c}_q=\square_q\quad \mbox{ and }\quad   \widehat{\square}^{\omega^c}_q=\widehat{\square}_q.$$ In particular, $\widehat{\square}^{\omega^c}_q\not=\square^{\omega^c}_q$ since $\widehat{\square}_q=\ast\circ \square_q \circ \ast$.
\end{Proposition}
\begin{proof}
    Let $T$ $\in$ $\Mor(\delta^0,\Delta_P)$. Then $T(1)=b$ for some $b$ $\in$ $\H^2_q$ and for every $b$ $\in$ $\H^2_q$, we can define an element $T$ $\in$ $\Mor(\delta^0,\Delta_P)$ given by $T(1)=b$. This induces a linear isomorphism 
        \begin{equation}
            \label{lin}
            \H^2_q\cong \Mor(\delta^0,\Delta_P).
        \end{equation}
        By equations (\ref{ec.3.4.1}) (\ref{nec0}), for every $T$ $\in$ $\Mor(\delta^0,\Delta_P)$ with $T(1)=b$, we obtain $$D(b)=db=(db)\,\mathbbm{1} = (db)\,T^0_0(1);$$ so $$\nabla^{\omega^c}_0(T)=\Upsilon_0(D(T))=db\otimes_{\H^2_q} T^0_0.$$ Thus
        \begin{eqnarray*}
            \square^{\omega^c}_q(T)&=&-((\star^{-1}_\l\circ \ast)\otimes_{\H^2_q}\id) d^{\nabla^{\omega^c}_0} ((\ast\circ \star_\l)\otimes_{\H^2_q}\id)(db\otimes_{\H^2_q} T^0_0)
            \\
            &=&
            -((\star^{-1}_\l\circ \ast)\otimes_{\H^2_q}\id) d^{\nabla^{\omega^c}_0} ((\star_\l(db))^\ast \otimes_{\H^2_q} T^0_0)
            \\
            &=& 
            -((\star^{-1}_\l\circ \ast)\otimes_{\H^2_q}\id) (d((\star_\l(db))^\ast\otimes_{\H^2_q} T^0_0-(\star_\l(db))^\ast\,\nabla^{\omega^c}_0(T^0_0))
            \\
            &=& 
            -((\star^{-1}_\l\circ \ast)\otimes_{\H^2_q}\id) (d((\star_\l(db))^\ast\otimes_{\H^2_q} T^0_0)
            \\
            &=& 
            d^{\star_\l}db\otimes_{\H^2_q} T^0_0
            \\
            &=&
            \square_q(b)\otimes_{\H^2_q} T^0_0
            \\
            &=&
            \square_q(T(1))\otimes_{\H^2_q} T^0_0
        \end{eqnarray*}
        
        and considering the linear isomorphism of equation (\ref{lin}), it follows that $\square^{\omega^c}_q=\square_q$.

        On the other hand, $$D(b)=db=\mathbbm{1}\,(db)=T^0_0(1)\,(db);$$ so $$\widehat{\nabla}^{\omega^c}_0(T)=\widehat{\Upsilon}_0(D(T))=T^0_0\otimes_{\H^2_q} db.$$ Thus
        \begin{eqnarray*}
            \widehat{\square}^{\omega^c}_q(T)&=&-(\id\otimes_{\H^2_q}\star^{-1}_\r) d^{\widehat{\nabla}^{\omega^c}_0} (\id\otimes_{\H^2_q} \star_\r)(T^0_0\otimes_{\H^2_q}db )
            \\
            &=&
            -(\id\otimes_{\H^2_q}\star^{-1}_\r) d^{\widehat{\nabla}^{\omega^c}_0} (T^0_0 \otimes_{\H^2_q} \star_\r(db))
            \\
            &=& 
            -(\id\otimes_{\H^2_q}\star^{-1}_\r) (\widehat{\nabla}^{\omega^c}_0(T^0_0)\,(\star_\r(db))+T^0_0\otimes_{\H^2_q} d(\star_\r(db)))
            \\
            &=& 
            -(\id\otimes_{\H^2_q}\star^{-1}_\r) (T^0_0\otimes_{\H^2_q} d(\star_\r(db)))
            \\
            &=& 
            T^0_0\otimes_{\H^2_q}  d^{\star_\r}db
            \\
            &=&
            T^0_0\otimes_{\H^2_q} \widehat{\square}_q(b)
            \\
            &=&
            T^0_0\otimes_{\H^2_q} \widehat{\square}_q(T(1))
        \end{eqnarray*}
        and considering the linear isomorphism of equation (\ref{lin}) it follows that $\widehat{\square}^{\omega^c}_q=\widehat{\square}_q$.
\end{proof}

\subsection{The Spectrum of the Left Quantum Gauge Laplacian}
This subsection is based in direct calculations using the commutation relations presented in the whole text, particularly in  Corollary \ref{coro1}.

\subsubsection{For $n=0$}

Let $T$, $U$ $\in$ $\Mor(\delta^0,\Delta_P)$. If 
\begin{equation}
    \label{n=0.0.1}
    T(1)=\mathbbm{1},
\end{equation}
then
\begin{equation}
    \label{n=0.0.2}
    \square_q(T)=0\,T.
\end{equation}
If
\begin{equation}
    \label{n=0.1.1}
    T(1)=\alpha^k\,\gamma^{\ast k} \qquad \mbox{ with }\qquad k\in \N,
\end{equation}
then
\begin{equation}
    \label{n=0.1.2}
    \square_q(T)=-(q^{-2k}\,(1+q^4)\,[k]\,[k+1])\,T.
\end{equation}
If
\begin{equation}
    \label{n=0.2.1}
    T(1)=\alpha^{\ast k}\,\gamma^k  \qquad \mbox{ with }\qquad k\in \N,
\end{equation}
then
\begin{equation}
    \label{n=0.2.2}
    \square_q(T)=-(q^{-2k}\,(1+q^4)\,[k]\,[k+1])\,T.
\end{equation}
If 
\begin{equation}
    \label{n=0.3.1}
    T(1)=\gamma^k\,\gamma^{\ast k} \quad \mbox{ and }\quad U(1)=\gamma^{k-1}\,\gamma^{\ast k-1}  \qquad \mbox{ with }\qquad k\in \N,
\end{equation}
then
\begin{equation}
    \label{n=0.3.2}
    \square_q(T)=-((q^{-2k}+q^4)\,[k]+q^{-2k+2}\,(1+q^2)\,[k]^2)\,T\,- q^{-2k}\,(1+q^4)\,[k]^2\,U. 
\end{equation}
If 
\begin{equation}
    \label{n=0.4.1}
    T(1)=\alpha^t \gamma^k\,\gamma^{\ast l} \quad \mbox{ and }\quad U(1)=\alpha^{t} \gamma^{k-1}\,\gamma^{\ast l-1} \quad \mbox{ such that }\quad t+k-l=0
\end{equation}
with $t$, $k$, $l$ $\in$ $\N$, then
\begin{equation}
    \label{n=0.4.2}
    \begin{aligned}
        \square_q(T)=&-q^{-2l}\,([l]\,[t+1]+q^{2t+2}\,[l]\,[k]+q^4\,[t]\,[l+1]+q^{4+2t}\,[k]\,[l+1])\,T
         \\
         -&
    q^{-2l+2t}\,(1+q^4)\,[l]\,[k]\,U.
    \end{aligned}
\end{equation}
Finally, if
\begin{equation}
    \label{n=0.5.1}
    T(1)=\alpha^{\ast t} \gamma^k\,\gamma^{\ast l} \quad \mbox{ and }\quad U(1)=\alpha^{\ast t} \gamma^{k-1}\,\gamma^{\ast l-1}\quad \mbox{ such that }\quad -t+k-l=0
\end{equation}
with $t$, $k$, $l$ $\in$ $\N$, then
\begin{equation}
    \label{n=0.5.2}
    \begin{aligned}
        \square_q(T)=&-q^{-2t}\,([t]\,[k+1]+q^{-2l}\,[l]\,[k+1]+q^4\,[k]\,[t+1]+q^{-2l+4}\,[k]\,[l])\,T
         \\
         -&
          q^{-2t-2l}\,(1+q^4)\,[l]\,[k]\,U.
    \end{aligned}
\end{equation}

It is worth mentioning that $$\beta_0:=\{ T: \C\longrightarrow P \mid T(1)=\alpha^t\,\gamma^k\,\gamma^{\ast l} \;\mbox{ such that }\; t+k-l=0 \;\mbox{ with }\; t\in\Z,\; k,l\in \N_0 \}$$ is a linear basis of $\Mor(\delta^0,\Delta_P)$. We proceed to show a linear basis of eigenvectors of $\square_q$. 

Consider the linear map $$T^k_{\gamma\gamma^\ast}:\C\longrightarrow P,\qquad T^k_{\gamma\gamma^\ast}(1)=\gamma^k\,\gamma^{\ast k} \quad \mbox{ with }\quad k\in \N_0.$$ \\

In particular $$T^0_{\gamma\gamma^\ast}(1)=\mathbbm{1}.$$ Then 
\begin{eqnarray*}  \square_q(T^k_{\gamma\gamma^\ast})&=&\lambda^k_1\,T^k_{\gamma\gamma^\ast}+\lambda^k_2\,T^{k-1}_{\gamma\gamma^\ast},\\
    \square_q(T^{k-1}_{\gamma\gamma^\ast})&=& \lambda^{k-1}_1\,T^{k-1}_{\gamma\gamma^\ast}+\lambda^{k-1}_2\,T^{k-2}_{\gamma\gamma^\ast},\\
    &\vdots&
    \\
\square_q(T^1_{\gamma\gamma^\ast})&=&\lambda^1_1\,T^1_{\gamma\gamma^\ast}+\lambda^1_2\,T^0_{\gamma\gamma^\ast},
    \\
\square_q(T^0_{\gamma\gamma^\ast})&=&\lambda^0_1\,T^0_{\gamma\gamma^\ast},
\end{eqnarray*}
where the polynomials $\lambda^j_1$, $\lambda^j_2$ in $q$ are given in equation (\ref{n=0.3.2}) for $j=1,...,k$ and $\lambda^0_1=0$ (see equation (\ref{n=0.0.2})). For $1\leq k$, consider the polynomial 
\begin{equation}
    \label{n=0.6.1}
p(T^k_{\gamma\gamma^\ast}):=T^k_{\gamma\gamma^\ast}+ \sum^{k-1}_{i=0} \prod^i_{j=0} {\lambda^{k-j}_2 \over (\lambda^k_1-\lambda^{k-(j+1)}_1)} T^{k-(i+1)}_{\gamma\gamma^\ast}.
\end{equation}
For $k=0$, consider the polynomial 
\begin{equation}
    \label{n=0.6.1.1}
p(T^0_{\gamma\gamma^\ast}):=T^0_{\gamma\gamma^\ast},
\end{equation}
which is an eigenvector with eigenvalue $\lambda^0_1$. It is worth mentioning that, since $q\not=1,-1$, we get $\lambda^k_1-\lambda^{k-(j+1)}_1$ $\not=0$. Thus

\begin{eqnarray*}    \square_q(p(T^k_{\gamma\gamma^\ast}))&=&\square_q(T^k_{\gamma\gamma^\ast})+ \sum^{k-1}_{i=0} \prod^i_{j=0} {\lambda^{k-j}_2 \over (\lambda^k_1-\lambda^{k-(j+1)}_1)} \square_q(T^{k-(i+1)}_{\gamma\gamma^\ast}).
\\
&=&
\lambda^k_1\,T^k_{\gamma\gamma^\ast}+\lambda^k_2\,T^{k-1}_{\gamma\gamma^\ast}\\
    &+& {\lambda^{k-1}_1\lambda^k_2 \over \lambda^k_1-\lambda^{k-1}_1}\,T^{k-1}_{\gamma\gamma^\ast}+{\lambda^{k-1}_2\lambda^k_2 \over \lambda^k_1-\lambda^{k-1}_1}\,T^{k-2}_{\gamma\gamma^\ast}
    \\
    &+&
    \\
    &\vdots&
    \\
    &+&
    {\lambda^1_1\lambda^k_2\,\lambda^{k-2}_2\cdots \lambda^2_2 \over (\lambda^k_1-\lambda^{k-1}_1)(\lambda^k_1-\lambda^{k-2}_1)\cdots (\lambda^k_1-\lambda^1_1)}\,T^{1}_{\gamma\gamma^\ast}
    \\
    &+&
    {\lambda^1_2\lambda^k_2\,\lambda^{k-2}_2\cdots \lambda^2_2 \over (\lambda^k_1-\lambda^{k-1}_1)(\lambda^k_1-\lambda^{k-2}_1)\cdots (\lambda^k_1-\lambda^1_1)}\,T^{0}_{\gamma\gamma^\ast}.
    \\
    &=& \lambda^k_1 \,T^k_{\gamma\gamma^\ast}+\lambda^k_1 \,{\lambda^k_2 \over \lambda^k_1-\lambda^{k-1}_1}\,T^{k-1}_{\gamma\gamma^\ast}+
    \\
    &\vdots&
    \\
    &+& 
    \lambda^k_1\,{\lambda^k_2\,\lambda^{k-2}_2\cdots \lambda^2_2 \over (\lambda^k_1-\lambda^{k-1}_1)(\lambda^k_1-\lambda^{k-2}_1)\cdots (\lambda^k_1-\lambda^1_1)}\,T^{1}_{\gamma\gamma^\ast}
    \\
    &+&
    \lambda^k_1\,{\lambda^k_2\,\lambda^{k-2}_2\cdots \lambda^2_2\lambda^1_2 \over \lambda^k_1(\lambda^k_1-\lambda^{k-1}_1)(\lambda^k_1-\lambda^{k-2}_1)\cdots (\lambda^k_1-\lambda^1_1)}\,T^{0}_{\gamma\gamma^\ast}
    \\
    &=&
    \lambda^k_1\,(T^k_{\gamma\gamma^\ast}+{\lambda^k_2 \over \lambda^k_1-\lambda^{k-1}_1}\,T^{k-1}_{\gamma\gamma^\ast}+
    \\
    &\vdots&
    \\
    &+&
    {\lambda^k_2\,\lambda^{k-2}_2\cdots \lambda^2_2 \over (\lambda^k_1-\lambda^{k-1}_1)(\lambda^k_1-\lambda^{k-2}_1)\cdots (\lambda^k_1-\lambda^1_1)}\,T^{1}_{\gamma\gamma^\ast}
    \\
    &+&
    {\lambda^k_2\,\lambda^{k-2}_2\cdots \lambda^2_2\lambda^1_2 \over \lambda^k_1(\lambda^k_1-\lambda^{k-1}_1)(\lambda^k_1-\lambda^{k-2}_1)\cdots (\lambda^k_1-\lambda^1_1)}\,T^{0}_{\gamma\gamma^\ast})
    \\
    &=&
    \lambda^k_1\,p(T^k_{\gamma\gamma^\ast}).
\end{eqnarray*}

\noindent Hence 
\begin{equation}
    \label{n=0.6.2}
\square_q(p(T^k_{\gamma\gamma^\ast}))=\lambda^k_1\,p(T^k_{\gamma\gamma^\ast}),
\end{equation}
where 
\begin{equation}
    \label{n=0.6.3}
\lambda^k_1=-((q^{-2k}+q^4)\,[k]+q^{-2k+2}\,(1+q^2)\,[k]^2).
\end{equation}
Furthermore, for each $k$ $\in$ $\N_0$, the set $$\{ p(T^j_{\gamma\gamma^\ast})\mid j=0,...,k \} $$ is clearly linear independent and $$\mathrm{span}_\C\{p(T^j_{\gamma\gamma^\ast})\mid j=0,...,k \}=\mathrm{span}_\C\{T^j_{\gamma\gamma^\ast}\mid j=0,...,k \}.$$ 

Similarly, consider the linear map $$T^{t,k,l}_{\alpha\gamma^\ast}:\C\longrightarrow P,\qquad T^{t,k,l}_{\alpha\gamma^\ast}(1)=\alpha^t\,\gamma^k\,\gamma^{\ast l} \quad \mbox{ with }\quad t+k-l=0,\quad t,l\,\in\,\N,\; k\,\in\, \N_0$$ ($T^{t,0,l}(1)=\alpha\gamma^\ast$). Notice that $k < l$. Then 
\begin{eqnarray*}  \square_q(T^{t,k,l}_{\alpha\gamma^\ast})&=&w^k_1 \,T^{t,k,l}_{\alpha\gamma^\ast}+w^k_2 \,T^{t,k-1,l-1}_{\alpha\gamma^\ast},\\
    \square_q(T^{t,k-1,l-1}_{\alpha\gamma^\ast})&=&w^{k-1}_1 \,T^{t,k-1,l-1}_{\alpha\gamma^\ast}+w^{k-1}_2 \,T^{t,k-2,l-2}_{\alpha\gamma^\ast},\\
    &\vdots&
    \\
\square_q(T^{t,1,l-k+1}_{\alpha\gamma^\ast})&=&w^1_1\,T^{t,1,l-k+1}_{\alpha\gamma^\ast}+w^1_2\,T^{t,0,l-k}_{\alpha\gamma^\ast},
    \\
\square_q(T^{t,0,l-k}_{\alpha\gamma^\ast})&=& w^0_1\,T^{t,0,l-k}_{\alpha\gamma^\ast},
\end{eqnarray*}
where the polynomials $w^j_1$, $w^j_2$ in $q$ are given in equation (\ref{n=0.4.2}) for $j=1,...,k$ and $w^0_1=-(q^{-2k}\,(1+q^4)\,[k]\,[k+1])$ (see equation (\ref{n=0.1.2})). For $1\leq k$, consider the polynomial 
\begin{equation}
    \label{n=0.7.1}
p(T^{t,k,l}_{\alpha\gamma^\ast}):=T^{t,k,l}_{\alpha\gamma^\ast}+ \sum^{k-1}_{i=0} \prod^i_{j=0} {w^{k-j}_2 \over (w^k_1-w^{k-(j+1)}_1)} T^{t,k-(i+1),l-(i+1)}_{\alpha\gamma^\ast}.
\end{equation}
For $k=0$, consider the polynomial 
\begin{equation}
    \label{n=0.7.1.1}
p(T^{t,0,l}_{\alpha\gamma^\ast}):=T^{t,0,l}_{\alpha\gamma^\ast},
\end{equation}
which is an eigenvector with eigenvalue $w^0_1$. As above, since $q\not=-1,1$ we get  $w^k_1-w^{k-(j+1)}_1\not=0$, and a direct calculation shows that
\begin{equation}
    \label{n=0.7.2}
\square_q(p(T^{t,k,l}_{\alpha\gamma^\ast}))=w^k_1\,p(T^{t,k,l}_{\alpha\gamma^\ast}),
\end{equation}
where 
\begin{equation}
    \label{n=0.7.3}
w^k_1=-q^{-2l}\,([l]\,[t+1]+q^{2t+2}\,[l]\,[k]+q^4\,[t]\,[l+1]+q^{4+2t}\,[k]\,[l+1]).
\end{equation}
Furthermore, for each $t$, $l$ $\in$ $\N$, $k$ $\in$ $\N_0$ such that $t+k-l=0$, the set $$\{ p(T^{t,k-j,l-j}_{\alpha\gamma^\ast})\mid j=0,...,k \} $$ is clearly linear independent and $$\mathrm{span}_\C\{p(T^{t,k-j,l-j}_{\alpha\gamma^\ast})\mid j=0,...,k \}=\mathrm{span}_\C\{T^{t,k,l}_{\alpha\gamma^\ast}\mid j=0,...,k \}.$$ 

On the other hand, consider the linear map
$$T^{t,k,l}_{\alpha^\ast\gamma}:\C\longrightarrow P,\qquad T^{t,k,l}_{\alpha^\ast\gamma}(1)=\alpha^{\ast t}\,\gamma^k\,\gamma^{\ast l} \quad \mbox{ with }\quad -t+k-l=0,\quad t,k\in\,\N,\; l\,\in\, \N_0$$ ($T^{t,k,0}_{\alpha^\ast\gamma}=\alpha^\ast\gamma$). Notice that $l < k$. Thus
\begin{eqnarray*}  \square_q(T^{t,k,l}_{\alpha^\ast\gamma})&=&u^l_1 \,T^{t,k,l}_{\alpha^\ast\gamma}+u^l_2 \,T^{t,k-1,l-1}_{\alpha^\ast\gamma},\\
    \square_q(T^{t,k-1,l-1}_{\alpha^\ast\gamma})&=&u^{l-1}_1 \,T^{t,k-1,l-1}_{\alpha^\ast\gamma}+u^{l-1}_2 \,T^{t,k-2,l-2}_{\alpha^\ast\gamma},\\
    &\vdots&
    \\
\square_q(T^{t,k-l+1,1}_{\alpha^\ast\gamma})&=&u^1_1\,T^{t,k-l+1,1}_{\alpha^\ast\gamma}+u^1_2\,T^{t,k-l,0}_{\alpha^\ast\gamma},
    \\
\square_q(T^{t,k-l,0}_{\alpha^\ast\gamma})&=& u^0_1\,T^{t,k-l,0}_{\alpha^\ast\gamma},
\end{eqnarray*}
where the polynomials $u^j_1$, $u^j_2$ in $q$ are given in equation (\ref{n=0.5.2}) for $j=1,...,k$ and $u^0_1=-(q^{-2l}\,(1+q^4)\,[l]\,[l+1])$ (see equation (\ref{n=0.2.2})). Consider the polynomial 
\begin{equation}
    \label{n=0.8.1}
p(T^{t,k,l}_{\alpha^\ast\gamma}):=T^{t,k,l}_{\alpha^\ast\gamma}+ \sum^{l-1}_{i=0} \prod^i_{j=0} {u^{l-j}_2 \over (u^l_1-u^{l-(j+1)}_1)} T^{t,k-(i+1),l-(i+1)}_{\alpha^\ast\gamma}.
\end{equation}
For $l=0$, consider the polynomial 
\begin{equation}
    \label{n=0.8.1.1}
p(T^{t,k,0}_{\alpha\gamma^\ast}):=T^{t,k,0}_{\alpha\gamma^\ast},
\end{equation}
which is an eigenvector with eigenvalue $u^0_1$. As above, since $q\not=-1,1$ we get  $u^l_1-u^{l-(j+1)}_1\not=0$, and a direct calculation shows that
\begin{equation}
    \label{n=0.8.2}
\square_q(p(T^{t,k,l}_{\alpha^\ast\gamma}))=u^l_1\,p(T^{t,k,l}_{\alpha^\ast\gamma}),
\end{equation}
where 
\begin{equation}
    \label{n=0.8.3}
u^l_1=-q^{-2t}\,([t]\,[k+1]+q^{-2l}\,[l]\,[k+1]+q^4\,[k]\,[t+1]+q^{-2l+4}\,[k]\,[l]).
\end{equation}
Furthermore, for each $t$, $k$ $\in$ $\N$,  $l$ $\in$ $\N_0$ such that $-t+k-l=0$, the set $$\{ p(T^{t,k-j,l-j}_{\alpha^\ast\gamma})\mid j=0,...,k \} $$ is clearly linear independent and $$\mathrm{span}_\C\{p(T^{t,k-j,l-j}_{\alpha^\ast\gamma})\mid j=0,...,k \}=\mathrm{span}_\C\{T^{t,k,l}_{\alpha^\ast\gamma}\mid j=0,...,k \}.$$ 

Therefore, the set 
\begin{equation}
    \label{n=0.9}
\{ p(T^k_{\gamma\gamma^\ast}), \; p(T^{t,k,l}_{\alpha\gamma^\ast}),\; p(T^{t,k,l}_{\alpha^\ast\gamma})  \}
\end{equation}
is a linear basis of $\Mor(\delta^0,\Delta_P)$ composed of eigenvectors of $\square_q$. The spectrum of the left quantum Laplacian is presented in table $1$ (see equation (\ref{spec1})), and it is worth mentioning that it is an infinite discrete set.

\begin{center}
\begin{table}[b]
\centering
\begin{tabular}{|c|c|c|c|c|}
\hline 
\multicolumn{1}{|c|}{$T(1)$} & $n=0$ & \multicolumn{1}{|c|}{$\mathrm{Eigenvalue}$}\rule[-0.3cm]{0cm}{0.8cm}\\\hline
\multicolumn{1}{|c|}{$p(\gamma^k\gamma^{\ast\,k})$} & $k\in \N_0$ & \multicolumn{1}{|c|} {$-((q^{-2k}+q^4)\,[k]+q^{-2k+2}\,(1+q^2)\,[k]^2)$} \rule[-0.4cm]{0cm}{1.1cm}\\\hline
\multicolumn{1}{|c|}{$p(\alpha^{t}\gamma^k\gamma^{\ast\,l})$} &  $t+k-l=0$ & \multicolumn{1}{|c|} {$-q^{-2l}\,([l]\,[t+1]+q^{2t+2}\,[l]\,[k]+q^4\,[t]\,[l+1]+q^{4+2t}\,[k]\,[l+1])$} \rule[-0.3cm]{0cm}{0.8cm}\\\hline
\multicolumn{1}{|c|}{$p(\alpha^{\ast t}\gamma^k\gamma^{\ast\,l})$} &$-t+k-l=0$   & \multicolumn{1}{|c|} {$-q^{-2t}\,([t]\,[k+1]+q^{-2l}\,[l]\,[k+1]+q^4\,[k]\,[t+1]+q^{-2l+4}\,[k]\,[l])$} \rule[-0.3cm]{0cm}{0.8cm}\\\hline
\end{tabular}
\caption{Spectrum of $\square_q$.}
\end{table}
\end{center}

\clearpage

\subsubsection{For $n\geq 1$}

Let $T$, $U$ $\in$ $\Mor(\delta^n,\Delta_P)$. 
If
\begin{equation}
    \label{0<n.1.1}
    T(1)=\alpha^n,
\end{equation}
then 
\begin{equation}
    \label{0<n.1.2}
     \square^{\omega^c}_q(T)=-q^4\,[n]\,T.
\end{equation}
If  
\begin{equation}
    \label{0<n.2.1}
    T(1)=\alpha^{t}\,\gamma^k \qquad \mbox{ such that }\qquad t+k=n, \quad t,\,k\,\in\, \N,
\end{equation}
then 
\begin{equation}
    \label{0<n.2.2}
     \square^{\omega^c}_q(T)=-q^{4}\,([t]+q^{2t}[k])\, T.
\end{equation}
If  
\begin{equation}
    \label{0<n.3.1}
    T(1)=\gamma^n,
\end{equation}
then 
\begin{equation}
    \label{0<n.3.2}
     \square^{\omega^c}_q(T)=-q^{4}\,[n]\, T.
\end{equation}
If 
\begin{equation}
    \label{0<n.4.1}
    T(1)=\alpha^t\,\gamma^{\ast l} \quad \mbox{ such that }\quad t-l=n,\quad t,\,l\,\in\, \N,
\end{equation}
then
\begin{equation}
    \label{0<n.4.2}
     \square^{\omega^c}_q(T)=-q^{-2l}\,([t+1]\,[l]+q^4\,[t]\,[l+1])\,T.
\end{equation}
If
\begin{equation}
    \label{n=1.2.1}
    T(1)=\alpha^{\ast t}\,\gamma^k  \qquad \mbox{ with }\qquad  -t+k=n,
\end{equation}
then
\begin{equation}
    \label{n=1.2.2}
    \square^{\omega^c}_q(T)=-q^{-2t}\,([t]\,[k+1]+q^4\,[t+1]\,[k])\,T.
\end{equation}
If 
\begin{equation}
    \label{n=1.3.1}
    T(1)=\gamma^k\,\gamma^{\ast l} \quad \mbox{ and } \quad U(1)=\gamma^{k-1}\,\gamma^{\ast l-1} \quad \mbox{ such that } \quad k-l=n
\end{equation}
with $k,\,l\,\in\,\N,$ then
\begin{equation}
    \label{n=1.3.2}
    \square^{\omega^c}_q(T)=-(q^{-2l}\,(1+q^2\,[k])\,[l]+q^4\,(1+q^{-2l}\,[l])\,[k])\,T\,-\, q^{-2l}\,(1+q^4)\,[k]\,[l]\,U. 
\end{equation}
If 
\begin{equation}
    \label{n=1.4.1}
    T(1)=\alpha^t \gamma^k\,\gamma^{\ast l} \quad \mbox{ and }\quad U(1)=\alpha^{t} \gamma^{k-1}\,\gamma^{\ast l-1} \quad \mbox{ such that }\quad t+k-l=n
\end{equation}
with $t$, $k$, $l$ $\in$ $\N$, then
\begin{equation}
    \label{n=1.4.2}
    \begin{aligned}
        \square^{\omega^c}_q(T)=&-q^{-2l}\,([l]\,[t+1]+q^{2t+2}\,[l]\,[k]+q^4\,[t]\,[l+1]+q^{4+2t}\,[k]\,[l+1])\,T
         \\
         -&
    q^{-2l+2t}\,(1+q^4)\,[l]\,[k]\,U.
    \end{aligned}
\end{equation}
Finally, if
\begin{equation}
    \label{n=1.5.1}
    T(1)=\alpha^{\ast t} \gamma^k\,\gamma^{\ast l} \quad \mbox{ and }\quad U(1)=\alpha^{\ast t} \gamma^{k-1}\,\gamma^{\ast l-1}\quad \mbox{ such that }\quad -t+k-l=n
\end{equation}
with $t$, $k$, $l$ $\in$ $\N$, then
\begin{equation}
    \label{n=1.5.2}
    \begin{aligned}
        \square^{\omega^c}_q(T)=&-q^{-2t}\,([t]\,[k+1]+q^{-2l}\,[l]\,[k+1]+q^4\,[k]\,[t+1]+q^{-2l+4}\,[k]\,[l])\,T
         \\
         -&
          q^{-2t-2l}\,(1+q^4)\,[l]\,[k]\,U.
    \end{aligned}
\end{equation}
It is worth mentioning that $$\beta_n:=\{ T: \C\longrightarrow P \mid T(1)=\alpha^t\,\gamma^k\,\gamma^{\ast l} \;\mbox{ such that }\; t+k-l=n \;\mbox{ with }\; t\in\Z,\; k,l\in \N_0 \}$$ is a linear basis of $\Mor(\delta^n,\Delta_P)$ and the operator $\square^{\omega^c}_q$ acting on the elements of $\beta_n$ follows a behavior similar to that shown in the previous subsection. In this way, {\it mutatis mutandis}, we can find a linear basis of $\Mor(\delta^n,\Delta_P)$ composed of eigenvectors of $\square^{\omega^c}_q$. The spectrum of the left quantum gauge Laplacian for $n\geq 1$ is presented in table 2 (see equation (\ref{spec1})), and as in the last subsection, it is an infinite discrete set.

\begin{center}
\begin{table}[b]
\centering
\begin{tabular}{|c|c|c|c|c|}
\hline 
\multicolumn{1}{|c|}{$T(1)$} & $n\in \N$ & \multicolumn{1}{|c|}{$\mathrm{Eigenvalue}$}\rule[-0.3cm]{0cm}{0.8cm}\\\hline
\multicolumn{1}{|c|}{$p(\gamma^k\gamma^{\ast\,l})$} & $k-l=n $ & \multicolumn{1}{|c|} {$-(q^{-2l}\,(1+q^2\,[k])\,[l]+q^4\,(1+q^{-2l}\,[l])\,[k])$} \rule[-0.4cm]{0cm}{1.1cm}\\\hline
\multicolumn{1}{|c|}{$p(\alpha^{t}\gamma^k\gamma^{\ast\,l})$} & $t+k-l=n$  & \multicolumn{1}{|c|} {$-q^{-2l}\,([l]\,[t+1]+q^{2t+2}\,[l]\,[k]+q^4\,[t]\,[l+1]+q^{4+2t}\,[k]\,[l+1])$} \rule[-0.3cm]{0cm}{0.8cm}\\\hline
\multicolumn{1}{|c|}{$p(\alpha^{\ast t}\gamma^k\gamma^{\ast\,l})$} & $-t+k-l=n$ & \multicolumn{1}{|c|} {$-q^{-2t}\,([t]\,[k+1]+q^{-2l}\,[l]\,[k+1]+q^4\,[k]\,[t+1]+q^{-2l+4}\,[k]\,[l])$} \rule[-0.3cm]{0cm}{0.8cm}\\\hline
\end{tabular}
\caption{Spectrum of $\square^{\omega^c}_q$ for $n$ $\in$ $\N$.}
\end{table}
\end{center}

\clearpage

\subsubsection{For $n=-m$ with $m\geq 1$}

Let $T$, $U$ $\in$ $\Mor(\delta^{-m},\Delta_P)$. 
If
\begin{equation}
    \label{0>n.1.1}
    T(1)=\alpha^{\ast m},
\end{equation}
then 
\begin{equation}
    \label{0>n.1.2}
     \square^{\omega^c}_q(T)=-q^{-2m}\,[m]\,T.
\end{equation}
If  
\begin{equation}
    \label{0>n.2.1}
    T(1)=\alpha^{\ast t}\,\gamma^{\ast l} \qquad \mbox{ such that }\qquad -t-l=-m, \quad t,\,l\,\in\, \N,
\end{equation}
then 
\begin{equation}
    \label{0>n.2.2}
     \square^{\omega^c}_q(T)=-q^{-2t}\,([t]+q^{-2l}\,[l])\, T.
\end{equation}
If  
\begin{equation}
    \label{0>n.3.1}
    T(1)=\gamma^{\ast m},
\end{equation}
then 
\begin{equation}
    \label{0>n.3.2}
     \square^{\omega^c}_q(T)=-q^{-2m}\,[m]\, T.
\end{equation}
If 
\begin{equation}
    \label{0>n.4.1}
    T(1)=\alpha^t\,\gamma^{\ast l} \quad \mbox{ such that }\quad t-l=-m,\quad t,\,l\,\in\, \N,
\end{equation}
then
\begin{equation}
    \label{0>n.4.2}
     \square^{\omega^c}_q(T)=-q^{-2l}\,([t+1]\,[l]+q^4\,[t]\,[l+1])\,T.
\end{equation}
If
\begin{equation}
    \label{n>1.2.1}
    T(1)=\alpha^{\ast t}\,\gamma^k  \qquad \mbox{ with }\qquad  -t+k=-m,
\end{equation}
then
\begin{equation}
    \label{n>1.2.2}
    \square^{\omega^c}_q(T)=-q^{-2t}\,([t]\,[k+1]+q^4\,[t+1]\,[k])\,T.
\end{equation}
If 
\begin{equation}
    \label{n>1.3.1}
    T(1)=\gamma^k\,\gamma^{\ast l} \quad \mbox{ and } \quad U(1)=\gamma^{k-1}\,\gamma^{\ast l-1} \quad \mbox{ such that } \quad k-l=-m
\end{equation}
with $k,\,l\,\in\,\N,$ then
\begin{equation}
    \label{n>1.3.2}
    \square^{\omega^c}_q(T)=-(q^{-2l}\,(1+q^2\,[k])\,[l]+q^4\,(1+q^{-2l}\,[l])\,[k])\,T\,-\, q^{-2l}\,(1+q^4)\,[k]\,[l]\,U. 
\end{equation}
If 
\begin{equation}
    \label{n>1.4.1}
    T(1)=\alpha^t \gamma^k\,\gamma^{\ast l} \quad \mbox{ and }\quad U(1)=\alpha^{t} \gamma^{k-1}\,\gamma^{\ast l-1} \quad \mbox{ such that }\quad t+k-l=-m
\end{equation}
with $t$, $k$, $l$ $\in$ $\N$, then
\begin{equation}
    \label{n>1.4.2}
    \begin{aligned}
        \square^{\omega^c}_q(T)=&-q^{-2l}\,([l]\,[t+1]+q^{2t+2}\,[l]\,[k]+q^4\,[t]\,[l+1]+q^{4+2t}\,[k]\,[l+1])\,T
         \\
         -&
    q^{-2l+2t}\,(1+q^4)\,[l]\,[k]\,U.
    \end{aligned}
\end{equation}
Finally, if
\begin{equation}
    \label{n>1.5.1}
    T(1)=\alpha^{\ast t} \gamma^k\,\gamma^{\ast l} \quad \mbox{ and }\quad U(1)=\alpha^{\ast t} \gamma^{k-1}\,\gamma^{\ast l-1}\quad \mbox{ such that }\quad -t+k-l=-m
\end{equation}
with $t$, $k$, $l$ $\in$ $\N$, then
\begin{equation}
    \label{n>1.5.2}
    \begin{aligned}
        \square^{\omega^c}_q(T)=&-q^{-2t}\,([t]\,[k+1]+q^{-2l}\,[l]\,[k+1]+q^4\,[k]\,[t+1]+q^{-2l+4}\,[k]\,[l])\,T
         \\
         -&
          q^{-2t-2l}\,(1+q^4)\,[l]\,[k]\,U.
    \end{aligned}
\end{equation}
It is worth mentioning that $$\beta_{-m}:=\{ T: \C\longrightarrow P \mid T(1)=\alpha^t\,\gamma^k\,\gamma^{\ast l} \;\mbox{ such that }\; t+k-l=-m \;\mbox{ with }\; t\in\Z,\; k,l\in \N_0 \}$$ is a linear basis of $\Mor(\delta^{-m},\Delta_P)$ and the operator $\square^{\omega^c}_q$ acting on the elements of $\beta_{-m}$ follows a behavior similar to that shown in  Subsection 4.1.1. In this way, {\it mutatis mutandis}, we can find a linear basis of $\Mor(\delta^{-m},\Delta_P)$ composed of eigenvectors of $\square^{\omega^c}_q$. The spectrum of the left quantum gauge Laplacian for $n\leq -1$ is presented in table 3 (see equation (\ref{spec1})), and as in the previous subsections, it is an infinite discrete set.

\begin{center}
\begin{table}[b]
\centering
\begin{tabular}{|c|c|c|c|c|}
\hline 
\multicolumn{1}{|c|}{$T(1)$} & $n=-m$, $m\in \N$ & \multicolumn{1}{|c|}{$\mathrm{Eigenvalue}$}\rule[-0.3cm]{0cm}{0.8cm}\\\hline
\multicolumn{1}{|c|}{$p(\gamma^k\gamma^{\ast\,l})$} & $k-l=-m$ & \multicolumn{1}{|c|} {$-(q^{-2l}\,(1+q^2\,[k])\,[l]+q^4\,(1+q^{-2l}\,[l])\,[k])$} \rule[-0.4cm]{0cm}{1.1cm}\\\hline
\multicolumn{1}{|c|}{$p(\alpha^{t}\gamma^k\gamma^{\ast\,l})$} & $t+k-l=-m$  & \multicolumn{1}{|c|} {$-q^{-2l}\,([l]\,[t+1]+q^{2t+2}\,[l]\,[k]+q^4\,[t]\,[l+1]+q^{4+2t}\,[k]\,[l+1])$} \rule[-0.3cm]{0cm}{0.8cm}\\\hline
\multicolumn{1}{|c|}{$p(\alpha^{\ast t}\gamma^k\gamma^{\ast\,l})$} & $-t+k-l=-m$ & \multicolumn{1}{|c|} {$-q^{-2t}\,([t]\,[k+1]+q^{-2l}\,[l]\,[k+1]+q^4\,[k]\,[t+1]+q^{-2l+4}\,[k]\,[l])$} \rule[-0.3cm]{0cm}{0.8cm}\\\hline
\end{tabular}
\caption{Spectrum of $\square^{\omega^c}_q$ for $n=-m$, $m$ $\in$ $\N$.}
\end{table}
\end{center}

\clearpage

\subsection{The Spectrum of the Right Quantum Gauge Laplacian}
As above, this subsection is based in direct calculations using the commutation relations presented in the whole text and Corollary \ref{coro1}.

\subsubsection{For $n=0$}

Let $T$, $U$ $\in$ $\Mor(\delta^0,\Delta_P)$. If 
\begin{equation}
    \label{-n=0.0.1}
    T(1)=\mathbbm{1}
\end{equation}
then
\begin{equation}
    \label{-n=0.0.2}
    \widehat{\square}_q(T)=0\,T.
\end{equation}
If
\begin{equation}
    \label{-n=0.1.1}
    T(1)=\alpha^k\,\gamma^{\ast k} \qquad \mbox{ with }\qquad k\in \N,
\end{equation}
then
\begin{equation}
    \label{-n=0.1.2}
    \widehat{\square}_q(T)=-(q^{-2k}\,(1+q^4)\,[k]\,[k+1])\,T.
\end{equation}
If
\begin{equation}
    \label{-n=0.2.1}
    T(1)=\alpha^{\ast k}\,\gamma^k  \qquad \mbox{ with }\qquad k\in \N,
\end{equation}
then
\begin{equation}
    \label{-n=0.2.2}
    \widehat{\square}_q(T)=-(q^{-2k}\,(1+q^4)\,[k]\,[k+1])\,T.
\end{equation}
If 
\begin{equation}
    \label{-n=0.3.1}
    T(1)=\gamma^k\,\gamma^{\ast k} \quad \mbox{ and }\quad U(1)=\gamma^{k-1}\,\gamma^{\ast k-1}  \qquad \mbox{ with }\qquad k\in \N,
\end{equation}
then
\begin{equation}
    \label{-n=0.3.2}
    \widehat{\square}_q(T)=-((q^{-2k}+q^4)\,[k]+q^{-2k+2}\,(1+q^2)\,[k]^2)\,T\,- q^{-2k}\,(1+q^4)\,[k]^2\,U. 
\end{equation}
If 
\begin{equation}
    \label{-n=0.4.1}
    T(1)=\alpha^t \gamma^k\,\gamma^{\ast l} \quad \mbox{ and }\quad U(1)=\alpha^{t} \gamma^{k-1}\,\gamma^{\ast l-1} \quad \mbox{ such that }\quad t+k-l=0
\end{equation}
with $t$, $k$, $l$ $\in$ $\N$, then
\begin{equation}
    \label{-n=0.4.2}
    \begin{aligned}
        \widehat{\square}_q(T)=&-q^{-2t}\,([t]\,[l+1]+q^{-2k}\,[k]\,[l+1]+q^4\,[l]\,[t+1]+q^{-2k+4}\,[l]\,[k]) \,T
         \\
         -&
     q^{-2k}\,(1+q^4)\,[k]\,[l]\,U.
    \end{aligned}
\end{equation}
Finally, if
\begin{equation}
    \label{-n=0.5.1}
    T(1)=\alpha^{\ast t} \gamma^k\,\gamma^{\ast l} \quad \mbox{ and }\quad U(1)=\alpha^{\ast t} \gamma^{k-1}\,\gamma^{\ast l-1}\quad \mbox{ such that }\quad -t+k-l=0
\end{equation}
with $t$, $k$, $l$ $\in$ $\N$, then
\begin{equation}
    \label{-n=0.5.2}
    \begin{aligned}
        \widehat{\square}_q(T)=&-q^{-2k}\,([k]\,[t+1]+q^{2t+2}\,[k]\,[l]+q^4\,[t]\,[k+1]+q^{4+2t}\,[l]\,[k+1])\,T
         \\
         -&
          q^{-2k}\,(1+q^4)\,[k]\,[l]\,U.
    \end{aligned}
\end{equation}

It is worth mentioning that, since $t+k-l=0$, it follows that 
\begin{eqnarray*}
    & & q^{-2t}\,([t]\,[l+1]+q^{-2k}\,[k]\,[l+1]+q^4\,[l]\,[t+1]+q^{-2k+4}\,[l]\,[k])
    \\
    &=&
    q^{-2l}\,([l]\,[t+1]+q^{2t+2}\,[l]\,[k]+q^4\,[t]\,[l+1]+q^{4+2t}\,[k]\,[l+1])
\end{eqnarray*}
and 
\begin{eqnarray*}
q^{-2l+2t}\,(1+q^4)\,[l]\,[k]=q^{-2k}\,(1+q^4)\,[l]\,[k].
\end{eqnarray*}
In addition, since $-t+k-l=0$, it follows that
\begin{eqnarray*}
 & &
    q^{-2k}\,([k]\,[t+1]+q^{2t+2}\,[k]\,[l]+q^4\,[t]\,[k+1]+q^{4+2t}\,[l]\,[k+1])
    \\
    &=&
    q^{-2t}\,([t]\,[k+1]+q^{-2l}\,[l]\,[k+1]+q^4\,[k]\,[t+1]+q^{-2l+4}\,[k]\,[l])
\end{eqnarray*}
and 
\begin{eqnarray*}
q^{-2t-2l}\,(1+q^4)\,[l]\,[k]=q^{-2k}\,(1+q^4)\,[l]\,[k].
\end{eqnarray*}

Therefore, in accordance with Subsection 4.1.1, we conclude that 
\begin{equation}
    \label{noesposible}
    \square_q=\widehat{\square}_q.
\end{equation}
This is not an obvious relation, since $$\star_\l\circ \ast\not= \ast\circ \star_\l.$$

\subsubsection{For $n\geq 1$}

Let $T$, $U$ $\in$ $\Mor(\delta^n,\Delta_P)$. 
If
\begin{equation}
    \label{r0<n.1.1}
    T(1)=\alpha^n,
\end{equation}
then 
\begin{equation}
    \label{r0<n.1.2}
     \widehat{\square}^{\omega^c}_q(T)=-q^{-2n}\,[n]\,T.
\end{equation}
If  
\begin{equation}
    \label{r0<n.2.1}
    T(1)=\alpha^{t}\,\gamma^k \qquad \mbox{ such that }\qquad t+k=n, \quad t,\,k\,\in\, \N,
\end{equation}
then 
\begin{equation}
    \label{r0<n.2.2}
     \widehat{\square}^{\omega^c}_q(T)=-q^{-2k}\,(q^{-2t}\,[t]+[k])\, T.
\end{equation}
If  
\begin{equation}
    \label{r0<n.3.1}
    T(1)=\gamma^n,
\end{equation}
then 
\begin{equation}
    \label{r0<n.3.2}
     \widehat{\square}^{\omega^c}_q(T)=-q^{-2n}\,[n]\, T.
\end{equation}
If 
\begin{equation}
    \label{r0<n.4.1}
    T(1)=\alpha^t\,\gamma^{\ast l} \quad \mbox{ such that }\quad t-l=n,\quad t,\,l\,\in\, \N,
\end{equation}
then
\begin{equation}
    \label{r0<n.4.2}
     \widehat{\square}^{\omega^c}_q(T)=-q^{-2t}\,(q^4\,[t+1]\,[l]+[t]\,[l+1])\,T.
\end{equation}
If
\begin{equation}
    \label{rn<1.2.1}
    T(1)=\alpha^{\ast t}\,\gamma^k  \qquad \mbox{ with }\qquad  -t+k=n,
\end{equation}
then
\begin{equation}
    \label{rn<1.2.2}
    \widehat{\square}^{\omega^c}_q(T)=-q^{-2k}\,([k]\,[t+1]+q^4\,[t]\,[k+1])\,T.
\end{equation}
If 
\begin{equation}
    \label{rn<1.3.1}
    T(1)=\gamma^k\,\gamma^{\ast l} \quad \mbox{ and } \quad U(1)=\gamma^{k-1}\,\gamma^{\ast l-1} \quad \mbox{ such that } \quad k-l=n
\end{equation}
with $k,\,l\,\in\,\N,$ then
\begin{equation}
    \label{rn<1.3.2}
    \widehat{\square}^{\omega^c}_q(T)=-(q^{-2k}\,(1+q^2\,[l])\,[k]+q^4\,(1+q^{-2k}\,[k])\,[l])\,T\,-\, q^{-2k}\,(1+q^4)\,[l]\,[k]\,U. 
\end{equation}
If 
\begin{equation}
    \label{rn<1.4.1}
    T(1)=\alpha^t \gamma^k\,\gamma^{\ast l} \quad \mbox{ and }\quad U(1)=\alpha^{t} \gamma^{k-1}\,\gamma^{\ast l-1} \quad \mbox{ such that }\quad t+k-l=n
\end{equation}
with $t$, $k$, $l$ $\in$ $\N$, then
\begin{equation}
    \label{rn<1.4.2}
    \begin{aligned}
         \widehat{\square}^{\omega^c}_q(T)=&-q^{-2t}\,([t]\,[l+1]+q^{-2k}\,[k]\,[l+1]+q^4\,[l]\,[t+1]+q^{-2k+4}\,[l]\,[k]) \,T
         \\
         -&
     q^{-2k}\,(1+q^4)\,[k]\,[l]\,U.
    \end{aligned}
\end{equation}
Finally, if
\begin{equation}
    \label{rn<1.5.1}
    T(1)=\alpha^{\ast t} \gamma^k\,\gamma^{\ast l} \quad \mbox{ and }\quad U(1)=\alpha^{\ast t} \gamma^{k-1}\,\gamma^{\ast l-1}\quad \mbox{ such that }\quad -t+k-l=n
\end{equation}
with $t$, $k$, $l$ $\in$ $\N$, then
\begin{equation}
    \label{rn<1.5.2}
    \begin{aligned}
        \widehat{\square}^{\omega^c}_q(T)=&-q^{-2k}\,([k]\,[t+1]+q^{2t+2}\,[k]\,[l]+q^4\,[t]\,[k+1]+q^{4+2t}\,[l]\,[k+1])\,T
         \\
         -&
          q^{-2k}\,(1+q^4)\,[k]\,[l]\,U.
    \end{aligned}
\end{equation}

In accordance with Subsection 4.1.3, we conclude that 
\begin{equation}
    \label{noesposible2}
    \widehat{\square}^{\omega^c}_q=\ast \circ \square^{\omega^c}_q\circ \ast
\end{equation}
for $n$ $\in$ $\N$. In this way, {\it mutatis mutandis}, we can find a linear basis of $\Mor(\delta^n,\Delta_P)$ composed of eigenvectors of $\widehat{\square}^{\omega^c}_q$. As we will show in the next subsection, this basis cannot be equal to the one obtained in Subsection 4.1.2. 

The spectrum of the right quantum gauge Laplacian for $n\geq 1$ is presented in table 4 (see equation (\ref{spec2})), and as in the previous subsections, it is an infinite discrete set.

\begin{center}
\begin{table}[b]
\centering
\begin{tabular}{|c|c|c|c|c|}
\hline 
\multicolumn{1}{|c|}{$T(1)$} & $n\in \N$ & \multicolumn{1}{|c|}{$\mathrm{Eigenvalue}$}\rule[-0.3cm]{0cm}{0.8cm}\\\hline
\multicolumn{1}{|c|}{$p(\gamma^k\gamma^{\ast\,k})$} & $k\in \N_0$ & \multicolumn{1}{|c|} {$-((q^{-2k}+q^4)\,[k]+q^{-2k+2}\,(1+q^2)\,[k]^2)$} \rule[-0.4cm]{0cm}{1.1cm}\\\hline
\multicolumn{1}{|c|}{$p(\alpha^{t}\gamma^k\gamma^{\ast\,l})$} &  $t+k-l=0$ & \multicolumn{1}{|c|} {$-q^{-2t}\,([t]\,[l+1]+q^{-2k}\,[k]\,[l+1]+q^4\,[l]\,[t+1]+q^{-2k+4}\,[l]\,[k])$} \rule[-0.3cm]{0cm}{0.8cm}\\\hline
\multicolumn{1}{|c|}{$p(\alpha^{\ast t}\gamma^k\gamma^{\ast\,l})$} &$-t+k-l=0$   & \multicolumn{1}{|c|} {$-q^{-2k}\,([k]\,[t+1]+q^{2t+2}\,[k]\,[l]+q^4\,[t]\,[k+1]+q^{4+2t}\,[l]\,[k+1])$} \rule[-0.3cm]{0cm}{0.8cm}\\\hline
\end{tabular}
\caption{Spectrum of $\widehat{\square}^{\omega^c}_q$ for $n\in \N$.}
\end{table}
\end{center}

\clearpage

\subsubsection{For $n=-m$ with $m\geq 1$}

Let $T$, $U$ $\in$ $\Mor(\delta^{-m},\Delta_P)$. 
If 
\begin{equation}
    \label{0<n.1.1r}
    T(1)=\alpha^{\ast m},
\end{equation}
then 
\begin{equation}
    \label{0<n.1.2r}
     \widehat{\square}^{\omega^c}_q(T)=-q^{4}\,[m]\,T.
\end{equation}
If  
\begin{equation}
    \label{0<n.2.1r}
    T(1)=\alpha^{\ast t}\,\gamma^{\ast l} \qquad \mbox{ such that }\qquad -t-l=-m, \quad t,\,l\,\in\, \N,
\end{equation}
then 
\begin{equation}
    \label{0<n.2.2r}
     \widehat{\square}^{\omega^c}_q(T)=-q^{4}\,([t]+q^{2t}\,[l])\, T.
\end{equation}
If  
\begin{equation}
    \label{0<n.3.1r}
    T(1)=\gamma^{\ast m},
\end{equation}
then 
\begin{equation}
    \label{0<n.3.2r}
     \widehat{\square}^{\omega^c}_q(T)=-q^{4}\,[m]\, T.
\end{equation}
If 
\begin{equation}
    \label{0<n.4.1r}
    T(1)=\alpha^t\,\gamma^{\ast l} \quad \mbox{ such that }\quad t-l=-m,\quad t,\,l\,\in\, \N,
\end{equation}
then
\begin{equation}
    \label{0<n.4.2r}
     \widehat{\square}^{\omega^c}_q(T)=-q^{-2t}\,([t]\,[l+1]+q^4\,[t+1]\,[l])\,T.
\end{equation}
If
\begin{equation}
    \label{n<1.2.1r}
    T(1)=\alpha^{\ast t}\,\gamma^k  \qquad \mbox{ with }\qquad  -t+k=-m,
\end{equation}
then
\begin{equation}
    \label{n<1.2.2r}
    \widehat{\square}^{\omega^c}_q(T)=-q^{-2k}\,([t+1]\,[k]+q^4\,[t]\,[k+1])\,T.
\end{equation}
If 
\begin{equation}
    \label{n<1.3.1r}
    T(1)=\gamma^k\,\gamma^{\ast l} \quad \mbox{ and } \quad U(1)=\gamma^{k-1}\,\gamma^{\ast l-1} \quad \mbox{ such that } \quad k-l=-m
\end{equation}
with $k,\,l\,\in\,\N,$ then
\begin{equation}
    \label{n<1.3.2r}
    \widehat{\square}^{\omega^c}_q(T)=-(q^{-2k}\,(1+q^2\,[l])\,[k]+q^4\,(1+q^{-2k}\,[k])\,[l])\,T\,-\, q^{-2k}\,(1+q^4)\,[l]\,[k]\,U. 
\end{equation}
If 
\begin{equation}
    \label{n<1.4.1r}
    T(1)=\alpha^t \gamma^k\,\gamma^{\ast l} \quad \mbox{ and }\quad U(1)=\alpha^{t} \gamma^{k-1}\,\gamma^{\ast l-1} \quad \mbox{ such that }\quad t+k-l=-m
\end{equation}
with $t$, $k$, $l$ $\in$ $\N$, then
\begin{equation}
    \label{n<1.4.2r}
    \begin{aligned}
        \widehat{\square}^{\omega^c}_q(T)=&-q^{-2t}\,([t]\,[l+1]+q^{-2k}\,[k]\,[l+1]+q^4\,[l]\,[t+1]+q^{-2k+4}\,[l]\,[k]) \,T
         \\
         -&
     q^{-2k}\,(1+q^4)\,[k]\,[l]\,U.
    \end{aligned}
\end{equation}
Finally, if
\begin{equation}
    \label{n<1.5.1r}
    T(1)=\alpha^{\ast t} \gamma^k\,\gamma^{\ast l} \quad \mbox{ and }\quad U(1)=\alpha^{\ast t} \gamma^{k-1}\,\gamma^{\ast l-1}\quad \mbox{ such that }\quad -t+k-l=-m
\end{equation}
with $t$, $k$, $l$ $\in$ $\N$, then
\begin{equation}
    \label{n<1.5.2r}
    \begin{aligned}
        \widehat{\square}^{\omega^c}_q(T)=&-q^{-2k}\,([k]\,[t+1]+q^{2t+2}\,[k]\,[l]+q^4\,[t]\,[k+1]+q^{4+2t}\,[l]\,[k+1])\,T
         \\
         -&
          q^{-2k}\,(1+q^4)\,[k]\,[l]\,U.
    \end{aligned}
\end{equation}

In accordance with Subsection 4.1.3, we conclude that 
\begin{equation}
    \label{noesposible3}
    \widehat{\square}^{\omega^c}_q=\ast \circ \square^{\omega^c}_q\circ \ast
\end{equation}
for $n=-m$, $m$ $\in$ $\N$. In this way, {\it mutatis mutandis}, we can find a linear basis of $\Mor(\delta^{-m},\Delta_P)$ composed of eigenvectors of $\widehat{\square}^{\omega^c}_q$. As we will show in the next subsection, this basis cannot be equal to the one obtained in Subsection 4.1.3. 

The spectrum of the right quantum gauge Laplacian for $n\leq -1$ is presented in table 5 (see equation (\ref{spec2})), and as in the previous subsections, it is an infinite and discrete set.

\begin{center}
\begin{table}[b]
\centering
\begin{tabular}{|c|c|c|c|c|}
\hline 
\multicolumn{1}{|c|}{$T(1)$} & $n\in \N$ & \multicolumn{1}{|c|}{$\mathrm{Eigenvalue}$}\rule[-0.3cm]{0cm}{0.8cm}\\\hline
\multicolumn{1}{|c|}{$p(\gamma^k\gamma^{\ast\,k})$} & $k\in \N_0$ & \multicolumn{1}{|c|} {$-(q^{-2k}\,(1+q^2\,[l])\,[k]+q^4\,(1+q^{-2k}\,[k])\,[l])$} \rule[-0.4cm]{0cm}{1.1cm}\\\hline
\multicolumn{1}{|c|}{$p(\alpha^{t}\gamma^k\gamma^{\ast\,l})$} &  $t+k-l=0$ & \multicolumn{1}{|c|} {$-q^{-2t}\,([t]\,[l+1]+q^{-2k}\,[k]\,[l+1]+q^4\,[l]\,[t+1]+q^{-2k+4}\,[l]\,[k])$} \rule[-0.3cm]{0cm}{0.8cm}\\\hline
\multicolumn{1}{|c|}{$p(\alpha^{\ast t}\gamma^k\gamma^{\ast\,l})$} &$-t+k-l=0$   & \multicolumn{1}{|c|} {$-q^{-2k}\,([k]\,[t+1]+q^{2t+2}\,[k]\,[l]+q^4\,[t]\,[k+1]+q^{4+2t}\,[l]\,[k+1])$} \rule[-0.3cm]{0cm}{0.8cm}\\\hline
\end{tabular}
\caption{Spectrum of $\widehat{\square}^{\omega^c}_q$ for $n=-m$, $m\in \N$.}
\end{table}
\end{center}

\clearpage

\subsection{The Non--Commutativity of the Quantum Gauge Laplacians}
In this subsection, we show that the two quantum gauge Laplacians do not commute each other. This is highly important because most of the papers in the literature on the subject consider only one of the quantum gauge Laplacians, ignoring the other.

\begin{Theorem}
    \label{gauge0}
    Let $0\not=n$ $\in$ $\Z$. We have $$[\widehat{\square}^{\omega^c}_q,\square^{\omega^c}_q]\not=0.$$
\end{Theorem}

\begin{proof}
Let $n$ $\in$ $\N$ and consider $T$, $U$ $\in$ $\Mor(\delta^n,\Delta_P)$ defined by $$T(1)=\alpha^n\,\gamma\,\gamma^\ast, \qquad U(1)=\alpha^n.$$ Thus
\begin{eqnarray*}
\widehat{\square}^{\omega^c}_q(\square^{\omega^c}_q(T))=b_1\,\widehat{\square}^{\omega^c}_q(T)+b_2\,\widehat{\square}^{\omega^c}_q(U)&=&b_1\,(c_1\,T+c_2\,U)+b_2\,c_3\,U
    \\
    &=&
    b_1\,c_1\,T+b_1\,c_2\, U+b_2\,c_3\,U
    \\
    &=&
    b_1\,c_1\,T+(b_1\,c_2+b_2\,c_3)\,U
\end{eqnarray*}
and
\begin{eqnarray*}
\square^{\omega^c}_q(\widehat{\square}^{\omega^c}_q(T))=c_1\,\square^{\omega^c}_q(T)+c_2\,\square^{\omega^c}_q(U)&=&c_1\,(b_1\,T+b_2\,U)+c_2\,b_3\,U
    \\
    &=&
    b_1\,c_1\,T+b_2\,c_1\, U+b_3\,c_2\,U
    \\
    &=&
    b_1\,c_1\,T+(b_2\,c_1+b_3\,c_2)\,U,
\end{eqnarray*}
where, in accordance with the calculations of $\mathrm{Spec}(\square^{\omega^c}_q)$ and $\mathrm{Spec}(\widehat{\square}^{\omega^c}_q)$ in Section 4.1, 4.2, we have $$b_1=-([n+1]+q^{2n+2}+q^4\,(1+q^2)\,[n]+q^{4+2n}\,(1+q^2)),\quad b_2=-q^{-2+2n}\,(1+q^4),$$ $$b_3=-q^4\,[n], $$ $$c_1=-q^{-2n}\,((1+q^2)\,[n]+q^{-2}\,(1+q^2)+q^4\,[n+1]+q^{2}),\qquad c_2=-q^{-2}(1+q^4),$$ $$c_3=-q^{-2n}\,[n]. $$
Then $$[\widehat{\square}^{\omega^c}_q,\square^{\omega^c}_q](T)=(b_1\,c_2+b_2\,c_3-b_2\,c_1-b_3\,c_2)\,U\not=0.$$ Let $n=-m$ with $m$ $\in$ $\N$. Following the same strategy as above, we have $$[\widehat{\square}^{\omega^c}_q,\square^{\omega^c}_q](T)\not=0,$$ where $T$ $\in$ $\Mor(\delta^{-m},\Delta_P)$ is defined by $$T(1)=\alpha^{\ast n}\,\gamma\,\gamma^\ast.$$
\end{proof}

\begin{Corollary}
    \label{corochido}
    Let $0\not=n$ $\in$ $\Z$. Then, the operators $\square^{\omega^c}_q$, $\widehat{\square}^{\omega^c}_q$ are not simultaneously diagonalizable.
\end{Corollary}

\begin{Remark}
    \label{final}
    Although the two gauge Laplacians do not commute in general, they share some eigenvectors. For example,  the map $$T:\mathbb{C}\longrightarrow P \qquad \mbox{ such that }
\qquad T(1)=\alpha^n.$$
Therefore, there exist linear subspaces of $\Mor(\delta^n,\Delta_P)$  on which the two gauge Laplacians commute.
\end{Remark}

It is worth mentioning that the relation $\widehat{\square}^{\omega^c}_q=\ast \circ \square^{\omega^c}_q\circ \ast$ is because $\omega^c$ is regular (and hence $D=D^{\omega^c}=\widehat{D}^{\omega^c}$). By Proposition \ref{3.6}, we know that $\omega^c$ is the only regular qpc; so in general, the relation $\widehat{\square}^{\omega}_q=\ast \circ \square^{\omega}_q\circ \ast$ is not true for every qpc $\omega$. 

In the limit $q\longrightarrow 1$ (the {\it classical} case), we get $$\square^{\omega^c}_q=\widehat{\square}^{\omega^c}_q$$ and the set of all possible eigenvalues is $$\{0,-1,-2,-4\}.$$ This spectrum does not correspond to the one shown in the literature, for example in \cite{otto}. The reason is because in the {\it classical} case, the spectrum is taken only for $L^2$--elements \cite{otto} (for $L^2$--elements, the codifferential is actually the formal adjoint operator of the differential). However, for $q\longrightarrow 1$, elements of $\H^2_q$ are not $L^2$--functions. For example $\mathbbm{1}$, $\alpha\gamma^\ast$, $\alpha^\ast\gamma$ $\in$  $\H^2_q$ are not $L^2$--functions in the limit $q\longrightarrow 1$.

\begin{appendix}

\markright{Appendix A}

\section{The Universal Differential Envelope $\ast$--Calculus}
This appendix is a brief summary of the theory of the universal differential envelope $\ast$--calculus. For more details, see  references \cite{micho1,stheve}.

Let $(\Lambda,d)$ be a $\ast$--First Order Differential Calculus ($\ast$--FODC) over a $\ast$--Hopf algebra $(A,\cdot,\mathbbm{1},\ast,\Delta,\epsilon,S)$ and consider the graded vector space
$$\otimes^\bullet_A\Lambda:=\bigoplus_k (\otimes^k_A\Lambda)\quad \mbox{ with } \quad \otimes^0_A \Lambda=A,\quad  \otimes^k_A\Lambda:=\underbrace{\Lambda\otimes_A\cdots\otimes_A \Lambda}_{k\; times}$$ ($k\in \N$) endowed with its canonical graded $\ast$--algebra structure, which is given by
$$(\vartheta_1\otimes_{A}\cdots\otimes_{A}\vartheta_k)\cdot(\vartheta'_1\otimes_{A}\cdots\otimes_{A}\vartheta'_l):=\vartheta_1\otimes_{A}\cdots\otimes_{A}\vartheta_k\otimes_{A}\vartheta'_1\otimes_{A}\cdots\otimes_{A}\vartheta'_l,$$ $$(\vartheta_1\otimes_{A}\cdots\otimes_{A}\vartheta_k)^{\ast}:=(-1)^{k(k-1)\over 2}\,\vartheta^{\ast}_k\otimes_{A}\cdots\otimes_{A}\vartheta^{\ast}_1, $$ for $\vartheta_1\otimes_{A}\cdots\otimes_{A}\vartheta_k$ $\in$ $\otimes^{k}_A\Lambda$ and $\vartheta'_1\otimes_{A}\cdots\otimes_{A}\vartheta'_l$ $\in$ $\otimes^{l}_A\Lambda$. Now, let us consider the quotient graded space
  \begin{equation}
      \label{udtensor}     \Lambda^\wedge:=\otimes^\bullet_A\Lambda/\mathcal{Q},
  \end{equation}
  where $\mathcal{Q}$ is the two--side ideal of $\otimes^\bullet_A\Lambda$ generated by elements
  \begin{equation}
      \label{udtensor1}
     \sum_i dg_i\otimes_A dh_i \quad \mbox{ such that } \quad \sum_i g_i\,dh_i=0,
  \end{equation}
  for all $g_i$, $h_i$ $\in$ $A$. According to \cite{micho1,stheve}, the graded $\ast$--algebra structure of $\otimes^\bullet_A \Lambda$ endows $\Lambda^\wedge$ with structure of graded $\ast$--algebra. The product in $\Lambda^\wedge$ will be denoted simply by juxtaposition of elements. On the other hand, for a given $t=\vartheta_1\cdots \vartheta_n$ $\in$ $\Lambda^{\wedge\,n }$ with $\vartheta_1$,..., $\vartheta_n$ $\in$ $\Lambda$, the linear map
  \begin{equation}
      \label{diff12}
     d:\Lambda^\wedge\longrightarrow \Lambda^\wedge 
  \end{equation}
  given by $$d(t)=d(\vartheta_1\cdots \vartheta_n)=\displaystyle \sum^n_{j=1}(-1)^{j-1}\vartheta_1\cdots \vartheta_{j-1}\cdot d\vartheta_j\cdot \vartheta_{j+1}\cdots \vartheta_n \; \in \; \Lambda^{\wedge\,n+1 },$$ where $d\vartheta_j=\displaystyle \sum_l dg_l\,dh_l$ if $\vartheta_j=\displaystyle \sum_l g_l\,(dh_l)$ is well--defined, satisfies the graded Leibniz rule, $d^2=0$ and $d(t^\ast)=(dt)^\ast$ \cite{micho1,stheve}. In this way 
  \begin{equation}
    \label{2.f10.1}
     (\Lambda^\wedge,d,\ast)
\end{equation}
  is a graded differential $\ast$--algebra generated by its degree 0 elements $\Lambda^{\wedge\,0}= A$  and it is called {\it the universal differential envelope $\ast$--calculus}.

 In references \cite{micho1,stheve} the reader can find a proof of the following statement.

\begin{Proposition}
\label{A.1}
Suppose $(\Omega^\bullet=\bigoplus_k \Omega^k,d,\ast)$ is a graded differential $\ast$--algebra and $(\Lambda,d)$ is a $\ast$--FODC over $A$.

Let $$\phi^0: \Lambda^{\wedge 0}\longrightarrow \Omega^0$$ be a $\ast$--algebra morphism and $$\phi^1:\Lambda^{\wedge 1}\longrightarrow \Omega^1$$ be a linear map such that $$\phi^1(a \,db)=\phi^0(a)\,d(\phi^0(b))$$ for all $a$, $b$ $\in$ $A$. Then, there exist unique linear maps $$\phi^k: \Lambda^{\wedge k}\longrightarrow \Omega^k$$ for all $k\geq 2$ such that $$\phi:=\bigoplus_k\phi^k: \Lambda^\wedge\longrightarrow \Omega^\bullet$$ is a graded differential $\ast$--algebra morphism.
\end{Proposition}

For left covariant  $\ast$--FODC's there is another construction of the universal differential envelope $\ast$--calculus which is very useful. In fact, let $(\Lambda,d)$ be a left covariant $\ast$--FODC. Then, it is well--known that  $$(\Lambda,d)\cong (A\otimes {\Ker(\epsilon)\over R},d ) $$ for some right $A$--ideal $R$ $\subseteq$ $\Ker(\epsilon)$ that satisfies   $S(R)^\ast\subseteq R$ \cite{stheve,woro2}. Let us define  $$\mathfrak{qg}^\#:= {\Ker(\epsilon)\over R}.$$ In addition, let us take
\begin{equation}
    \label{2.f12.1}
    \begin{aligned}
        \mathfrak{qg}^{\#\wedge}=\otimes^\bullet \mathfrak{qg}^\#/S^\wedge, \qquad \otimes^\bullet\mathfrak{qg}^\#:=\bigoplus_k (\otimes^k\mathfrak{qg}^\#)\\
    \otimes^0\mathfrak{qg}^\#=\C\mathbbm{1},\qquad \otimes^k\mathfrak{qg}^\#:=\underbrace{\mathfrak{qg}^\#\otimes\cdots\otimes \mathfrak{qg}^\#}_{k\; times}
    \end{aligned}
\end{equation}
($k\in \N$), where $S^\wedge$ is the graded two--side ideal of $\otimes^\bullet\mathfrak{qg}^\#$ generated by elements 
\begin{equation}
    \label{2.f12.11}
    \pi(g^{(1)})\otimes \pi(g^{(2)})\qquad \mbox{ for all }\qquad g \,\in\, R.
\end{equation}
Then, in light of \cite{micho1}, we have
\begin{equation}
    \label{a.1}
    (\Lambda^\wedge,d,\ast)\cong (A\otimes \mathfrak{qg}^{\#\wedge},d,\ast).
\end{equation}
For degree 0, the previous isomorphism is given by $A\cong A\otimes \mathbbm{1}$ in the canonical way and for degree $1$, the previous isomorphism matches with the well--known isomorphism $(\Lambda,d)\cong (A\otimes \mathfrak{qg}^\#,d)$ \cite{micho1}. Furthermore, there is a Maurer--Cartan formula (\cite{micho1,stheve})
\begin{equation}
    \label{a.2}
    d\pi(a)=-\pi(a^{(1)})\pi(a^{(2)})
\end{equation}
for all $a$ $\in$ $A$, where $$\pi:A\longrightarrow \mathfrak{qg}^\#,\qquad a\longmapsto \pi(a):=S(a^{(1)})da^{(2)} $$ is the corresponding quantum germs map \cite{stheve}.

If the $\ast$--FODC $(\Lambda,d)$ is bicovariant, then, the coproduct $\Delta$ of $A$ can be extended to a graded differential $\ast$--algebra morphism (\cite{micho1,stheve})
\begin{equation}
    \Delta: \Lambda^\wedge\longrightarrow \Lambda^\wedge\otimes \Lambda^\wedge.
\end{equation}
In the last tensor product, we have taken the tensor product of graded differential $\ast$--algebras. This map satisfies
\begin{equation}
\label{2.f3.5}
\Delta(\theta)=\ad(\theta)+\mathbbm{1}\otimes \theta 
\end{equation}
for every $\theta$ $\in$ $\mathfrak{qg}^\#$, with $$\ad: \mathfrak{qg}^\#\longrightarrow \mathfrak{qg}^\#\otimes A $$ the $A$--corepresentation that satisfies $$\ad\circ \pi=(\pi\otimes \id_A)\circ \Ad,$$ where $$\Ad:A\longrightarrow A\otimes A, \qquad a\longmapsto a^{(2)}\otimes S(a^{(1)})a^{(3)}$$ 
is the right adjoint coaction of $A$ \cite{micho1,stheve}. In addition, for bicovariant $\ast$--FODC's, the counit $\epsilon$ and the antipode $S$ can also be extended to $\Gamma^\wedge$
\begin{equation}
\label{2.f3.6}
\epsilon: \Lambda^\wedge \longrightarrow \C,
\end{equation}
\begin{equation}
\label{2.f3.7}
S: \Lambda^\wedge \longrightarrow \Lambda^\wedge.
\end{equation}
The extension of the coint is given by
$$\epsilon|_{\Lambda^{\wedge\,k}}=0 \quad \mbox{ for all } \quad k\geq 1 $$ and the extension of the antipode is as follows: for every $\theta$ $\in$ $\mathfrak{qg}^\#$, $\theta=\pi(a)$ for some $a$ $\in$ $A$,  the formula
$$S(\theta)=S(\pi(a))=-\pi(a^{(2)})S(a^{(3)})S(S(a^{(1)}))$$ is well--defined \cite{micho1}. With this, it is possible to define $S:\Lambda\longrightarrow \Lambda$ such that
$$S(b\,\pi(a))=S(\pi(a))S(b),\qquad S(b\, da)=d(S(a))\,S(b)$$ and since $\Lambda^\wedge$ is generated by its degree 0, we can extend $S$ to whole space $\Lambda^\wedge$ \cite{micho1}.

In light of \cite{micho1},  $(\Lambda^\wedge,\cdot,\mathbbm{1},\ast,d,\Delta,\epsilon,S)$ is a graded differential $\ast$--Hopf algebra. With this structure, the right $A$--corepresentation $\Ad$ can also be extended, by means of
\begin{equation}
\label{2.f11}
\Ad:\Lambda^\wedge \longrightarrow \Lambda^\wedge \otimes \Lambda^\wedge
\end{equation}
(in the last tensor product, we have taken the tensor product of graded differential $\ast$--algebras) such that $$\Ad(t)=(-1)^{\partial t^{(1)}\partial t^{(2)}} t^{(2)}\otimes S(t^{(1)})t^{(3)}$$ for all $t$ $\in$ $\Gamma^\wedge$, where $\partial x$ denotes the grade of $x$ and $(\id_{\Gamma^\wedge}\otimes \Delta)\Delta(t)=(\Delta\otimes \id_{\Gamma^\wedge})\Delta(t)=t^{(1)}\otimes t^{(2)}\otimes t^{(3)}.$ 

\end{appendix}

The manuscript has no associated data.

\end{document}